\documentclass[10pt,a4paper]{amsart}
\usepackage[margin=1.1in]{geometry}
\usepackage[english]{babel}
\usepackage{amsmath,amssymb,amsthm,mathtools, amsfonts}
\usepackage{lmodern}
\usepackage[T1]{fontenc}
\usepackage{microtype}
\usepackage{enumitem}
\usepackage{tikz}
\usetikzlibrary{arrows, arrows.meta,calc,decorations.markings,patterns}
\usepackage{xcolor}
\usepackage{pgf}
\usepackage[colorlinks=true,linkcolor=blue!60!black,citecolor=blue!60!black,urlcolor=blue!60!black]{hyperref}

\newcommand{\R}{\mathbb{R}}
\newcommand{\Z}{\mathbb{Z}}
\newcommand{\N}{\mathbb{N}}

\newcommand{\A}{\mathbb{A}}
\newcommand{\Tp}{\mathbb{T}_{\pi}}
\newcommand{\Tt}{\mathbb{T}_{2\pi}}
\newcommand{\cH}{\mathcal{H}}
\newcommand{\cR}{\mathcal{R}}
\newcommand{\cS}{\mathcal{S}}
\newcommand{\cU}{\mathcal{U}}
\newcommand{\cC}{\mathcal{C}}
\newcommand{\cG}{\mathcal{G}}
\newcommand{\cV}{\mathcal{V}}

\newcommand{\cO}{\mathcal{O}}
\newcommand{\cP}{\mathcal{P}}
\newcommand{\dd}{\,\mathrm{d}}

\newcommand{\supp}{\operatorname{supp}}
\newcommand{\Fr}{\partial}
\newcommand{\Var}{\operatorname{Var}}
\newcommand{\Leb}{\mathcal{L}}
\newcommand{\dist}{\operatorname{dist}}
\newcommand{\id}{\mathrm{id}}
\newcommand{\Homeo}{\operatorname{Homeo}}
\newcommand{\Int}{\operatorname{int}}
\newcommand{\Hull}{\operatorname{conv}}
\newcommand{\bn}{\mathbf{n}}
\newcommand{\bt}{\mathbf{t}}

\DeclareMathOperator{\sgn}{sgn}

\theoremstyle{plain}
\newtheorem{theorem}{Theorem}[section]
\newtheorem*{mainthm}{Main Theorem}
\newtheorem*{mainrigthm}{Rigidity Theorem}
\newtheorem{proposition}[theorem]{Proposition}
\newtheorem{lemma}[theorem]{Lemma}
\newtheorem{corollary}[theorem]{Corollary}
\theoremstyle{definition}
\newtheorem{definition}[theorem]{Definition}
\newtheorem{remark}[theorem]{Remark}

\numberwithin{equation}{section}

\tikzset{
  heur support/.style={
    line width=.5pt,black!60
  },
  heur alternative/.style={
    line width=.5pt,black!60,dashed
  },
  heur vector/.style={
    -{Stealth[length=2mm,width=1.4mm]},
    line width=.8pt
  },
  heur small arrow/.style={
    -{Stealth[length=1.5mm,width=1.1mm]},
    line width=.55pt
  },
  heur curve/.style={
    line width=.8pt
  },
  heur label/.style={
    font=\small,inner sep=2pt
  },
heur panel/.style={
  font=\small\mdseries,
  anchor=west
},
  heur note/.style={
    font=\small,align=center,inner sep=2pt
  }
}

\title[Unbounded orbits in $C^1$-generic centrally symmetric outer billiards]{Unbounded orbits in $C^1$-generic centrally symmetric strictly convex outer billiards}
\author{Alfonso Sorrentino}
\address{Department of Mathematics, University of Rome Tor Vergata, Rome (Italy) \& School of Mathematics, Institute for Advanced Study, Princeton, NJ (USA)}
\email{sorrentino@mat.uniroma2.it, sorrentino@ias.edu}

\thanks{{\it AI Disclosure:} A near-final draft of this paper was proofread with the assistance of Claude (Anthropic), which also helped generating the TikZ code for all figures. The author takes full responsibility for the content of the paper.}

\begin{document}
\begin{abstract}
Let $\cH$ be the space of support functions of centrally symmetric, strictly convex, compact planar bodies whose boundary is a $C^1$ curve, endowed with the $C^1$ topology. We prove that for a residual subset of $\cH$ the outer billiard maps about the associated bodies admit unbounded orbits escaping to infinity. In particular, this residual set includes support functions arbitrarily $C^1$-close to those of disks or ellipses. This provides an affirmative answer, in the strictly convex
$C^1$ setting, to a famous  question of Moser and Neumann
concerning the existence of unbounded outer-billiard orbits. The key ingredient is a rigidity theorem: if $h\in\cH$ has a purely singular radius-of-curvature measure, then every continuous
invariant tangent graph consists entirely of periodic points.  We then show that for a generic table in $\cH$  these periodic invariant graphs are absent, allowing us to construct escaping orbits via Mather's diffusing mechanism in a Birkhoff region of instability.
\end{abstract}

\maketitle

\section{Introduction}\label{sec:intro}

The \emph{outer} (or \emph{dual}) \emph{billiard map}   about a compact strictly convex body  $K$  sends a point $x \notin K$ to its reflection through 
the point of tangency $p$ of the supporting line to $K$ through $x$, the tangent line being chosen by a fixed orientation convention (see Figure~\ref{figure:dual_billiard}). This map is an area-preserving homeomorphism from the exterior of $K$ to itself. 
With suitable adaptations, this dynamical model can be extended to polygons  ({\it i.e.}, omitting points that lie on a line through a side), where the dynamics consists of a repeated reflection about the vertices.  We refer to \cite{Tab95, Tab05} for a more comprehensive introduction; see also Section \ref{subsec:T}.

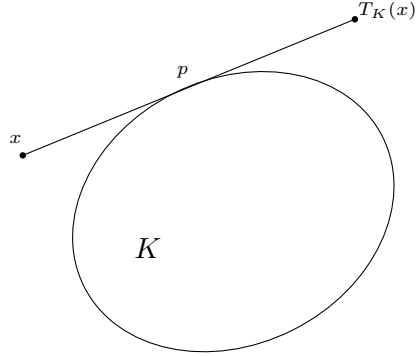
\begin{figure}[!h]
	\definecolor{uququq}{rgb}{0.25,0.25,0.25}
	\definecolor{qqqqff}{rgb}{0,0,1}
	\definecolor{xdxdff}{rgb}{0.49,0.49,1}
	\definecolor{cqcqcq}{rgb}{0.75,0.75,0.75}
	\begin{tikzpicture}[line cap=round,line join=round,>=triangle 45,x=3.0cm,y=3.0cm, scale=0.8]
		\clip(-1.4,-0.6) rectangle (1,1.4);
		\draw [rotate around={26.57:(0,0.25)}] (0,0.25) ellipse (2.76cm and 2.19cm);
		\draw (-1.16,0.56)-- (0.67,1.31);
		\begin{scriptsize}
			\draw[color=black] (-0.47,0.05) node {{\large $K$}};
			\draw[color=black] (-0.28,1.02) node {$p$};
			\fill [color=black] (-1.16,0.56) circle (1.5pt);
			\draw[color=black] (-1.2,0.65) node {$x$};
			\fill [color=black] (0.67,1.31) circle (1.5pt);
			\draw[color=black] (0.85,1.36) node {$T_K(x)$};
		\end{scriptsize}
	\end{tikzpicture}
	\caption{\small Outer billiard map about $K$}
	\label{figure:dual_billiard}
\end{figure}

The introduction of outer billiards is usually attributed to B.~H.~Neumann in 1959 \cite{Neu59}, although an earlier related construction can be traced in a work by M.~M.~Day in 1947 \cite{Day47}; we refer to \cite{Has26} for an interesting account of the earliest  works on this subject. J.~Moser popularized this dynamical system in the 1970s as a toy model for celestial mechanics: the orbit of a point around the billiard table can be seen as an idealization of the orbit of a celestial body around a planet or a star (see \cite{Mos73,Mos78}).

Since then, outer billiards have received a lot of attention and have been extensively studied; we refer to \cite{Sch22, Tab93, Tab95, Tab95b, Tab05, TD05} for  more detailed expositions.\\

Just as for classical inner billiards, where the dynamics takes place inside a domain with elastic reflections at the boundary (see for example \cite{Sch22,Tab95, Tab05}), many questions --- such as integrability, the existence of invariant curves ({\it e.g., caustics}), spectral rigidity properties, the measure of the set of periodic points, etc. --- can be posed analogously for the outer billiard map. However, unlike the classical inner case, the phase space of the outer billiard map is unbounded, being the exterior of a compact set, which naturally raises the question of whether unbounded orbits exist.\\

This problem first appeared stated in the work of P.~C.~Hammer~\cite{Ham63}. Moser brought it to a broader audience in his Hermann Weyl lectures~\cite[Section  II.4]{Mos73} and in~\cite{Mos78}, where he phrased it as the question: ``\emph{Decide whether every orbit is bounded}'', attributing it to Neumann. In fact, in~\cite{Neu59}, Neumann asked whether the map admits  invariant curves: these two formulations are closely related, since the existence of arbitrarily large invariant curves would confine all orbits and hence would imply boundedness.

\medskip

\noindent {\bf Question (Moser, Neumann):} \emph{Is there a convex table for which the outer billiard has an unbounded orbit?} 

\medskip
See \cite[Section 6.1]{Has26} for a detailed account of the state of the art on this question.

\medskip
For sufficiently smooth tables all orbits are bounded: if $\partial K$ is of class $C^6$ with positive curvature, Moser's invariant curve theorem, applied to the return map near infinity, produces invariant curves confining every orbit \cite{Mos73} (the complete proof is due to R.~Douady~\cite{Dou82}). E.~Gutkin--A.~Katok~\cite{GK95} carried out a systematic study of caustics of outer billiards in this regime, and P.~Boyland~\cite{Boy96} exploited the twist-map structure of the return map near infinity. \\

At the opposite extreme, for polygonal tables the dynamics is a piecewise isometry: orbits are bounded for lattice and quasi-rational polygons (F.~Vivaldi--A.~V.~Shaidenko~\cite{VS87}, R.~Ko\l odziej~\cite{Kol89}, E.~Gutkin--N.~Sim\'anyi~\cite{GS92}) and for trapezoids (L.~Li~\cite{Li09}), whereas R.~E.~Schwartz~\cite{Sch07,Sch09} showed that the outer billiard about the Penrose kite has unbounded orbits. D.~Dolgopyat--B.~Fayad~\cite{DF09} did  the same for the semicircle (their result can also be extended to domains obtained by nearly cutting a disk in half).\\

Between these two regimes lies the case of strictly convex tables whose boundary is of class $C^1$, but has no further regularity. 
In \cite[Question 6.9]{Sch22},  Schwartz asks:
{\it Is there a strictly convex domain or a $C^1$ domain with respect to which outer billiard has unbounded orbits?}

\medskip
The present paper provides an affirmative answer to Schwartz's question (actually in a stronger form: tables will be {\it both} strictly convex and $C^1$).\\

\noindent {\bf Main Theorem (informal statement).}
{\it The outer billiard map about a typical centrally symmetric, strictly convex compact body whose boundary is a regular $C^1$ curve admits unbounded orbits  escaping to infinity.}\\

In order to state the Main Theorem more precisely, let us introduce some notation and recall the needed definitions.

\subsection{Support function of a convex body} Given a compact convex body \(K\subset\mathbb R^2\) and choosing an origin in the interior of $K$, its {\it support function} $h_K$ is given by:
\[
h_K(\theta):=\max_{x \in K}\bigl(x \cdot (\cos\theta, \sin\theta)\bigr),
\qquad \theta\in \R/2\pi\Z.
\]
For each \(\theta\), \(h_K(\theta)\) is the signed distance from the origin to the supporting line of \(K\) with outward normal
\(
(\cos\theta,\sin\theta).
\)
The body \(K\) is recovered from its support function by
\[
K=\{x\in\mathbb R^2:\, x \cdot (\cos\theta, \sin\theta)\le h_K(\theta)
\ \text{for all }\theta\},
\]
If $K$ is strictly convex (so that $h_K$ is $C^1$), its boundary admits the normal-angle parametrization
\begin{eqnarray}
p: \R/2\pi\Z &\longrightarrow& \partial K \nonumber\\
\theta &\longmapsto& p(\theta):=h_K(\theta)\bn(\theta)+h_K'(\theta)\bt(\theta), \label{eq:paramptheta}
\end{eqnarray}
where $\bn(\theta):=(\cos\theta,\sin\theta)$ and $\bt(\theta):=(-\sin\theta,\cos\theta)$.\\

\begin{remark} \label{rem:suppfct}
{\bf (i)} There is a  duality between the strict convexity of \(K\) and the differentiability of its support function $h_K$, as discussed, for instance, in \cite[Section 1.7]{Sch14} and \cite[Section 23]{Rock70}.

\smallskip
\noindent {\bf (ii)} The support function depends on the choice of the origin.
It is easy to check that  placing the new origin at the point \(a=(a_1,a_2)\), with respect to the previous origin,
only adds a term $-a_1 \cos \theta - a_2 \sin \theta$ to the support function. Strict convexity, smoothness, and differentiability properties of the support function are unchanged by a change of origin.

In particular, if $K$ is centrally symmetric and we fix the origin to be the center of symmetry, we obtain that $h_K$ is $\pi$-periodic. 

\smallskip
\noindent {\bf (iii)} Let $h_K$ be $C^2$.
Using \eqref{eq:paramptheta}, one obtains that
\[
p'(\theta)=\bigl(h_K(\theta)+h_K''(\theta)\bigr)\bt(\theta).
\]
Thus $h_K+h_K''$ is the speed of the normal-angle
parametrization. Wherever this quantity is positive,
it equals the classical radius of curvature at $p(\theta)$ (see also \cite[Section 2.5]{Sch14}).
\end{remark}
\bigskip

\subsection{Set of centrally symmetric, strictly convex $C^1$ domains}
We denote by  $\cH$ the set of support functions of centrally symmetric, strictly convex compact bodies whose boundary is a regular $C^1$ curve. 
The set $\cH$ can be described more explicitly.\\

Hereafter,  for centrally symmetric bodies we will always assume that the origin is placed at the center of symmetry of the body (see Remark \ref{rem:suppfct} (ii)).\\

\begin{definition}
 In the light of Remark \ref{rem:suppfct} (iii),  for a $C^1$ function $h: \R/\pi\Z \to (0,+\infty)$ we define
\begin{equation}\label{eq:rho}
\rho_h \;:=\; h\dd\theta + Dh'
\end{equation}
where $D$ denotes the distributional derivative. When $h$ is the support function of a convex body, $\rho_h$ is a positive measure, which we call the \emph{radius of curvature measure} of the body. 
\end{definition}

\smallskip

We will prove  (see Section \ref{sec:cHproof}) that
\begin{equation}\label{eq:H}
\cH \;=\; \bigl\{\, h\in C^1(\R/\pi\Z)\;:\;  h>0,\quad \rho_h\ge 0,\quad \supp\rho_h=\R/\pi\Z \,\bigr\}.
\end{equation}
We endow $\cH$ with the  $C^1$ topology, {\it i.e.}, $\|h\|_{C^1}:=\|h\|_{\infty} + \|h'\|_{\infty}$ for $h\in \cH$. \\

\begin{remark}\label{rem:Hintro}
The three conditions have a precise geometric meaning (see also Section \ref{sec:cHproof} and Remark~\ref{rem:H}). 
\begin{itemize}
\item $\rho_h\ge0$ implies that $h$ is the support function of a convex body;
\item $\pi$-periodicity implies that $K_h=-K_h$;
\item $h\in C^1$ implies that $K_h$ is strictly convex, {\it i.e.}\ $\partial K_h$ contains no segment, and is equivalent to $\rho_h$ having no atoms;
\item $\supp\rho_h=\R/\pi\Z$ implies that $\partial K_h$ has no corners, {\it i.e.}, it is a regular $C^1$ curve. We stress that the last condition is not a consequence of $h\in C^1$: the support function of a strictly convex body with corners is $C^1$. 
\end{itemize}
\end{remark}

\medskip

\subsection{Main Theorem}\label{sec:mainthm}
We can now state the main theorem.

\begin{mainthm}\label{thm:main}
There exists a residual ({\it i.e.}, $G_\delta$ dense) subset $\cR\subset\cH$ with the following properties. For every $h\in\cR$:
\begin{enumerate}[label=\textup{(\roman*)}]
\item the radius of curvature measure $\rho_h$ is purely singular with respect to Lebesgue measure, {\it i.e.}, it is carried by a Lebesgue-null set;
\item the outer billiard map $T$ about the body $K_h$ with support function $h$ has no continuous invariant curve in $\R^2\setminus K_h$ which is a graph over the contact angle \textup{(}Definition~\textup{\ref{def:graph})};
\item there exist points $x\in\R^2\setminus K_h$ with $|T^n x|\to\infty$ as $n\to+\infty$. 
\end{enumerate}
In particular, there exist $h\in\cR$ arbitrarily $C^1$-close to the support functions of disks or ellipses.
\end{mainthm}

\medskip

\begin{remark}\label{rem:aftermainthm} 
{\bf (i)} The role of purely singular radius of curvature measures can be intuitively explained by saying that in the class of 
$C^1$ and strictly convex domains we are interested in those that behave (as far as their outer billiard dynamics is concerned) the most similarly to polygons. This polygon-like dynamical behaviour -- with its combination of {\it edge-like} and {\it vertex-like} points -- is, in fact, the source of a rigidity phenomenon on which the whole proof rests (see Remark \ref{rem:commentDT} and Section \ref{sec:heuristic}  for a more detailed explanation).

 Recall that for a convex polygon $P$  with edges $e_1,\dots,e_n$ (listed in counterclockwise order), if $\theta_k\in\R/2\pi\Z$ denotes the direction of the outward normal $\bn(\theta_k)$ of the edge $e_k$, then
\[
\rho_{h_P}=\sum_{k=1}^n\ell_k\,\delta_{\theta_k},
\]
where $h_P$ denotes the support function of $P$ and $\ell_k>0$
 the length of $e_k$.
In particular, $\rho_{h_P}$ is purely atomic: the edges correspond to the atoms, while the corners are the gaps in the support.
However, polygons are neither $C^1$ nor strictly convex: in order to stay  ``as close as possible'' to polygons, yet to have  $C^1$ and strictly convex domains, we need to consider purely singular $\rho_h$ with full support. This will be discussed more  thoroughly in Section \ref{sec:cHproof} (see also Remark \ref{rem:H}).
\\
{\bf (ii)} The Main Theorem does not conflict with the fact that for $C^6$ tables all orbits are bounded: smooth tables with positive curvature form a meager subset of $\cH$, and for them $\rho_h$ is absolutely continuous, so they lie outside $\cR$.\\
{\bf (iii)} Since the outer billiard maps about disks or ellipses are integrable (see, for instance, \cite{Bia24, Tab95}), Main Theorem (iii) can be interpreted as a manifestation of {\it Arnold diffusion} for $C^1$-generic perturbations of those integrable maps.\\
\end{remark}

\medskip
\subsection{Strategy of the proof} \label{sec:strategy}

Let us outline the main steps involved in the proof of the Main Theorem.

\medskip

\noindent {\bf $\bullet$ Step 1 (Rigidity of invariant graphs):} The whole proof hinges on the following rigidity result for tables with purely singular radius of curvature measure, which essentially relies on the polygonal-like dynamical behaviour  of these domains.

\medskip
\begin{mainrigthm}\label{thm:rigidity-intro}
Let $h\in\cH$ and suppose that $\rho_h$ is purely singular. Let $K_h$ be the centrally symmetric, strictly convex body with $C^1$ boundary defined by $h$ and let $T_{K_h}$ denote the outer billiard map about $K_h$. Then, every continuous $T_{ K_h}$-invariant curve in $\R^2\setminus K_h$ which is a graph over the contact angle consists of periodic points with the same period. 
\end{mainrigthm}

See Theorem~\ref{thm:rigidity} for a more detailed statement.
In Section~\ref{sec:heuristic}, we will provide a heuristic description of the mechanism that forces invariant graphs to be periodic. This is a result of the rigid and complementary behaviours of the map at edge-like and vertex-like points (see also Remark \ref{rem:commentDT}).

\medskip

\noindent {\bf $\bullet$ Step 2 (Destruction of invariant graphs):} 
For tables with purely singular radius of curvature measure,
the Rigidity Theorem reduces the elimination of invariant
graphs to the elimination of invariant graphs consisting
entirely of periodic points, which is far more tractable (and quite classical).
In fact,  periodic configurations are the critical points of an explicit action functional (Section~\ref{sec:action}), and one can show that for each period $q$ and each compact region of the phase space, 
there exists an open dense subset in $\cH$, for which the corresponding outer billiard maps do not admit periodic graphs of period $q$ in that compact region (Theorem~\ref{thm:dense}). Combining these open dense sets with the residual set
of tables having purely singular radius-of-curvature
measure yields the residual set $\cR$
(see Section~\ref{sec:baire}).

\medskip
\noindent {\bf $\bullet$ Step 3 (Diffusing orbits in a Birkhoff region of instability):}
The final step consists in showing the existence of unbounded orbits for billiard tables corresponding to elements of $\cR$. The passage from ``{\it no invariant curves}'' to ``{\it existence of orbits escaping to infinity}''  exploits Mather's mechanism for diffusing in a Birkhoff region of instability~\cite{Mat91}. Since in our setting the outer billiard map is only a homeomorphism, we use P.~Le Calvez's  topological formulation~\cite{LC23}. This  will complete the proof of the  Main Theorem. \\

\begin{remark}\label{rem:step2}
In Appendix \ref{app:Radon}  we construct examples of domains that are $C^1$, centrally symmetric, strictly convex, with purely singular radius of curvature measure, yet they admit an invariant curve consisting of periodic points of period $4$ (these tables  can be constructed arbitrarily $C^1$-close to  disks or ellipses). We remark that this  does not prevent a priori the existence of unbounded orbits for these tables. 
\end{remark}

\medskip


\subsection {Organization of the article} 
The rest of the article is organized as follows.
\begin{itemize}
\item In {\bf Section \ref{sec:prelim}} we introduce the setting and discuss some preliminary material. In particular, in {\bf Section \ref{sec:cHproof}} we prove that $\cH$ consists exactly  of support functions of centrally symmetric, strictly convex compact bodies whose boundary is a regular $C^1$ curve. In {\bf Section \ref{subsec:T}} we recall the main properties of the outer billiard map and we introduce the {\it sign-changed} map $F=-T$, that is better suited for studying the dynamics for centrally symmetric domains.

\smallskip
\item In {\bf Section \ref{sec:rigidity}} we prove the Rigidity Theorem, which represents the core of our approach (see Theorem \ref{thm:rigidity}). We will first consider the case of invariant graphs with a central symmetry (see {\bf Section \ref{sec:step1}}), providing a heuristic explanation of the main ideas involved ({\bf Section \ref{sec:heuristic}}). The general case will be treated in {\bf Section \ref{sec:step2}}, via a {\it symmetrization} argument.

\smallskip
\item In {\bf Section \ref{sec:variational}} we prove the $C^1$-generic absence of invariant graphs consisting of periodic points, which -- because of the Rigidity Theorem -- suffices to show the $C^1$-generic absence of invariant graphs. In particular, in {\bf Section \ref{sec:action}} we recall the variational principle that underlies the dynamics. The residual set $\cR$ of the Main Theorem will be constructed in {\bf Section \ref{sec:baire}}, yielding the proof of parts (i) and (ii) of the Main Theorem (see Theorem \ref{thm:no-graphs}).

\smallskip
\item In {\bf Section \ref{sec:topology}} we describe Le Calvez's  topological formulation of Mather's mechanism for constructing diffusing orbits in a Birkhoff region of instability. This will allow us to prove the existence of unbounded orbits for outer billiard maps corresponding to $h\in \cR$, thus completing the proof of the Main Theorem (iii) (see {\bf Section \ref{subsec:main-proof}}).

\smallskip 
\item In {\bf Appendix \ref{app:proofs}} we collect the proofs of the properties of the outer billiard map stated in Section \ref{subsec:T}.

\smallskip 
\item In {\bf Appendix \ref{app:Radon}} we construct the  tables mentioned in Remark \ref{rem:step2}.

\medskip

\end{itemize}

\subsection*{Acknowledgements}
The author is grateful to Sergei Tabachnikov and Ke Zhang for many valuable discussions on outer billiards, which motivated this project and were pivotal for its development.
This article was completed while the author was a Member of the Institute for Advanced Study in Princeton; he thanks the Institute and the anonymous donor, whose support made his membership in the 2026--27 academic year possible, for the opportunity to work in such a stimulating environment.
The author also acknowledges the support of the Department of Excellence grant MatMod@TOV (2023--27), awarded to the Department of Mathematics of the University of Rome Tor Vergata. He is a member of the INdAM research group GNAMPA and of the UMI group DinAmicI.

\bigskip

\section{Setting and some background material}\label{sec:prelim}

\medskip\noindent\textbf{Notation.} Throughout this article, we adopt the following notation:
\begin{itemize}
\item $\Tp:=\R/\pi\Z$ and $\Tt:=\R/2\pi\Z$. 
\item For centrally symmetric bodies, we place the origin  at the center of symmetry of the body, so their support function is $\pi$-periodic.
\item We write $\bn(\theta):=(\cos\theta,\sin\theta)$ and $\bt(\theta):=(-\sin\theta,\cos\theta)$, so that $\bt=\bn'$ and $\bn=-\bt'$. 
\item Euclidean inner product and norm are $x\cdot y$ and $|x|$.
\item  $|B|$ or $\Leb(B)$ is the Lebesgue measure of a set $B$. 
\item For a continuous function $u$ of bounded variation, $Du$ denotes its distributional derivative, which is an atomless finite signed measure.
\end{itemize}

\medskip

\subsection{The set $\cH$ and how to reconstruct a body from its radius of curvature measure}\label{sec:cHproof}

We introduced in \eqref{eq:H} the set
$$ \cH \;=\; \bigl\{\, h\in C^1(\R/\pi\Z)\;:\; h>0,\quad \rho_h\ge 0,\quad \supp\rho_h=\R/\pi\Z \,\bigr\},$$
endowed with the  $C^1$ topology, {\it i.e.}, $\|h\|_{C^1}:=\|h\|_{\infty} + \|h'\|_{\infty}$ for $h\in \cH$. \\

In the following proposition we show that  $\cH$ consists exactly  of support functions of centrally symmetric, strictly convex compact bodies whose boundary is a regular $C^1$ curve.\\

\begin{proposition}\label{prop:reconstruction}
Let $\rho$ be a finite positive atomless Borel measure on $\Tp$ with $\supp\rho=\Tp$, and set
\begin{equation}\label{eq:h-from-rho}
h_\rho(\theta)\;:=\;\frac12\int_{\Tp}|\sin(\theta-u)|\dd\rho(u),\qquad \theta\in\Tp .
\end{equation}
Then:
\begin{enumerate}[label=\textup{(\roman*)}]
\item $h_\rho\in C^1(\Tp)$, $h_\rho>0$, and $h_\rho\dd\theta+Dh_\rho'=\rho$; in particular $h_\rho\in\cH$.
\item The curve $$p(\theta):=h_\rho(\theta)\bn(\theta)+h_\rho'(\theta)\bt(\theta) \quad \theta\in\Tt$$ is a Jordan curve, which bounds a compact strictly convex body $K_{h_\rho}=-K_{h_\rho}$ whose support function is $h_\rho$, and $p(\theta)$ is the unique point of $K_{h_\rho}$ at which $x\mapsto x\cdot\bn(\theta)$ attains its maximum $h_\rho(\theta)$. The distributional derivative of $p$ is
\begin{equation}\label{eq:Dp}
Dp=\hat\rho \, \bt,
\end{equation}
where $\hat\rho$ is the ($\pi$-periodic) lift of $\rho$ to $\Tt$. \\
Moreover, parametrized by arclength, $\partial K_{h_\rho}$ is a regular $C^1$ curve with unit tangent $\bt(\theta)$ and outward normal $\bn(\theta)$ at $p(\theta)$, and its length is $|\partial K_{h_\rho}|=\int_0^{2\pi}h_\rho\dd\theta$.
\item Conversely, every $h\in\cH$ is given by \eqref{eq:h-from-rho} with $\rho=\rho_h$.
\end{enumerate}
\end{proposition}

\medskip

\begin{proof}
(i) Let  $k(\theta):=\tfrac12|\sin\theta|$, which  is a $\pi$-periodic Lipschitz function.
Observe that $h_{\rho}=k*\rho$, where $*$ denotes the convolution.

\smallskip
Since $h_{\rho}$ is the convolution of a continuous function with a finite measure, then it is continuous. Using the fact that $k''+k=\delta_0$ in the sense of distributions on $\Tp$, we obtain that  $h_{\rho}''+h_{\rho}=(k''+k)*\rho=\rho$ distributionally.

\smallskip
We have that $k'(\theta)=\tfrac12\sgn(\sin\theta)\cos\theta$ for $\theta\notin\pi\Z$. Define the function $$h_1(\theta):=\int k'(\theta-u)\dd\rho(u),$$ which is well defined since $\rho(\{\theta\})=0$, and it is continuous by dominated convergence.
Fubini's theorem gives $h_{\rho}(\theta)-h_{\rho}(\theta_0)=\int_{\theta_0}^{\theta}h_1(\sigma)\dd\sigma$. Thus $h_{\rho}\in C^1$ with $h_{\rho}'=h_1$. 

\smallskip
Positivity of $h_{\rho}$ follows from the fact that  $k(\theta-u)>0$ unless $u\equiv\theta$, a $\rho$-null set, and $\rho\neq 0$. 
Therefore, $\rho_{h_{\rho}}=\rho$ is positive with full support, so $h_{\rho}\in\cH$.\\

\noindent (ii) Since $Dh_{\rho}'=\rho-h_{\rho}\dd\theta$ is a finite measure, $h_{\rho}'$ is a continuous function of bounded variation, and so is $p$. Product differentiation gives, as vector-valued measures on $\Tt$,
\[
Dp=(h_{\rho}'\bn+h_{\rho}\bt)\dd\theta+\bt\,Dh_{\rho}'+h_{\rho}'\,(-\bn)\dd\theta=(h_{\rho}\dd\theta+Dh_{\rho}')\, \bt=\hat\rho \,\bt,
\]
which proves \eqref{eq:Dp}. Hence, for $\alpha<\theta<\alpha+\pi$,
\begin{equation}\label{eq:strict}
\bigl(p(\theta)-p(\alpha)\bigr)\cdot\bn(\alpha)=\int_{(\alpha,\theta]}\bt(u)\cdot\bn(\alpha)\dd\hat\rho(u)=-\int_{(\alpha,\theta]}\sin(u-\alpha)\dd\hat\rho(u)<0,
\end{equation}
because $\sin(u-\alpha)>0$ on $(\alpha,\theta]$ and $\hat\rho((\alpha,\theta])>0$ by full support. For $\alpha-\pi<\theta<\alpha$ one finds similarly 
$$(p(\theta)-p(\alpha))\cdot\bn(\alpha)=\int_{(\theta,\alpha]}\sin(u-\alpha)\dd\hat\rho(u)<0.$$ Therefore
\begin{equation}\label{eq:unique-max}
p(\theta)\cdot\bn(\alpha)<p(\alpha)\cdot\bn(\alpha)=h_\rho(\alpha)\qquad\text{for all }\theta\not\equiv\alpha \pmod{2\pi}.
\end{equation}
In particular $p$ is injective, so $\Gamma:=p(\Tt)$ is a Jordan curve. 

\medskip
Let $K_{h_\rho}:=\Hull(\Gamma)$ be the convex hull of $\Gamma$. We use the following elementary
observation: if $\phi$ is a linear functional and $M:=\max_\Gamma\phi$, then
$\max_{\Hull(\Gamma)}\phi=M$ and
$$\{x\in\Hull(\Gamma):\phi(x)=M\}=\Hull\{c\in \Gamma:\phi(c)=M\}.$$ Indeed, writing
$x\in\Hull(\Gamma)$ as a finite convex combination $x=\sum_i\lambda_ic_i$ with
$c_i\in \Gamma$ and $\lambda_i>0$, linearity gives $\phi(x)=\sum_i\lambda_i\phi(c_i)
\le M$, with equality if and only if $\phi(c_i)=M$ for every $i$. 

\smallskip
Applying this
to the linear functional $x\mapsto x\cdot\bn(\alpha)$, whose set of maximisers on $\Gamma$ is the
single point $p(\alpha)$ by \eqref{eq:unique-max}, we conclude that
$$h_{K_{h_\rho}}(\alpha)=\max_{K_{h_\rho}}x\cdot\bn(\alpha)=h_{\rho}(\alpha) \qquad \forall\;\alpha,$$ and that the
supporting line $\{x\cdot\bn(\alpha)=h_{\rho}(\alpha)\}$ meets $K_{h_\rho}$ only at $p(\alpha)$.
Thus every $p(\alpha)$ is an exposed point of $K_{h_\rho}$ and $\Gamma\subset\partial K_{h_\rho}$;
since a Jordan curve has no proper subset homeomorphic to a circle, we conclude that $\partial K_{h_\rho}=\Gamma$. \\
Every supporting line of $K_{h_\rho}$ meets $K_{h_\rho}$ in exactly one point, so $K_{h_\rho}$ is strictly convex. Central symmetry follows from $p(\theta+\pi)=-p(\theta)$.

\medskip
For the arclength parametrization, note that by \eqref{eq:Dp} the total variation measure of $p$ is $\hat\rho |\bt| =\hat\rho$, so the length of the arc $p([\alpha,\beta])$ is $\hat\rho([\alpha,\beta])$ and the total length is 
$$\hat\rho(\Tt)=2\rho(\Tp)=2\int_{\Tp}h_\rho\dd\theta=\int_0^{2\pi}h_\rho\dd\theta,$$ 
where we used $\int_{\Tp}Dh_\rho'=0$.

\smallskip
 Define the \emph{arclength function} $\ell:\R\to\R$ and the \emph{arclength parametrization} $\gamma:\R\to\partial K$ by
\begin{equation}\label{eq:arclength}
\ell(\theta):=\begin{cases}\hat\rho([0,\theta)),&\theta\ge0,\\ -\hat\rho([\theta,0)),&\theta<0,\end{cases}
\qquad\qquad \gamma:=p\circ\ell^{-1},
\end{equation}
so that $\ell(\theta)$ is the signed length of the boundary arc from $p(0)$ to $p(\theta)$, $\ell(\theta+2\pi)=\ell(\theta)+L$ with $L$ the length, and $p=\gamma\circ\ell$. Since $\hat\rho$ is atomless with full support, $\ell$ is continuous and strictly increasing, hence an increasing homeomorphism of $\R$, and the push-forward $\ell_\#\hat\rho$ is the Lebesgue measure. Then, by \eqref{eq:Dp} and the change of variables $\sigma=\ell(u)$,
\[
\gamma(\sigma_1)-\gamma(\sigma_0)=\int_{[\ell^{-1}\sigma_0,\ell^{-1}\sigma_1)}\bt(u)\dd\hat\rho(u)=\int_{\sigma_0}^{\sigma_1}\bt\bigl(\ell^{-1}(\sigma)\bigr)\dd\sigma ,
\]
so $\gamma$ is $C^1$ and $L$-periodic with $\gamma'(\sigma)=\bt(\ell^{-1}(\sigma))$. This is the regular $C^1$ parametrization. 

\bigskip
\noindent (iii) Let $h\in\cH$ and $\rho=\rho_h$. As noted in  Remark \ref{rem:Hintro}, $\rho$ is atomless (an atom of $Dh'$ would be a jump of the continuous function $h'$), positive, finite and of full support. Let $\tilde h$ be defined from $\rho$ by \eqref{eq:h-from-rho}. Then $u:=h-\tilde h\in C^1(\Tp)$ satisfies $u+u''=0$, therefore $u=A\cos\theta+B\sin\theta$; being $\pi$-periodic, this implies $u\equiv0$.
\end{proof}

\bigskip

\begin{remark}\label{rem:regularities}
The angular parametrization $\theta\mapsto p(\theta)$ is merely continuous and of bounded variation, even though $\partial K$ is a regular $C^1$ curve. If $\rho_h$ is purely singular, then $p$ is differentiable with $p'(\theta)=0$ for Lebesgue-almost every $\theta$. 
In that case the arclength $\ell(\theta)$ is a continuous strictly increasing singular function. 
The tangent angle, as a function
of arclength, is a singular continuous increasing function, and its derivative vanishes almost everywhere and all its increase is carried by a Lebesgue-null set (like a {\it Cantor's staircase} at all scales).
\smallskip
In fact, consider the
real lift $\ell:\R\to\R$ defined in \eqref{eq:arclength},
and write $\vartheta:=\ell^{-1}$.
Since $D\ell$ is purely singular, there is a Borel set
$E\subset\R$ such that
\[
\Leb(E)=0,
\qquad
D\ell(\R\setminus E)=0.
\]
For every Borel set $A\subset\R$, we have
\[
\Leb(\ell(A))=D\ell(A),
\qquad
D\vartheta(\ell(A))=\Leb(A).
\]
Indeed, these identities hold first on intervals, by the
definition of the Stieltjes measures, and then on all
Borel sets by uniqueness of locally finite measures.
Set $N:=\ell(\R\setminus E)$. The first identity gives
\[
\Leb(N)=D\ell(\R\setminus E)=0.
\]
Since $\ell$ is bijective, $\R\setminus N=\ell(E)$, and
the second identity yields
\[
D\vartheta(\R\setminus N)
=
D\vartheta(\ell(E))
=
\Leb(E)
=
0.
\]
Thus $D\vartheta$ is concentrated on the Lebesgue-null
set $N$. Since $\vartheta$ is continuous and strictly
increasing, it is a singular continuous increasing
function, with $\vartheta'(\sigma)=0$ for
Lebesgue-almost every arclength parameter $\sigma$.
The normal angle is $\vartheta(\sigma)$, and the tangent
angle is $\vartheta(\sigma)+\pi/2$, so the same conclusion
holds for the tangent angle.
\end{remark}

\medskip

\begin{remark}\label{rem:H}
Let $K=-K$ be a compact convex body with $0\in\Int K$ and support function $h_K$, which is then $\pi$-periodic, positive and Lipschitz. The following facts are classical (see \cite[Sections 1.7 and 4.2]{Sch14}): 
\begin{enumerate}[label=\textup{(\alph*)}]
\item \emph{Convexity.} $h_K+h_K''\ge0$ in the sense of distributions, and conversely every positive $\pi$-periodic $h$ with $h+h''\ge0$ is the support function of a convex body $K=-K$. The measure $\rho_{h_K}=h_K\dd\theta+Dh_K'$ is the length measure of $\partial K$ in the normal angle: $\rho_{h_K}(I)$ is the length of the boundary arc whose outward normals lie in the arc $I$. When $\partial K$ is $C^2$ with everywhere positive curvature $\kappa$, $\rho_{h_K}=\kappa^{-1}\dd\theta$.
\item \emph{Segments $\leftrightarrow$ atoms.} $h_K$ is differentiable at $\theta$ if and only if the face $$\{x\in K: x\cdot\bn(\theta)=h_K(\theta)\}$$ is a single point; in that case $h_K'(\theta)$ is the $\bt(\theta)$-coordinate of that point. A segment of $\partial K$ of length $L$ with outward normal $\bn(\theta_0)$ is a jump of $h_K'$ at $\theta_0$, {\it i.e.},  an atom of mass $L$ of $\rho_{h_K}$ at $\theta_0$. Hence
\[
h_K\in C^1\quad\Longleftrightarrow\quad\rho_{h_K}\text{ is atomless}\quad\Longleftrightarrow\quad K\text{ is strictly convex.}
\]
\item \emph{Corners $\leftrightarrow$ gaps in the support.} A boundary point $y$ is a corner (has more than one supporting line) if and only if its outward normals fill a non-degenerate arc $[\alpha,\beta]$; then $p(\theta)=y$ for $\theta\in[\alpha,\beta]$ and $\rho_{h_K}$ vanishes on $(\alpha,\beta)$. Conversely, if $\rho_{h_K}$ vanishes on an open arc, the contact point is constant there and is a corner. Hence
\[
\supp\rho_{h_K}=\Tp\quad\Longleftrightarrow\quad\partial K\text{ has no corners.}
\]
\end{enumerate}

\smallskip
We stress that the $C^1$ regularity of $\partial K$ asserted in Proposition~\ref{prop:reconstruction} (ii)  comes from (c), not from (b): it is because the arclength function $\ell(\theta)=\hat\rho([0,\theta))$ is \emph{strictly} increasing (full support) as well as continuous (no atoms) that its inverse is a homeomorphism and the unit tangent $\bt(\ell^{-1}(\sigma))$ is a continuous function of arclength. If $\rho_{h_K}$ vanished on an arc $(\alpha,\beta)$, the function $\ell$ would be constant there, $\ell^{-1}$ would jump from $\alpha$ to $\beta$, and the tangent direction would jump by $\beta-\alpha$ at the corresponding arclength, producing a corner. 
\end{remark}

\medskip

\subsection{The outer billiard map}\label{subsec:T}
In this section we describe the main properties of the outer billiard map that we  need in the following, and introduce a new map better suited to studying centrally symmetric domains.
Since most of the facts are quite well-known and the proofs are relatively elementary, to keep the presentation smooth, we collect all proofs  in Appendix~\ref{app:proofs}.

\medskip

Throughout this section $h\in\cH$, $K:=K_h$, and $p$, $\bn$, $\bt$ are as in Proposition~\ref{prop:reconstruction}; observe that $\bt(\theta)$ is the unit tangent to $\partial K$ at $p(\theta)$ pointing in the counterclockwise direction.

\medskip

The {\it outer billiard map} reflects a point in the exterior of $K$  through the point at which a supporting line through it touches $K$ (see Figure \ref{figure:dual_billiard}). 
Since $K$ is strictly convex, every supporting line touches $K$ at exactly one point; 
but a point $x\notin K$ lies on exactly two supporting lines, and we always choose the one whose contact point $p(\theta)$ satisfies $x=p(\theta)+s\,\bt(\theta)$ with $s>0$, that is,  the ray from the contact point through $x$ points in the direction in which $\partial K$ is traversed counterclockwise.
See Figure \ref{fig:outer}.\\

\begin{definition}\label{def:contact}
Let $x\in \R^2\setminus K$. We denote by $\theta_K(x)$  the {\it contact angle}, {\it i.e.}, the one singled out by the supporting line determined by the convention above, and $s_K(x)$ the distance from $x$ to the contact point ({\it tangent length}). Similarly, we define $\check\theta_K(x)$, $\check s_K(x)$ to be  the corresponding quantities for the other supporting line, on which $x$ lies behind the contact point (Figure~\ref{fig:outer}).
\end{definition}

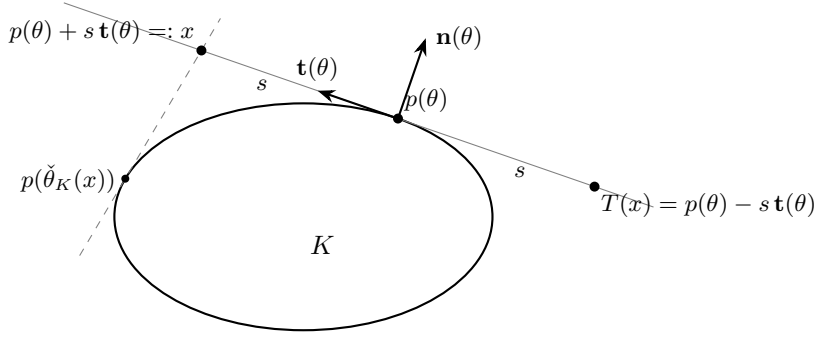
\begin{figure}[ht]
\centering
\begin{tikzpicture}[scale=1.25,>=Stealth]
\draw[thick] (0,0) ellipse (2 and 1.2);
\node at (0.2,-0.3) {$K$};
\coordinate (p) at (1,1.039);
\coordinate (x) at (-1.079,1.759);
\coordinate (Tx) at (3.079,0.319);
\coordinate (q) at (-1.884,0.403);
\draw[gray] ($(x)!-0.35!(Tx)$) -- ($(Tx)!-0.15!(x)$);
\draw[gray,dashed] ($(q)!-0.6!(x)$) -- ($(x)!-0.25!(q)$);
\fill (p) circle (1.5pt) node[above right=-1pt,font=\small] {$p(\theta)$};
\fill (x) circle (1.5pt) node[above left=-1pt,font=\small] {$p(\theta)+s\,\bt(\theta)=:x$};
\fill (Tx) circle (1.5pt) node[below right=-1pt,font=\small] {$T(x)=p(\theta)-s\,\bt(\theta)$};
\fill (q) circle (1.2pt) node[left,font=\small] {$p(\check\theta_K(x))$};
\draw[->,thick,black!70!black] (p) -- ($(p)+0.9*(0.327,0.945)$) node[right,font=\small] {$\bn(\theta)$};
\draw[->,thick,black!70!black] (p) -- ($(p)+0.9*(-0.945,0.327)$) node[above,font=\small] {$\bt(\theta)$};
\node[font=\small,black] at ($(p)!0.5!(x)+(-0.4, 0)$) {$s$};
\node[font=\small,black] at ($(p)!0.5!(Tx)+(0.25,-0.2)$) {$s$};
\end{tikzpicture}
\caption{\small The exterior point $x$ has contact angle $\theta=\theta_K(x)$ and tangent length $s=s_K(x)$, and $T(x)$ is its reflection through the contact point $p(\theta)$. The second supporting line through $x$ (dashed) touches $K$ at $p(\check\theta_K(x))$. The angle $\theta$ is the direction of the outward normal $\bn(\theta)$ at $p(\theta)$.}
\label{fig:outer}
\end{figure}

\bigskip

The following holds:

\begin{lemma}\label{lem:chart}
The map
\[
\Psi=\Psi_h:\ \Tt\times(0,\infty)\to\R^2\setminus K,\qquad \Psi(\theta,s):=p(\theta)+s\,\bt(\theta),
\]
is a homeomorphism with inverse $x\mapsto(\theta_K(x),s_K(x))$. The contact angle $\theta_K$ is of class $C^1$ on $\R^2\setminus K$, with
\begin{equation}\label{eq:Dtheta}
D\theta_K(x)\,v=-\frac{v\cdot\bn(\theta_K(x))}{s_K(x)}\qquad v\in\R^2,
\end{equation}
and the tangent length $s_K$ is continuous. Likewise $\check\theta_K$ is $C^1$ with $D\check\theta_K(x)v=v\cdot\bn(\check\theta_K(x))/\check s_K(x)$.
\end{lemma}

\medskip
\begin{remark}
We stress that the contact angle is a $C^1$ function of
the exterior point although the boundary parametrization
$\theta\mapsto p(\theta)$ need not be differentiable everywhere.
\end{remark}

\medskip

\begin{definition}\label{def:T}
The \emph{outer billiard map} about $K=K_h$ is given by
\[
T=T_h:\ \R^2\setminus K\to\R^2\setminus K,\qquad T(x):=2p(\theta_K(x))-x .
\]
In the $(\theta,s)$ coordinates,  $T(\Psi(\theta,s))=p(\theta)-s\,\bt(\theta)$.
\end{definition}

\medskip

\begin{remark}\label{rem:affinvariance}
The outer billiard map commutes with invertible orientation-preserving affine maps (it is {\it affinely equivariant}). Namely, if $A$ is
such a transformation and $K$ a convex body, then
$$ T_{A(K)} \circ A = A \circ T_K.$$
In particular, from the point of view of dual billiards, there is no difference between disks and ellipses.
\end{remark}

\medskip

We recall here some useful properties of $T$.\\

\begin{proposition}\label{prop:T-basic}
Let $h\in\cH$ and $T=T_h$ the associated outer billiard map.
\begin{enumerate}[label=\textup{(\alph*)}]
\item $T$ is a homeomorphism of $\R^2\setminus K$, with inverse $T^{-1}(x)=2p(\check\theta_K(x))-x$.
\item $T(-x)=-T(x)$ for all $x$.
\item $T$ preserves the Lebesgue measure. 
\item Writing $M:=\max_{y\in\partial K}|y|$, one has $|T(x)|\le|x|+2M$ and $|T^{-1}(x)|\le|x|+2M$; moreover $s_K(x)-M\le|x|\le s_K(x)+M$.
\end{enumerate}
\end{proposition}

\medskip

Although $T$ is in general only a homeomorphism, it is differentiable almost everywhere. We can prove the following properties of $DT$.

\begin{proposition}\label{prop:DT}
Let $h\in\cH$ and let $\rho_h=r\dd\theta+\rho_h^{\rm sing}$ be the Lebesgue decomposition of the radius of curvature measure, $r\in L^1(\Tp)$, $r\ge0$.
\begin{enumerate}[label=\textup{(\alph*)}]
\item There is a Lebesgue-null Borel set $E\subset\Tt$ such that for every $x=\Psi(\theta,s)$ with $\theta\notin E$ the map $T$ is differentiable at $x$, with
\begin{equation}\label{eq:DT}
DT(x)\,v=-v-\frac{2r(\theta)}{s}\,\bigl(v\cdot\bn(\theta)\bigr)\,\bt(\theta),\qquad v\in\R^2 ,
\end{equation}
that is, $DT(x)=-\mathrm{Id}-\frac{2r(\theta)}{s}\,\bt(\theta)\otimes\bn(\theta)$ ({\it i.e.,} a point reflection composed with a shear along the supporting line, of strength proportional to the density of the absolutely continuous part of $\rho_h$ at the contact angle). The exceptional set $\Psi(E\times(0,\infty))$ has zero Lebesgue measure.
\item In particular, if $\rho_h$ is purely singular, then $T$ is differentiable at Lebesgue-almost every point of $\R^2\setminus K$ with
\[
DT(x)=-\mathrm{Id}.
\]
\end{enumerate}
\end{proposition}

\bigskip

\begin{remark}\label{rem:commentDT}
Equation \eqref{eq:DT}  displays a fundamental difference between smooth and singular tables, on which most of our strategy relies.
\begin{itemize}
\item At a point $x$ with contact angle $\theta$, $DT(x)$ is the point reflection $-\mathrm{Id}$ composed with a shear along the supporting line, whose strength is  $\frac{2 r(\theta)}{s}$. 
\item For a $C^2$ table with everywhere positive curvature, $r$ is the radius of curvature, so $T$ is nowhere a point reflection (and this shear is what produces the twist condition). 
 \end{itemize}
When the radius of curvature measure is purely singular, $r=0$ almost everywhere; therefore, at almost every point, $T$ has the same first-order behaviour as the half-turn about its contact point: $x\mapsto 2p(\theta)-x$. We will get back to this polygon-like behaviour of the map in Section \ref{sec:heuristic}.
\end{remark}

\bigskip

Because of Proposition~\ref{prop:T-basic} (b), which follows from central symmetry, it is convenient to define the following map:
\begin{equation}\label{eq:F}
F:=-T,\qquad F(x)=x-2p(\theta_K(x)),\qquad F^2=T^2 .
\end{equation}
Observe that an orbit $\{x_n=T^nx_0\}_{n\in \Z}$ of $T$ corresponds to the orbit $\{y_n:=(-1)^nx_n\}_{n\in\Z}$ of the  map $F$. Central symmetry implies that $F(-x)=-F(x)$.\\

The following properties of $F$ will be needed in the following (see Figure \ref{fig:F}). \\

\begin{proposition}\label{prop:exact}
Let $x=\Psi(\theta,s)$ and write $F(x)=\Psi(\theta^+,s^+)$. Then $$d(\theta,s):=\theta^+-\theta\in(0,\pi)\pmod{2\pi}$$ and the following equations hold (we denote $d:=d(\theta,s)$):
\begin{align}
s\,\sin d&=h(\theta+d)+h(\theta)\cos d+h'(\theta)\sin d,\label{eq:exact1}\\
s^+&=s\cos d+h(\theta)\sin d-h'(\theta)\cos d-h'(\theta+d).\label{eq:exact2}
\end{align}
Moreover, \begin{enumerate}[label=\textup{(\roman*)}]
\item $\bigl|p(\theta)+p(\theta^+)\bigr|^2=(s-s^+)^2+4\,s\,s^+\sin^2(d/2)\;\le\;4M^2$;
\item $s\sin d=\bigl(p(\theta^+)-p(\theta+\pi)\bigr)\cdot\bn(\theta^+)>0$, and for every $\varepsilon\in(0,\pi)$ there is $c_\varepsilon>0$ such that $s\sin d\ge c_\varepsilon$ whenever $\pi-d\ge\varepsilon$.
\end{enumerate}

\end{proposition}

\medskip

\begin{figure}[ht]
\centering
\begin{tikzpicture}[scale=1.15,>=Stealth]
\draw[thick] (0,0) ellipse (2 and 1.2);
\coordinate (p) at (1,1.039);
\coordinate (mp) at (-1,-1.039);
\coordinate (x) at (-1.079,1.759);
\coordinate (Tx) at (3.079,0.319);
\coordinate (Fx) at (-3.079,-0.319);
\coordinate (pp) at (-1.523,0.778);
\draw[gray] ($(x)!-0.3!(Tx)$) -- ($(Tx)!-0.15!(x)$);
\draw[gray] ($(Fx)!-0.15!(mp)$) -- ($(mp)!-0.6!(Fx)$);
\draw[gray] ($(Fx)!-0.15!(pp)$) -- ($(pp)!-0.9!(Fx)$);
\fill (p) circle (1.4pt) node[above right=-2pt,font=\small] {$p(\theta)$};
\fill (mp) circle (1.4pt) node[below right=2pt,font=\small,xshift=16pt,yshift=-3pt] {$-p(\theta)=p(\theta+\pi)$};
\fill (pp) circle (1.4pt) node[above left=-2pt,font=\small] {$p(\theta^+)$};
\fill (x) circle (1.4pt) node[above,font=\small] {$x$};
\fill (Tx) circle (1.4pt) node[right,font=\small] {$T(x)$};
\fill (Fx) circle (1.4pt) node[left,font=\small] {$F(x)=-T(x)$};
\draw[dotted] (Tx) -- (Fx);
\draw[->,thick,black!70!black] (p) -- ($(p)+0.8*(0.327,0.945)$) node[right,font=\small] {$\bn(\theta)$};
\draw[->,thick,black!70!black] (pp) -- ($(pp)+0.8*(-0.577,0.817)$) node[left,font=\small] {$\bn(\theta^+)$};
\draw[->,thick,black!70!black] (mp) -- ($(mp)+0.8*(-0.327,-0.945)$) node[left,font=\small] {$\bn(\theta+\pi)$};
\end{tikzpicture}
\caption{\small The point $F(x)$ lies on the supporting lines with normals $\bn(\theta^+)$ (tangent length $s^+>0$) and $\bn(\theta+\pi)$ (signed tangent length $-s<0$); the angular displacement is $d=\theta^+-\theta\in(0,\pi)$.}
\label{fig:F}
\end{figure}
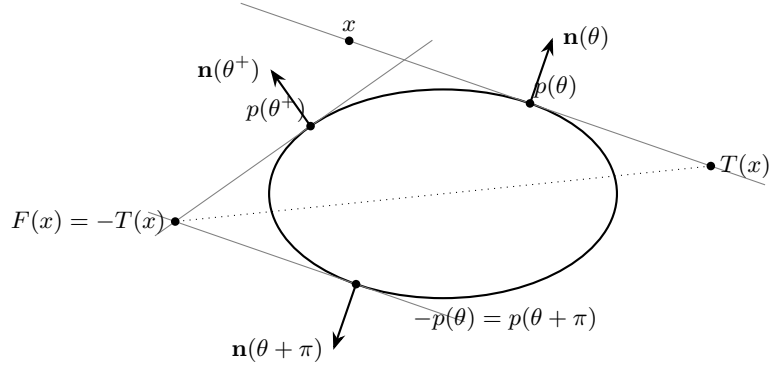

\bigskip

\begin{definition}\label{def:lift}
The function $d:\Tt\times(0,\infty)\to(0,\pi)$ defined in Proposition \ref{prop:exact} is continuous (it follows from Lemma~\ref{lem:chart}), so we define the \emph{lift} of $F$ in the $(\theta,s)$ coordinates by
\[
\tilde F(\theta,s):=\bigl(\theta+d(\theta,s),\,s^+(\theta,s)\bigr),\qquad (\theta,s)\in\R\times(0,\infty).
\]
By Proposition~\ref{prop:T-basic} (b), $\tilde F$ commutes with $(\theta,s)\mapsto(\theta+\pi,s)$. The corresponding lift of $T$ is $\tilde T(\theta,s)=(\theta+\delta(\theta,s),s^+)$ with $\delta=d-\pi\in(-\pi,0)$.
\end{definition}

\bigskip

\section{Rigidity Theorem}\label{sec:rigidity}
In this section we prove the Rigidity Theorem stated in the introduction, which lies behind our strategy.  We will state a more precise version in Theorem \ref{thm:rigidity}.

\medskip
Let us first start by introducing the definition of invariant graph.

\medskip

\begin{definition}\label{def:graph}
An \emph{invariant graph} for the outer billiard map $T$ is a curve
\[
\Gamma=\{Z(\theta):\theta\in\Tt\},\qquad Z(\theta):=p(\theta)+s(\theta)\,\bt(\theta),
\]
where $s:\Tt\to(0,\infty)$ is continuous, such that $T(\Gamma)=\Gamma$.  
We call $s$ the {\it height function of $\Gamma$}.\\
The invariant graph is said to be \emph{central} if $s(\theta+\pi)=s(\theta)$, {\it i.e.}, $\Gamma=-\Gamma$.

\smallskip
We call  \emph{region below} $\Gamma$, the bounded component of $(\R^2\setminus K)\setminus\Gamma$, namely  $W_\Gamma:=\Psi(\{(\theta,\sigma):0<\sigma<s(\theta)\})$.
\end{definition}

\medskip

\begin{lemma}\label{lem:induced}
Let $\Gamma$ be a central invariant  graph for $T$. Then $F(\Gamma)=\Gamma$, $T(W_\Gamma)=W_\Gamma$, and there is a unique increasing homeomorphism $\tilde g:\R\to\R$ such that
\begin{equation}\label{eq:g}
F\bigl(Z(\theta)\bigr)=Z\bigl(\tilde g(\theta)\bigr),\qquad 0<\tilde g(\theta)-\theta<\pi,\qquad \tilde g(\theta+\pi)=\tilde g(\theta)+\pi\qquad(\theta\in\R),
\end{equation}
where $Z$ is regarded as a $2\pi$-periodic map on $\R$. In particular, there exist  $0<d_*\le d^*<\pi$ with $d_*\le\tilde g(\theta)-\theta\le d^*$ for all $\theta$.
\end{lemma}

\begin{proof}
$F=R\circ T$ with $R=-\id$, and $R(\Gamma)=\Gamma$ by centrality, so $F(\Gamma)=\Gamma$. 
The homeomorphism $T$ of $\R^2\setminus K$ preserves $\Gamma$, hence either fixes or permutes the two components of $(\R^2\setminus K)\setminus\Gamma$; since $T$ maps bounded sets to bounded sets (Proposition~\ref{prop:T-basic} (d)) and the other component is unbounded, then $T(W_\Gamma)=W_\Gamma$.

\smallskip
Let $g:=Z^{-1}\circ F\circ Z$, a homeomorphism of $\Tt$. By Proposition~\ref{prop:exact}, $g(\theta)-\theta\in(0,\pi)\pmod{2\pi}$ for every $\theta$; in particular $g$ has no fixed point, hence $g$ must preserve orientation (since an orientation-reversing circle homeomorphism always has a fixed point).  

\smallskip
\noindent The map $\theta\mapsto d(\theta,s(\theta))$ is continuous with values in $(0,\pi)$, so $\tilde g(\theta):=\theta+d(\theta,s(\theta))$ is a continuous lift of $g$; as a lift of an orientation preserving circle homeomorphism it is an increasing homeomorphism of $\R$, and it is the unique lift with $\tilde g(\theta)-\theta\in(0,\pi)$. Since $Z(\theta+\pi)=-Z(\theta)$ and $F(-x)=-F(x)$, we have $F(Z(\theta+\pi))=-Z(g(\theta))=Z(g(\theta)+\pi)$, so $\tilde g(\theta+\pi)-\tilde g(\theta)-\pi\in2\pi\Z$; this continuous function of $\theta$ lies in $(-\pi,\pi)$, hence vanishes. 

\smallskip
The bounds $d_*,d^*$ exist by compactness of $\Tt$.
\end{proof}

\bigskip
\bigskip

Our goal is to prove the following theorem.

\begin{theorem}[{\bf Rigidity Theorem}]\label{thm:rigidity}
Let $h\in\cH$ with $\rho_h$ purely singular, and let $\Gamma$ be an invariant graph with height function $s$. Then there is an integer $q\ge2$ such that
\begin{equation}\label{eq:halfturn-chart}
\tilde F^{\,q}\bigl(\theta,s(\theta)\bigr)=\bigl(\theta+\pi,\,s(\theta)\bigr)\qquad\text{for all }\theta\in\R ,
\end{equation}
that is, $F^q(x)=-x$ for every $x\in\Gamma$. Consequently:
\begin{enumerate}[label=\textup{(\roman*)}]
\item every point of $\Gamma$ is $T$-periodic: $T^q|_\Gamma=\id$ if $q$ is odd, and $T^q|_\Gamma=-\id$, $T^{2q}|_\Gamma=\id$ if $q$ is even;
\item the rotation number of $T|_\Gamma$, for the lift with $\delta\in(-\pi,0)$, is $-\tfrac12+\tfrac1{2q}$; in particular no invariant graph has irrational rotation number;
\item if $q$ is even, then $\Gamma$ is central.
\end{enumerate}
\end{theorem}

\bigskip

\begin{remark}\label{rem:rigidity-comments}
{\bf (i)} Theorem~\ref{thm:rigidity} implies that every invariant graph consists of periodic points and satisfies a sort of {\it half-turn return} \eqref{eq:halfturn-chart} with a single $q$. In particular there are no invariant graphs with irrational rotation number, and no invariant graphs with rational rotation number carrying non-periodic points. Clearly, Theorem~\ref{thm:rigidity} implies the Rigidity Theorem stated in the Introduction (Section \ref{sec:intro}).

\noindent {\bf (ii)}
For odd $q$ the theorem does not exclude non-central invariant graphs; however, we do not know any example in which they  occur.

\noindent  {\bf (iii)} The negative sign of the rotation number depends on our definition of the outer billiard map. With $\theta$ increasing counterclockwise, the map $T$ moves points clockwise around $K$, so the rotation numbers of $T$ are negative, ranging from $0$ near $\partial K$ to $-\frac12$ near infinity. 
If $\Gamma$ is central, the rotation number of $F$ restricted to $\Gamma$ is $\frac1{2q}$, while the rotation number for $T^{-1}$  is $\frac12-\frac1{2q}$. The parity of $q$ affects only the reduced form of $\frac{1-q}{2q}$, whose denominator is the minimal period: $q$ if $q$ is odd, $2q$ if $q$ is even.
\end{remark}

\bigskip

The proof of Theorem \ref{thm:rigidity} will be in  two stages:

\medskip
\begin{itemize}
\item In Section \ref{sec:step1}, we treat the special case of a \emph{central} graph $\Gamma$, for which $F(\Gamma)=\Gamma$,  and the whole argument can be phrased in terms of the induced lift $\tilde g$ of Lemma~\ref{lem:induced}; see Theorem \ref{thm:rigiditycentral}.

\smallskip
\item In Section \ref{sec:step2}, we prove Theorem \ref{thm:rigidity},  by  a {\it symmetrization} argument and  Theorem \ref{thm:rigiditycentral}.
\end{itemize}

\medskip

\subsection{Rigidity result for central invariant graphs}\label{sec:step1}

We prove the following result.

\begin{theorem}\label{thm:rigiditycentral}
Let $h\in\cH$ with $\rho_h$ purely singular, and let $\Gamma$ be a central invariant graph with induced lift $\tilde g$. Then there is an integer $q\ge2$ such that
\begin{equation}\label{eq:halfturn}
\tilde g^{\,q}(\theta)=\theta+\pi\qquad\text{for all }\theta\in\R .
\end{equation}
Consequently $F^q|_\Gamma=-\id$ and $T^{2q}|_\Gamma=\id$.
\end{theorem}

\medskip

\begin{remark}\label{rem:afterthmrigcen}
For a central graph, \eqref{eq:halfturn} is the same as \eqref{eq:halfturn-chart}, because the restriction of $\tilde F$ to the graph of $s$ is $(\theta,s(\theta))\mapsto(\tilde g(\theta),s(\tilde g(\theta)))$: both lifts have angular increments in $(0,\pi)$. 
\end{remark}

\bigskip

Before discussing the technical aspects of this proof (see Section \ref{sec:rigproofrig}), let us describe informally the mechanism behind this rigidity result for central invariant graphs.

\medskip

\subsubsection{\bf Heuristic explanation of the mechanism}\label{sec:heuristic}
 The key is the relation
between the motion of the  \emph{contact point} and that of the 
\emph{contact angle}.

\medskip

Choose an increasing parametrization $X_0(\tau)$ of the central invariant graph $\Gamma$, obtained
by reparametrizing $Z$, and write
\[
 X_j(\tau):=F^j(X_0(\tau)),\qquad
 \theta_j:=\theta_K\circ X_j,\qquad P_j:=p\circ\theta_j.
\]

\smallskip

For the moment, suppose that a common parameter $\tau$ has been
chosen in which all these curves are locally Lipschitz, so that
they can be differentiated almost everywhere.
This choice will be justified in Lemma~\ref{lem:lipschitz}.

 Observe that $j$ represents the 
advance along an orbit, whereas $\tau$ varies its initial point. Thus $v_j:=X_j'(\tau)$ represents the infinitesimal separation
of neighbouring orbits at their $j$-th points; see Figure~\ref{fig:heur-objects}.

\begin{figure}[ht]
\centering
\resizebox{\linewidth}{!}{
\begin{tikzpicture}[
  x=1cm,y=1cm,
  every node/.style={heur label}
]

\path[use as bounding box]
  (-3.55,-2.0) rectangle (11.9,3.25);


\begin{scope}[shift={(-.15,.15)},scale=.94]

  \def\A{1.35}
  \def\B{.95}
  \def\S{1.75}

  \pgfmathsetmacro{\phi}{2*atan(\B/\S)}
  \pgfmathsetmacro{\px}{\A*cos(\phi)}
  \pgfmathsetmacro{\py}{\B*sin(\phi)}

  \coordinate (P)      at (\A,0);
  \coordinate (Pminus) at (-\A,0);
  \coordinate (X)      at (\A,\S);
  \coordinate (Y)      at (-\A,\S);
  \coordinate (Q)      at (\px,\py);

  \filldraw[fill=black!5,heur curve]
    (0,0) ellipse[x radius=\A,y radius=\B];

  \node at (-.15,-.61) {$K=-K$};
  \fill (0,0) circle[radius=1.1pt];
  \node[below=2pt] at (0,0) {$O$};

  \draw[heur support]
    (\A,-1.22) -- (\A,2.26);

  \draw[heur alternative]
    (-\A,-1.22) -- (-\A,2.26);

  \draw[heur support]
    ($(Q)!-.28!(Y)$) -- ($(Q)!1.18!(Y)$);

  \draw[heur vector]
    (X) -- node[above=3pt] {$-2P_j$} (Y);

  \draw[heur vector]
    (P) -- ++(.90,0)
    node[right] {$\bn(\theta_j)$};

  \draw[heur vector]
    (P) -- ++(0,.63);

  \node[right=5pt] at (\A,.53)
    {$\bt(\theta_j)$};

  \node[right=5pt] at (\A,1.14)
    {$s_j$};

  \fill (P) circle[radius=1.7pt];
  \node[below right=2pt] at (P) {$P_j$};

  \fill (Pminus) circle[radius=1.5pt];
  \node[below left=2pt] at (Pminus) {$-P_j$};

  \fill (X) circle[radius=1.7pt];
  \node[above right=2pt] at (X) {$X_j$};

  \fill (Y) circle[radius=1.7pt];
  \node[above left=2pt] at (Y)
    {$X_{j+1}=F(X_j)$};

  \fill (Q) circle[radius=1.5pt];
  \node[below left=1pt] at (Q) {$P_{j+1}$};

\end{scope}


\begin{scope}[shift={(7.2,0)}]

  \coordinate (A) at (-1.9,-.15);
  \coordinate (B) at (-1.66,.75);
  \coordinate (C) at (1.75,-.15);
  \coordinate (D) at (1.96,.53);

  \draw[heur curve]
    (-2.10,-.9)
    .. controls (-1.99,-.45) and (-1.81,.26) ..
    (-1.66,.75)
    .. controls (-1.57,1.05) and (-1.41,1.40) ..
    (-1.13,1.72);

  \draw[heur curve]
    (1.36,-.91)
    .. controls (1.63,-.50) and (1.78,-.16) ..
    (1.96,.53)
    .. controls (2.05,.88) and (2.08,1.34) ..
    (2.02,1.72);

  \node[left] at (-2.10,-.9) {$\Gamma$};
  \node[right] at (2.02,1.72) {$\Gamma$};

  \draw[heur vector]
    (A) -- node[left=4pt] {$\varepsilon v_j$} (B);

  \draw[heur vector]
    (C) -- node[right=4pt] {$\varepsilon v_{j+1}$} (D);

  \draw[heur small arrow]
    (A) to[bend right=17]
    node[below=2pt] {$F$} (C);

  \draw[heur small arrow]
    (B) to[bend left=17]
    node[above=2pt] {$F$} (D);

  \foreach \name in {A,B,C,D}{
    \fill (\name) circle[radius=1.6pt];
  }

  \node[below left=2pt] at (A)
    {$X_j(\tau)$};

  \node[above left=2pt] at (B)
    {$X_j(\tau+\varepsilon)$};

  \node[below right=2pt] at (C)
    {$X_{j+1}(\tau)$};

  \node[above right=2pt] at (D)
    {$X_{j+1}(\tau+\varepsilon)$};

\end{scope}

\node[heur panel,anchor=center] at (0,-1.8)
  {(a)};

\node[heur panel,anchor=center] at (7.25,-1.8)
  {(b)};
\end{tikzpicture}%
}

\caption{\small Contact data (a) and variations of an orbit (b), where
separation arrows are meant to be first-order illustrations (the finite displacement is \(\varepsilon v_j+o(\varepsilon)\)).
}
\label{fig:heur-objects}

\end{figure}
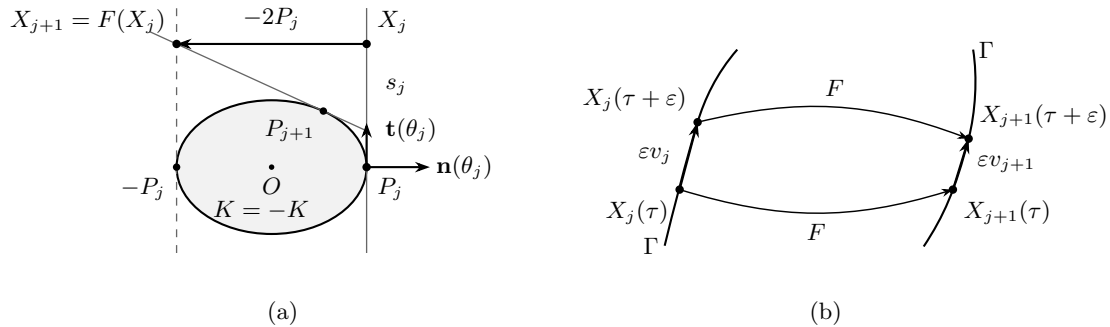

\smallskip

\medskip
Recall that
\begin{equation}\label{eq:heur-recursion}
 X_{j+1}=X_j-2P_j,
 \qquad X_j=P_j+s_j\bt(\theta_j),\qquad s_j=s\circ\theta_j.
\end{equation}
The contact point moves in the positive tangent direction, so write
$P_j'=a_j\bt(\theta_j)$ with $a_j\geq0$. Differentiating gives
\begin{equation}\label{eq:heur-velocity}
 v_{j+1}=v_j-2a_j\bt(\theta_j),
 \qquad v_j\cdot\bn(\theta_j)=-s_j\theta_j'\leq0.
\end{equation}

Observe that the inequality in \eqref{eq:heur-velocity} expresses order preservation: as the initial point
moves forward along $\Gamma$, the contact angles of all its iterates
are nondecreasing.

\smallskip
We will show (see Lemma~\ref{lem:complementary}) that the fact that $\rho_h$ is purely singular (which implies that
$p'(\theta)=0$ for almost every $\theta$, see Remark~\ref{rem:regularities}), yields the following relation:
\begin{equation}\label{eq:heur-complementarity}
 a_j\theta_j'=0\qquad\text{for almost every }\tau.
\end{equation}

\smallskip
In short, the contact point can ``{\it slide}'' or the normal can ``{\it turn}'', but
they cannot do both at non-zero speed. In particular:

\begin{itemize}
\item When $a_j=0$, the contact point is stationary at first order and
\eqref{eq:heur-velocity} gives $v_{j+1}=v_j$: the infinitesimal
separation is carried unchanged. This is the \emph{vertex-like}
regime, in which $F$ acts as a translation to first order along the
curve. 

\smallskip
\item When $a_j>0$, instead, \eqref{eq:heur-complementarity} gives
$\theta_j'=0$, and \eqref{eq:heur-velocity} forces
$v_j\parallel\bt(\theta_j)$, possibly with $v_j=0$.
This is the \emph{edge-like} regime: the infinitesimal separation can be modified,
but only along the supporting line. 
\end{itemize}
Figure~\ref{fig:heur-motions}
illustrates these two regimes. Both speeds may also vanish.
Unlike a polygon, our table has neither vertices nor edges; their
first-order behaviours are interwoven at arbitrarily small scales,
as explained in Remark~\ref{rem:antipodalswitch}.

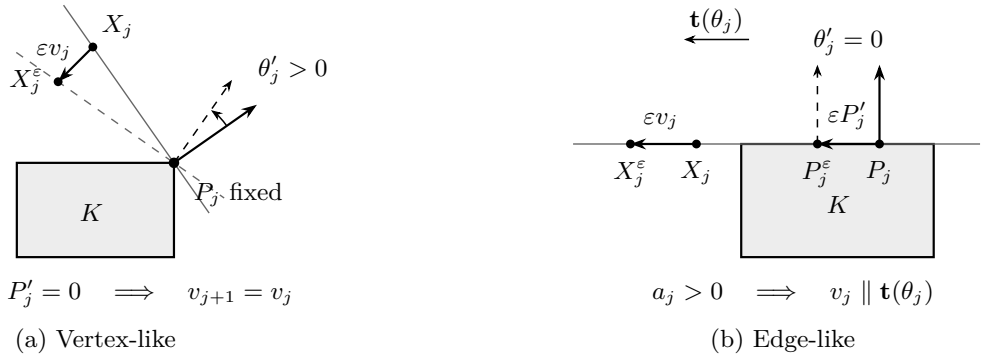
\begin{figure}[ht]
\centering
\resizebox{\linewidth}{!}{
\begin{tikzpicture}[
  x=1cm,y=1cm,
  every node/.style={heur label}
]

\path[use as bounding box]
  (-3.15,-2.3) rectangle (11.45,3.0);


\begin{scope}[shift={(0,.02)}]

  \coordinate (P)  at (0,0);
  \coordinate (X)  at (125:1.8);
  \coordinate (Xe) at (145:1.8);

  \filldraw[fill=black!7,heur curve]
    (-2.0,-1.20) --
    (0,-1.20) --
    (0,0) --
    (-2.0,0) -- cycle;

  \node at (-1.05,-.63) {$K$};

  \draw[heur support]
    (125:2.45) -- (-55:.78);

  \draw[heur alternative]
    (145:2.45) -- (-35:.78);

  \draw[heur vector]
    (P) -- (35:1.29);

  \draw[heur small arrow,dashed]
    (P) -- (55:1.29);

  \draw[heur small arrow]
    (35:.82)
    arc[start angle=35,end angle=55,radius=.82];

  \node[above right] at (43:1.36)
    {$\theta_j'>0$};

  \draw[heur vector]
    (X) -- node[above left=-1pt]
    {$\varepsilon v_j$} (Xe);

  \fill (X) circle[radius=1.6pt];
  \fill (Xe) circle[radius=1.6pt];

  \node[above right=1pt] at (X) {$X_j$};
  \node[left=3pt] at (Xe) {$X_j^\varepsilon$};

  \fill (P) circle[radius=1.9pt];
  \node[below right=5pt] at (P)
    {$P_j\ \text{fixed}$};

\end{scope}

\node[heur note] at (-.30,-1.64)
  {$P_j'=0
    \quad\Longrightarrow\quad
    v_{j+1}=v_j$};


\begin{scope}[shift={(7.7,.02)}]

  \filldraw[fill=black!7,heur curve]
    (-.48,-1.20) --
    (1.97,-1.20) --
    (1.97,.24) --
    (-.48,.24) -- cycle;

  \node at (.76,-.53) {$K$};

  \draw[heur support]
    (-2.62,.24) -- (2.56,.24);

  \coordinate (P)  at (1.28,.24);
  \coordinate (Pe) at (.49,.24);
  \coordinate (X)  at (-1.05,.24);
  \coordinate (Xe) at (-1.89,.24);

  \draw[heur vector] (P) -- (Pe);
  \node[above=3pt] at (.87,.24)
    {$\varepsilon P_j'$};

  \draw[heur vector] (X) -- (Xe);
  \node[above=3pt] at (-1.50,.24)
    {$\varepsilon v_j$};

  \draw[heur vector]
    (P) -- ++(0,1.01);

  \draw[heur small arrow,dashed]
    (Pe) -- ++(0,1.01);

  \node[above] at (.885,1.29)
    {$\theta_j'=0$};

  \draw[heur small arrow]
    (-.38,1.57) -- (-1.21,1.57)
    node[midway,above] {$\bt(\theta_j)$};

  \foreach \name in {P,Pe,X,Xe}{
    \fill (\name) circle[radius=1.6pt];
  }

  \node[below=4pt] at (P)  {$P_j$};
  \node[below=4pt] at (Pe) {$P_j^\varepsilon$};
  \node[below=4pt] at (X)  {$X_j$};
  \node[below=4pt] at (Xe) {$X_j^\varepsilon$};

  \node[heur note] at (.18,-1.66)
    {$a_j>0
      \quad\Longrightarrow\quad
      v_j\parallel\bt(\theta_j)$};

\end{scope}

\node[heur panel,anchor=center] at (-1,-2.25)
  {(a) Vertex-like};

\node[heur panel,anchor=center] at (7.75,-2.25)
  {(b) Edge-like};

\end{tikzpicture}
}

\caption{\small An idealization of vertex-like and edge-like behaviours. The polygonal shapes are local models of first-order behaviour, not actual corners or edges of the table.}
\label{fig:heur-motions}

\end{figure}

\smallskip\noindent
Let us now explain why periodicity is forced.
Recall that the lifted contact angles advance at each iteration
by an amount in $(0,\pi)$, bounded away from zero on $\Gamma$.

\smallskip
Fix a parameter at which all the above identities hold and
$a_0>0$ (the edge-like case).
By \eqref{eq:heur-complementarity} and
\eqref{eq:heur-velocity}, both $v_0$ and $v_1$ are parallel
to $\bt(\theta_0)$.

\smallskip
Suppose first that $v_1\ne0$.
Applying the inequality in \eqref{eq:heur-velocity} at $j=1$,
and using $0<\theta_1-\theta_0<\pi$, gives
\[
v_1=-c\bt(\theta_0)
\qquad\text{for some }c>0.
\]
As long as $a_1,\ldots,a_{j-1}$ vanish, this vector is
transported unchanged:
\[
v_j=v_1=-c\bt(\theta_0).
\]
It cannot remain unchanged up to the first iterate whose
contact angle exceeds $\theta_0+\pi$.
Indeed, because each angular increment is smaller than $\pi$,
at that iterate we would have
$\pi<\theta_j-\theta_0<2\pi$, and hence
\[
v_j\cdot\bn(\theta_j)
=-c\sin(\theta_j-\theta_0)>0,
\]
contrary to \eqref{eq:heur-velocity}.

Therefore a new slide must occur strictly before the contact angle exceeds $\theta_0+\pi$. Denote by $j^*\geq1$ the first index with
$a_{j^*}>0$. Then
\[
0<\theta_{j^*}-\theta_0\leq\pi.
\]
At this index, the incoming vector
$v_{j^*}=v_1\ne0$ must be parallel to
$\bt(\theta_{j^*})$. Since it is also parallel to
$\bt(\theta_0)$, the two supporting directions must be parallel.
Consequently,
\[
\theta_{j^*}=\theta_0+\pi.
\]

\medskip
In short, a new slide is impossible before the opposite
supporting direction, while passing that direction without
a slide would violate the order constraint.
The next slide must therefore occur exactly at the half-turn.

\smallskip
If $v_1=0$, then the same strategy can be applied
backwards. 

\smallskip
A precise proof of all of this will be given in Lemma~\ref{lem:switch}.

\medskip
Finally, the set of parameters to which the preceding argument applies ({\it i.e.}, $a_0(\tau)>0$)
meet every interval
(see Remark~\ref{rem:antipodalswitch}).
Thus the half-turn occurs for a dense set of initial points and
we will show that the number of iterations must be the same integer $q$ throughout
this dense set. Continuity then extends the identity
$\widetilde g^{\,q}(\theta)=\theta+\pi$ to every initial angle, thus proving that the invariant graph is periodic.

\bigskip

\subsubsection{\bf Rigorous proof of the Rigidity Theorem for central graphs}\label{sec:rigproofrig}
Let us now proceed more rigorously and prove that the mechanism described above does actually hold.

\smallskip

Let us first prove some regularity properties of the graph.

\medskip

\begin{lemma}\label{lem:BV}
The function $s$ is of bounded variation on $\Tt$. Hence $Z$ is a continuous function of bounded variation and $\Gamma$ is rectifiable.
\end{lemma}

\begin{proof}
Solving \eqref{eq:exact1} with $d=\tilde g(\theta)-\theta$ gives
\begin{equation}\label{eq:s-formula}
s(\theta)=h'(\theta)+\frac{h(\tilde g(\theta))+h(\theta)\cos(\tilde g(\theta)-\theta)}{\sin(\tilde g(\theta)-\theta)} .
\end{equation}
We have that:
\begin{itemize}
\item $h'\in BV$ because $Dh'=\rho_h-h\dd\theta$ is a finite measure;
\item $h\circ\tilde g\in BV$ because $h$ is Lipschitz and $\tilde g$ is monotone; 
\item $\tilde g(\theta)-\theta$ is of bounded variation, with values in $[d_*,d^*]\subset(0,\pi)$ (Lemma~\ref{lem:induced}), so $\cos(\tilde g-\theta)$ and $1/\sin(\tilde g-\theta)$ are compositions of Lipschitz functions with a $BV$ function and are $BV$. 
\end{itemize}
This proves that $s$ is BV, since it is obtained by products and sums of $BV$ functions.

\smallskip
\noindent Finally $p$ is $BV$ by \eqref{eq:Dp}, so $Z=p+s\bt$ is $BV$; a continuous $BV$ curve is rectifiable.
\end{proof}

\bigskip

For $j\in\Z$ let $X_j:=Z\circ\tilde g^{\,j}:\R\to\R^2$, a $2\pi$-periodic continuous function of bounded variation. 
The heuristic discussion in Section~\ref{sec:heuristic} assumed that the $X_j$ can be differentiated in a common parameter. In the singular case this is quite delicate. In fact,  $Z$ is only of bounded variation in the normal angle $\theta$ (it has ``staircase'' pieces where the point moves while $\theta$ hardly changes), and if one uses the arclength of $\Gamma$ instead, then $X_0$ becomes Lipschitz but $X_j=X_0\circ g^j$ need not be. Let us introduce a single parameter, 
in which \emph{all} the $X_j$ are Lipschitz simultaneously.

\bigskip
\bigskip

We write $|DX_j|$ for the  variation measure of $X_j$ on $\Tt$, {\it i.e.}, the finite positive measure with $|DX_j|([\alpha,\beta])=\Var(X_j;[\alpha,\beta])$ for $\alpha<\beta<\alpha+2\pi$, where
\[
\Var(X;[\alpha,\beta]):=\sup\Bigl\{\sum_{i=0}^{m-1}|X(t_{i+1})-X(t_i)|:\ \alpha=t_0<t_1<\dots<t_m=\beta\Bigr\}.
\]

\medskip

\begin{lemma}\label{lem:variation}
For every $j\in\Z$ and $\alpha<\beta<\alpha+2\pi$,
\[
|DX_j|([\alpha,\beta])=|DZ|\bigl([\tilde g^{\,j}\alpha,\tilde g^{\,j}\beta]\bigr)=\text{length of the arc of }\Gamma\text{ from }X_j(\alpha)\text{ to }X_j(\beta).
\]
In particular $|DX_j|(\Tt)=|\Gamma|$, the length of $\Gamma$, for every $j$, and $|DX_j|$ is atomless and has full support.
\end{lemma}

\medskip

\begin{proof}
Recall that the total variation is invariant under increasing reparametrization; namely, if $\phi:[\alpha,\beta]\to[\alpha',\beta']$ is an increasing homeomorphism, then 
$$\Var(Z\circ\phi;[\alpha,\beta])=\Var(Z;[\alpha',\beta'])$$
($\phi$ induces a bijection between the partitions of $[\alpha,\beta]$ and those of $[\alpha',\beta']$).
Apply this with
$\phi=\tilde g^{\,j}|_{[\alpha,\beta]}$, an increasing homeomorphism onto $[\tilde g^{\,j}\alpha,\tilde g^{\,j}\beta]$ (Lemma~\ref{lem:induced}): this gives the first equality.

\smallskip
For the second equality, observe that $Z$ is a continuous injective parametrization of the Jordan curve $\Gamma$ on intervals of length $<2\pi$, and the variation of an injective continuous curve is by definition the length of its image. Taking $\beta\uparrow\alpha+2\pi$ and using that 
$$\tilde g^{\,j}(\alpha+2\pi)=\tilde g^{\,j}(\alpha)+2\pi$$ gives $|DX_j|(\Tt)=|DZ|(\Tt)=|\Gamma|$, which is finite by Lemma~\ref{lem:BV}.

\smallskip
An atom of $|DX_j|$ at $\theta$ would have mass $\lim_{\epsilon\to0}\Var(X_j;[\theta-\epsilon,\theta+\epsilon])$, which is the size of the jump of $X_j$ at $\theta$, hence $0$ since $X_j$ is continuous.

\smallskip
Finally $|DX_j|$ must charge every non-degenerate arc because $X_j$ is injective on it, so its image has positive length.
\end{proof}

\medskip

Thus all the $X_j$ have the same total variation, namely the length of $\Gamma$: each of them traverses the curve exactly once per period, only with a different time law. We now choose a parameter on $\Gamma$ in which all these time laws are Lipschitz at once. 

\medskip

\begin{definition}
We introduce 
\begin{equation}\label{eq:Lambda}
\Lambda:=\sum_{j\in\Z}2^{-|j|}\,|DX_j|,
\end{equation}
which is a finite positive atomless measure of full support on $\Tt$, of total mass $3|\Gamma|=:L$.

\medskip
Let us define
\begin{eqnarray*}
\tau: [0,2\pi) &\longrightarrow& [0,L)\\
\theta &\longmapsto& \Lambda([0,\theta)),
\end{eqnarray*}
which is a homeomorphism, and it extends to an increasing homeomorphism $\tilde\tau:\R\to\R$ with $\tilde\tau(\theta+2\pi)=\tilde\tau(\theta)+L$.
\end{definition}

\medskip
We use $\tau$ as the new parameter on $\Gamma$ and write, abusing notation,
\[
\theta_0(\tau):=\tilde\tau^{-1}(\tau),\qquad \theta_j(\tau):=\tilde g^{\,j}(\theta_0(\tau)),\qquad X_j(\tau):=Z(\theta_j(\tau)),\qquad s_j(\tau):=s(\theta_j(\tau)),
\]
so that $X_j(\tau)=\Psi(\theta_j(\tau),s_j(\tau))$ and $X_{j+1}=F\circ X_j$.\\

\medskip

Let us now prove that with this choice of a new parameter, all the $X_j$ become Lipschitz continuous.

\medskip

\begin{lemma}\label{lem:lipschitz}
For every $j\in\Z$ the map $\tau\mapsto X_j(\tau)$ is Lipschitz with constant $2^{|j|}$, and $\tau\mapsto\theta_j(\tau)$ is an increasing homeomorphism of $\R$ which is locally Lipschitz.
\end{lemma}

\begin{proof}
For $\theta<\theta'<\theta+2\pi$:
$$|X_j(\theta)-X_j(\theta')|\le\Var(X_j;[\theta,\theta'])=|DX_j|([\theta,\theta'])\le2^{|j|}\Lambda([\theta,\theta'])=2^{|j|}(\tilde\tau(\theta')-\tilde\tau(\theta)).$$
Since $\theta_j=\theta_K\circ X_j$ with $\theta_K$ of class $C^1$ and $|D\theta_K|=1/s_K\le1/\min s$ on $\Gamma$ (Lemma~\ref{lem:chart}), then $\theta_j$ is locally Lipschitz; it is increasing because $\tilde\tau^{-1}$ and $\tilde g$ are.
\end{proof}

\medskip
Since $X_{j+1}=F(X_j)=X_j-2p(\theta_j)$, the contact point of $X_j$ is given by:
\begin{equation}\label{eq:Pj}
P_j(\tau):=p(\theta_j(\tau))=\tfrac12\bigl(X_j(\tau)-X_{j+1}(\tau)\bigr),
\end{equation}
a Lipschitz function of $\tau$.

\bigskip
We define $N\subset\R$ to be the set of parameters at which all $X_j$, $\theta_j$, $P_j$ ($j\in\Z$) are differentiable; by Rademacher's theorem its complement is Lebesgue-null. 

\bigskip

Let us prove the following relations. On $N$, we set $v_j:=X_j'$.

\medskip
\begin{lemma}\label{lem:bookkeeping}
For $\tau\in N$ and $j\in\Z$:
\begin{enumerate}[label=\textup{(\alph*)}]
\item $P_j'(\tau)=a_j(\tau)\,\bt(\theta_j(\tau))$ for a number $a_j(\tau)\ge0$;
\item $v_{j+1}=v_j-2a_j\,\bt(\theta_j)$;
\item $\theta_j'\ge0$ and $v_j\cdot\bn(\theta_j)=-s_j\,\theta_j'$.
\end{enumerate}
\end{lemma}

\begin{proof}
(a) Let $\ell:\R\to\R$ and $\gamma:\R\to\partial K$ be the arclength function and the arclength parametrization \eqref{eq:arclength}, so that $p=\gamma\circ\ell$ and $P_j=\gamma\circ\ell_j$, where $\ell_j:=\ell\circ\theta_j$ (the arclength position of the contact point of $X_j$) is continuous and increasing, being a composition of two such functions.

Since $\gamma$ is $C^1$ with $|\gamma'|=1$, then there is $\eta>0$ with $$\tfrac12|a-b|\le|\gamma(a)-\gamma(b)|\le|a-b|
\qquad {\rm whenever}\quad |a-b|\le\eta.$$ Fix $\tau\in N$ and let $w=P_j'(\tau)$. For small $u\ne0$ put $\Delta\ell:=\ell_j(\tau+u)-\ell_j(\tau)$, which tends to $0$ as $u\to0$. Then $$|\Delta\ell|\le2|P_j(\tau+u)-P_j(\tau)|\le C|u|,$$ so $\Delta\ell/u$ is bounded, and
\[
\frac{P_j(\tau+u)-P_j(\tau)}{u}=\gamma'(\ell_j(\tau))\,\frac{\Delta\ell}{u}+\frac{o(|\Delta\ell|)}{u}
=\bt(\theta_j(\tau))\,\frac{\Delta\ell}{u}+o(1).
\]
Taking the scalar product with $\bt(\theta_j(\tau))$ shows that $a_j:=\lim_{u\to0}\Delta\ell/u$ exists, and then $$w=a_j\bt(\theta_j(\tau)),$$ with $a_j\ge0$ since $\ell_j$ is increasing.

\smallskip
\noindent (b) This relation simply follows from differentiating $X_{j+1}=X_j-2P_j$.

\smallskip
\noindent (c) $\theta_j$ is increasing, so $\theta_j'\ge0$. Since $\theta_j=\theta_K\circ X_j$ with $\theta_K\in C^1$, the chain rule and \eqref{eq:Dtheta} give $$\theta_j'=D\theta_K(X_j)v_j=-v_j\cdot\bn(\theta_j)/s_j.$$
\end{proof}

\bigskip
We now bring the singularity of $\rho_h$ into the discussion and prove what we described in Section \ref{sec:heuristic}: the contact point can slide or the normal can turn, but not both at non-zero speed. 

\medskip

\begin{lemma}\label{lem:complementary}
Let $\rho_h$ be purely singular.
After removing from $N$ a Lebesgue-null set, one has 
\begin{equation}\label{eq:complementary}
a_j(\tau)\,\theta_j'(\tau)=0 \qquad \forall \; j\in\Z, \quad \forall\; \tau\in N.
\end{equation}
In particular, if $\tau\in N$ and $a_j(\tau)>0$, then $\theta_j'(\tau)=0$ and $v_j(\tau)$ is parallel to $\bt(\theta_j(\tau))$.
\end{lemma}

\medskip
\begin{proof}
By Remark~\ref{rem:regularities}, $p$ is differentiable with $p'(\theta)=0$ for Lebesgue-almost every $\theta$ (this is a consequence of pure singularity). Let $E\subset\R$ be a Borel Lebesgue-null set containing every $\theta$ at which $p$ fails to be differentiable or $p'(\theta)\neq0$.

\smallskip
Fix $j$. If $\tau\in N$ and $\theta_j(\tau)\notin E$, then $p$ is differentiable at $\theta_j(\tau)$ with derivative $0$ and $\theta_j$ is differentiable at $\tau$. Then:
$$P_j'(\tau)=p'(\theta_j(\tau))\,\theta_j'(\tau)=0,$$
which implies, by Lemma \ref{lem:bookkeeping} (a), that  $a_j(\tau)=0$.

\smallskip
On the Borel set $B_j:=\theta_j^{-1}(E)$ we apply the area formula for the locally Lipschitz map $\theta_j:\R\to\R$ (see, for example, \cite[Section 3.3]{EG15}): for every bounded Borel set $A$
\[
\int_A|\theta_j'(\tau)|\dd\tau=\int_\R\#\bigl(A\cap\theta_j^{-1}(\{\vartheta\})\bigr)\dd\vartheta .
\]
For $A=B_j\cap[-R,R]$ the right-hand side is an integral over $\theta_j(A)\subset E$, a null set, hence zero. Since $\theta_j'\ge0$ we get $\theta_j'=0$ a.e.\ on $B_j$. Removing from $N$ the null sets where this fails, for all $j\in\Z$, gives \eqref{eq:complementary}.

The last assertion follows from Lemma~\ref{lem:bookkeeping} (c).
\end{proof}

\medskip

\begin{remark}\label{rem:antipodalswitch}
Therefore Lemma~\ref{lem:complementary} shows that, considering a suitable  parameter $\tau$, the heuristic described in Section \ref{sec:heuristic}  holds true.

Let us denote by 
$$V_j:=\{\tau\in N: a_j(\tau)>0\} \qquad {\rm and } \qquad A_j:=\{\tau\in N:\theta_j'(\tau)>0\}$$
the sets, respectively,  at which the contact point $P_j$ moves and the one at which the
contact angle advances. Lemma~\ref{lem:complementary} says precisely that $V_j\cap
A_j=\emptyset$. \\
Observe that neither set is small: since $\theta_j$ is an increasing
homeomorphism and absolutely continuous, then
$$\int_{I\cap
A_j}\theta_j'\,\dd\tau=\theta_j(\beta)-\theta_j(\alpha)>0 \qquad \mbox{on every interval $I=[\alpha,\beta]$}, \quad \alpha<\beta,$$
 and since $P_j$ is Lipschitz and injective on short intervals, we also get
$$\int_{I\cap V_j}a_j\,\dd\tau\ge|P_j(\beta)-P_j(\alpha)|>0,$$
where the last inequality holds for intervals shorter than a parameter period. Thus in every
interval of parameters the angle turns on one set of positive measure and the
contact point slides on another, disjoint, set of positive measure.
\end{remark}

\bigskip

We can now prove the {\it half-turn} phenomenon that underlies the  Rigidity theorem for central graphs.
The strategy is exactly what was described in Section \ref{sec:heuristic} and we will proceed along the same steps.

\medskip
\begin{lemma}\label{lem:switch}
Let $\tau\in N$ with $a_0(\tau)>0$. Then there is an integer $j\ge2$ such that $\theta_j(\tau)=\theta_0(\tau)+\pi$, i.e.\ $\tilde g^{\,j}(\theta_0(\tau))=\theta_0(\tau)+\pi$.
\end{lemma}

\begin{proof}
Since all quantities are evaluated at the fixed $\tau \in N$,  we suppress it in the notation. By Lemma~\ref{lem:complementary}, $v_0=b\,\bt(\theta_0)$ for some $b\in\R$, and by Lemma~\ref{lem:bookkeeping} (b)
$$v_1=(b-2a_0)\,\bt(\theta_0).$$
Let us recall the following properties:
\begin{itemize}
\item[(*)] For every $j$, $v_j\cdot\bn(\theta_j)\le0$ (Lemma~\ref{lem:bookkeeping} (c)).
\item[(**)] If $a_j>0$ then $v_j\parallel\bt(\theta_j)$ (Lemma~\ref{lem:complementary}).
\end{itemize}
Recall also that $\theta_{j+1}-\theta_j\in[d_*,d^*]\subset(0,\pi)$ for all $j$, and that $\bt(\alpha)\cdot\bn(\beta)=\sin(\beta-\alpha)$.

\smallskip
First, $b\le2a_0$: applying (*) at $j=1$ gives
$$v_1\cdot\bn(\theta_1)=(b-2a_0)\sin(\theta_1-\theta_0)\le0,$$
and $\sin(\theta_1-\theta_0)>0$. We distinguish two cases.

\begin{itemize}
\item {\bf Case $b<2a_0$.}
Then $v_1=-c\,\bt(\theta_0)$ with $c:=2a_0-b>0$. Put $\Delta_j:=\theta_j-\theta_0\in\R$; since $\Delta_1\in(0,\pi)$ and $\Delta_{j+1}-\Delta_j=\theta_{j+1}-\theta_j\in[d_*,d^*]\subset(0,\pi)$, the sequence $\Delta_j$ increases to $+\infty$, and
\[
j_*:=\min\{j\ge1:\ \Delta_j>\pi\}
\]
is well defined, with $j_*\ge2$ and $\Delta_{j_*}\in(\pi,2\pi)$ (the previous term is $\le\pi$ and the increment is $<\pi$); in particular $\sin\Delta_{j_*}<0$.

Let $j_1$ be the first index $j\ge1$ with $a_j>0$, if it exists. By Lemma~\ref{lem:bookkeeping} (b), $v_{j+1}=v_j-2a_j\bt(\theta_j)$, so the velocity is constant along any sequence of indices where $a_j=0$:
\[
v_j=v_1=-c\,\bt(\theta_0)\qquad\text{for }1\le j\le j_1
\]
(for all $j\ge1$ if $j_1$ does not exist).

\medskip
We claim that $j_1$ exists and $j_1<j_*$. Otherwise $a_j=0$ for $1\le j<j_*$, hence $v_{j_*}=-c\,\bt(\theta_0)$, and (*) at the index $j_*$ gives
\[
0\ \ge\ v_{j_*}\cdot\bn(\theta_{j_*})=-c\,\sin(\theta_{j_*}-\theta_0)=-c\,\sin\Delta_{j_*}>0,
\]
a contradiction. 
By the claim and the minimality of $j_*$,
\[
0<\Delta_1\le\Delta_{j_1}\le\pi .
\]

Finally, at $j_1$ using  (**) we get $v_{j_1}=-c\,\bt(\theta_0)\parallel\bt(\theta_{j_1})$, {\it i.e.}\ $\theta_{j_1}\equiv\theta_0\pmod\pi$. Since $\Delta_{j_1} \in (0, \pi]$, we conclude that $\theta_{j_1}=\theta_0+\pi$, and $j_1\ge2$ since a single increment is $<\pi$.

\medskip
\item {\bf Case $b=2a_0$.}
Then $v_1=0$, so nothing can be carried forward; but $v_0=b\,\bt(\theta_0)$ with $b=2a_0>0$ by the hypothesis $a_0>0$, so $v_0\ne0$ can be carried backward, similarly to what was done in the previous case.

Put $\Delta_{-m}:=\theta_{-m}-\theta_0\in\R$ for $m\ge0$; since $\Delta_0=0$ and
$$\Delta_{-m}-\Delta_{-(m+1)}=\theta_{-m}-\theta_{-(m+1)}\in[d_*,d^*]\subset(0,\pi),$$
the sequence $\Delta_{-m}$ decreases to $-\infty$, and
\[
m_*:=\min\{m\ge1:\ \Delta_{-m}<-\pi\}
\]
is well defined, with $m_*\ge2$ and $\Delta_{-m_*}\in(-2\pi,-\pi)$ (the previous term is $\ge-\pi$ and the step is $<\pi$); in particular $\sin\Delta_{-m_*}>0$.

Let $k$ be the first index $m\ge1$ with $a_{-m}>0$, if it exists. Lemma~\ref{lem:bookkeeping} (b) read backwards gives 
$$v_{-(m+1)}=v_{-m}+2a_{-(m+1)}\bt(\theta_{-(m+1)}),$$ so the velocity is constant along any sequence of indices where $a_{-m}=0$:
\[
v_{-m}=v_0=b\,\bt(\theta_0)\qquad\text{for }0\le m\le k-1
\]
(for all $m\ge0$ if $k$ does not exist).

\medskip
As before, we claim that $k$ exists and $k\le m_*$. Otherwise $a_{-m}=0$ for $1\le m\le m_*$, hence $v_{-m_*}=b\,\bt(\theta_0)$, and (*) at the index $-m_*$ gives
\[
0\ \ge\ v_{-m_*}\cdot\bn(\theta_{-m_*})=b\,\sin(\theta_{-m_*}-\theta_0)=b\,\sin\Delta_{-m_*}>0,
\]
a contradiction. By the claim and the minimality of $m_*$,
\[
-\pi\le\Delta_{-(k-1)}\le0 .
\]

Finally, at $-k$ a contact point moves, so (**) gives $v_{-k}\parallel\bt(\theta_{-k})$, and then also
$$v_{-(k-1)}=v_{-k}-2a_{-k}\bt(\theta_{-k})\parallel\bt(\theta_{-k}).$$
Since $v_{-(k-1)}=b\,\bt(\theta_0)\ne0$, we get $\theta_{-k}\equiv\theta_0\pmod\pi$. But
$$\Delta_{-k}\in(\Delta_{-(k-1)}-\pi,\Delta_{-(k-1)})\subset(-2\pi,0),$$
hence $\theta_{-k}=\theta_0-\pi$, and $k\ge2$ since $\Delta_{-1}\in(-\pi,0)$. Thus $\tilde g^{\,k}(\theta_0-\pi)=\theta_0$, and by \eqref{eq:g}, we conclude that $\tilde g^{\,k}(\theta_0)=\theta_0+\pi$.
\end{itemize}
This concludes the proof.
\end{proof}

\bigskip

We can now prove the  Rigidity Theorem for central graphs (Theorem \ref{thm:rigiditycentral}).

\medskip

\begin{proof}[Proof of Theorem~\textup{\ref{thm:rigiditycentral}}]
Let
$$\Theta:=\{\theta\in\R:\ \tilde g^{\,j}(\theta)=\theta+\pi\text{ for some }j\ge2\}.$$

\medskip
We first show that $\Theta$ is dense in $\R$. Let $I=[\alpha,\beta]$ be an interval with $\beta-\alpha<L$. The map $\theta_0$ is a homeomorphism and $p$ is injective on intervals of length $<2\pi$, so $P_0=p\circ\theta_0$ is injective on $I$, in particular $P_0(\alpha)\ne P_0(\beta)$. As $P_0$ is Lipschitz,
\[
\int_I a_0(\tau)\dd\tau=\int_I|P_0'(\tau)|\dd\tau\ \ge\ |P_0(\beta)-P_0(\alpha)|>0 ,
\]
so $\{a_0>0\}\cap N\cap I$ has positive measure, in particular is non-empty. By Lemma~\ref{lem:switch}, $\theta_0(\tau)\in\Theta$ for every $\tau$ in this set. Since $\theta_0$ is a homeomorphism, the image of a set meeting every interval is dense; hence $\Theta$ is dense.

\medskip
For $\theta\in\Theta$ let $q(\theta)\ge2$ be an integer with $\tilde g^{\,q(\theta)}(\theta)=\theta+\pi$. If $\theta,\theta'\in\Theta$ with $q=q(\theta)$, $q'=q(\theta')$, put $n:=qq'$. Then
\begin{equation}\label{eq:dimrigiditythm}
\tilde g^{\,n}(\theta)-\theta=q'\pi \quad {\rm and} \quad \tilde g^{\,n}(\theta')-\theta'=q\pi.
\end{equation}
Observe that, since $\tilde g:\R\to\R$ is an increasing homeomorphism with $\tilde g(\theta+\pi)=\tilde g(\theta)+\pi$ (Lemma~\ref{lem:induced}), for all $\theta,\theta'\in\R$ and $n\ge1$ we have
\begin{equation}\label{eq:dimrigiditythm2}
\bigl|\,(\tilde g^n(\theta)-\theta)-(\tilde g^n(\theta')-\theta')\,\bigr|<\pi .
\end{equation}
(Choose $k\in\Z$ with $\theta+k\pi\le\theta'<\theta+(k+1)\pi$; applying the increasing map $\tilde g^n$, which commutes with the translation by $\pi$, gives $\tilde g^n(\theta)+k\pi\le \tilde g^n(\theta')<\tilde g^n(\theta)+(k+1)\pi$, which yields the claim.)

\medskip
Now \eqref{eq:dimrigiditythm} and \eqref{eq:dimrigiditythm2} give $|q-q'|\pi<\pi$, {\it i.e.}\ $q=q'$. Thus $q$ is constant on $\Theta$, and the continuous function $\tilde g^{\,q}(\theta)-\theta-\pi$ vanishes on the dense set $\Theta$, hence identically. This proves \eqref{eq:halfturn}.

\medskip
Finally $$F^q(Z(\theta))=Z(\tilde g^{\,q}\theta)=Z(\theta+\pi)=-Z(\theta)$$ 
by the assumption that $\Gamma$ is central.  Therefore, $F^q|_\Gamma=-\id$ and $T^{2q}=F^{2q}=\id$ on $\Gamma$, which concludes the proof.
\end{proof}

\bigskip

\subsection{General invariant graphs: proof of Theorem~\ref{thm:rigidity}}\label{sec:step2}

Let now $\Gamma$ be an arbitrary invariant graph, with height function $s$. Since $\Gamma$ is a priori not central, the curves $\Gamma$ and $-\Gamma$ may cross each other. The idea is to construct two central invariant graphs out of $\Gamma$ and $-\Gamma$, and apply Theorem \ref{thm:rigiditycentral}.\\

Before proving the theorem, let us recall the following simple lemma, which will play a role in the proof.\\

\medskip

\begin{lemma}\label{lem:frontier}
Let $\phi:Y\to Y$ be a homeomorphism of a topological space and $A\subset Y$ a subset with $\phi(A)=A$. Then
\begin{equation}\label{eq:frontier}
\phi(\Fr A)=\Fr A,\qquad\text{where }\Fr A:=\overline A\setminus\Int A .
\end{equation}
\end{lemma}

\begin{proof}
A homeomorphism maps closures to closures and interiors to interiors, $\phi(\overline A)=\overline{\phi(A)}$ and $\phi(\Int A)=\Int\phi(A)$, because it is a bijection which, together with its inverse, maps open sets to open sets; hence
$$\phi(\Fr A)=\phi(\overline A)\setminus\phi(\Int A)=\overline{\phi(A)}\setminus\Int\phi(A)=\Fr\phi(A)=\Fr A.$$
\end{proof}

\medskip

Let us start by constructing two central invariant graphs.

\medskip

\begin{lemma}\label{lem:symmetrize}
Let $\Gamma$ be an invariant graph with height function $s$, and put
\[
\bar s(\theta):=\max\{s(\theta),s(\theta+\pi)\},\qquad \underline s(\theta):=\min\{s(\theta),s(\theta+\pi)\}.
\]
Then the graphs $\bar\Gamma$ and $\underline\Gamma$ with height functions, respectively, $\bar s$ and $\underline s$, are central invariant graphs and
\[
\Gamma\subset\bar\Gamma\cup\underline\Gamma,\qquad \Gamma\cap\bar\Gamma\ne\emptyset,\qquad \Gamma\cap\underline\Gamma\ne\emptyset .
\]
\end{lemma}

\medskip

\begin{proof}
The functions $\bar s$ and $\underline s$ are continuous, positive and $\pi$-periodic, so it is easy to see that $\bar\Gamma$ and $\underline\Gamma$ are central graphs (Definition~\ref{def:graph}). 
At every $\theta$, the value $s(\theta)$ equals either
$\bar s(\theta)$ or $\underline s(\theta)$, so
\[
\Gamma\subset\bar\Gamma\cup\underline\Gamma.
\]
To prove that $\Gamma$ intersects both graphs, define
\[
d(\theta):=s(\theta)-s(\theta+\pi).
\]
Since $s$ is continuous and $2\pi$-periodic, $d$ is continuous
and satisfies
\[
d(\theta+\pi)
=
s(\theta+\pi)-s(\theta+2\pi)
=
-d(\theta).
\]
Thus, for any $\theta_0$, either $d(\theta_0)=0$ or
$d(\theta_0)$ and $d(\theta_0+\pi)$ have opposite signs.
By the intermediate value theorem, there exists $\theta_*$
such that $d(\theta_*)=0$. In particular, both sets
$\{d\ge0\}$ and $\{d\le0\}$ are nonempty.
At this angle,
\[
s(\theta_*)=s(\theta_*+\pi)
=\bar s(\theta_*)=\underline s(\theta_*),
\]
and therefore
\[
\Psi(\theta_*,s(\theta_*))
\in\Gamma\cap\bar\Gamma\cap\underline\Gamma.
\]
This proves the two required nonempty intersections.

\smallskip It remains to prove invariance. Observe the following facts:
\begin{itemize}
\item   $-\Gamma$ is an invariant graph. In fact, by Proposition~\ref{prop:T-basic} (b), $T(-\Gamma)=-T(\Gamma)=-\Gamma$. Since $-\Psi(\theta,\sigma)=\Psi(\theta+\pi,\sigma)$,
\[
-\Gamma=\{\Psi(\theta+\pi,s(\theta)):\theta\in\Tt\}=\{\Psi(\varphi,s(\varphi+\pi)):\varphi\in\Tt\},
\]
so $-\Gamma$ is the graph with height function $\varphi\mapsto s(\varphi+\pi)$, continuous and positive.

\smallskip
\item The regions below $\Gamma$ and $-\Gamma$ are invariant. We prove it for $\Gamma$, since the other proof is identical. \\
Let $W_{\Gamma}=\Psi(\{(\theta,\sigma):0<\sigma<s(\theta)\})$ be the region below $\Gamma$. The set $(\R^2\setminus K)\setminus\Gamma$ has exactly two connected components, the bounded one $W_{\Gamma}$ and an unbounded one.
Since $T$ is a homeomorphism of $\R^2\setminus K$ with $T(\Gamma)=\Gamma$, it maps $(\R^2\setminus K)\setminus\Gamma$ onto itself and either fixes or exchanges its two components. As $T$ maps bounded sets to bounded sets (Proposition~\ref{prop:T-basic} (d)), $T(W_{\Gamma})$ is bounded, so $T(W_{\Gamma})=W_{\Gamma}$, hence $W_\Gamma$ is invariant.

\smallskip
\item The sets $W_\Gamma\cup W_{-\Gamma}$ and $W_\Gamma\cap W_{-\Gamma}$ are $T$-invariant, being a union and an intersection of $T$-invariant sets; in particular:
\begin{eqnarray*}
W_\Gamma\cup W_{-\Gamma} &=&\Psi\bigl(\{(\theta,\sigma):0<\sigma<\bar s(\theta)\}\bigr)\\
W_\Gamma\cap W_{-\Gamma}&=&\Psi\bigl(\{(\theta,\sigma):0<\sigma<\underline s(\theta)\}\bigr).
\end{eqnarray*}
In fact, $\{0<\sigma<s(\theta)\}\cup\{0<\sigma<s(\theta+\pi)\}=\{0<\sigma<\bar s(\theta)\}$ and likewise for the other. Thus they are the regions below $\bar\Gamma$ and below $\underline\Gamma$.
\end{itemize}

\smallskip

Let us prove that $\bar\Gamma$ is invariant (the proof that $\underline \Gamma$ is invariant is identical). By Lemma~\ref{lem:frontier}, applied to the homeomorphism $T$ of $Y=\R^2\setminus K$ and the invariant set $W_{\bar \Gamma}= W_\Gamma \cup W_{-\Gamma}$, the frontier of $W_{\bar \Gamma}$ in $\R^2\setminus K$ is $T$-invariant. Since $\bar s$ is continuous and positive and 
$$\Psi:\Tt\times(0,\infty)\to\R^2\setminus K$$ is a homeomorphism, this frontier is
\[
\Fr W_{\bar \Gamma} =\Psi\bigl(\Fr\{0<\sigma<\bar s(\theta)\}\bigr)=\Psi\bigl(\{(\theta, \bar s(\theta)):\theta\in\Tt\}\bigr) ,
\]
{\it i.e.}, the graph of $\bar s$.  Hence $T(\bar\Gamma)=\bar\Gamma$.
\end{proof}

\bigskip

To compare $\Gamma$ with $\bar\Gamma$ and $\underline\Gamma$ we use the rotation number of the circle homeomorphism induced by $T$ on an invariant graph. Recall from Definition~\ref{def:lift} that $\tilde T(\theta,s)=(\theta+\delta(\theta,s),s^+)$ with $\delta=d-\pi\in(-\pi,0)$.

\begin{lemma}\label{lem:rotation}
Let $\Gamma$ be an invariant graph with height function $s$ and parametrization $Z(\theta)=\Psi(\theta,s(\theta))$. Then $Z^{-1}\circ T\circ Z$ is an orientation-preserving homeomorphism of $\Tt$, and
\[
\tilde T_\Gamma(\theta):=\theta+\delta\bigl(\theta,s(\theta)\bigr),\qquad\theta\in\R ,
\]
is its unique lift with $\tilde T_\Gamma(\theta)-\theta\in(-\pi,0)$. Its rotation number
\begin{equation}\label{eq:rotnumber}
\rho(\Gamma):=\lim_{n\to\infty}\frac{\tilde T_\Gamma^{\,n}(\theta)-\theta}{2\pi n}\in\bigl[-\tfrac12,0\bigr]
\end{equation}
exists and is independent of $\theta$. Moreover:
\begin{enumerate}[label=\textup{(\alph*)}]
\item if $\Gamma_*$ is another invariant graph and $\Gamma\cap\Gamma_*\neq\emptyset$, then $\rho(\Gamma)=\rho(\Gamma_*)$;
\item if $\Gamma$ is central and satisfies \eqref{eq:halfturn} with the integer $q$, then $\rho(\Gamma)=-\dfrac12+\dfrac1{2q}$.
\end{enumerate}
\end{lemma}

\begin{proof}
The map $\theta\mapsto Z^{-1}(T(Z(\theta)))$ is a homeomorphism of $\Tt$ and, by Proposition~\ref{prop:exact}, it moves every point by $\delta(\theta,s(\theta))\in(-\pi,0)$ (mod~$2\pi$). In particular it has no fixed point, so it preserves orientation (an orientation-reversing circle homeomorphism always has a fixed point). 
The function $\tilde T_\Gamma$ is a continuous lift of it, hence an increasing homeomorphism of $\R$, and it is the unique lift with increments in $(-\pi,0)$. 

\smallskip
The existence of the limit in \eqref{eq:rotnumber} and its independence of $\theta$ follow from Poincar\'e's theorem on rotation numbers (see for example \cite{KH95}), and $-n\pi<\tilde T_\Gamma^{\,n}(\theta)-\theta<0$ gives the range.

\medskip

\noindent (a) Let $x\in\Gamma\cap\Gamma_*$ and $\vartheta_0:=\theta_K(x)$, so that $x=Z(\vartheta_0)=Z_*(\vartheta_0)$, where $Z_*$ parametrizes $\Gamma_*$. 
\smallskip
Put $\vartheta_{n+1}:=\vartheta_n+\delta\bigl(\theta_K(T^nx),s_K(T^nx)\bigr)$. 
By induction on $n$:
$$\tilde T_\Gamma^{\,n}(\vartheta_0)=\tilde T_{\Gamma_*}^{\,n}(\vartheta_0)=\vartheta_n \quad {\rm and} \quad Z(\vartheta_n)=Z_*(\vartheta_n)=T^nx.$$
Indeed, if $Z(\vartheta_n)=T^nx$ then $s(\vartheta_n)=s_K(T^nx)$, so 
$$\tilde T_\Gamma(\vartheta_n)=\vartheta_n+\delta(\theta_K(T^nx),s_K(T^nx))=\vartheta_{n+1}$$ and
$$Z(\vartheta_{n+1})=T(Z(\vartheta_n))=T^{n+1}x,$$ and the same for $\Gamma_*$. Thus, the two lifts have the same orbit of $\vartheta_0$, hence the same rotation number.\\

\noindent (b) For a central graph, $\tilde F$ restricted to the graph of $s$ is 
$$(\theta,s(\theta))\mapsto(\tilde g(\theta),s(\tilde g(\theta))),$$ therefore 
$$\tilde T_\Gamma(\theta)=\theta+d(\theta,s(\theta))-\pi=\tilde g(\theta)-\pi.$$ Since $\tilde g$ commutes with the translation by $\pi$, $\tilde T_\Gamma^{\,n}=\tilde g^{\,n}-n\pi$, and \eqref{eq:halfturn} gives 
$$\tilde T_\Gamma^{\,kq}(\theta)=\theta+k\pi-kq\pi=\theta+k(1-q)\pi \qquad \forall\, k\ge1.$$ Hence
\[
\rho(\Gamma)=\lim_{k\to\infty}\frac{k(1-q)\pi}{2\pi kq}=\frac{1-q}{2q}=-\frac12+\frac1{2q}.
\]
\end{proof}

\medskip

We can now prove Theorem \ref{thm:rigidity}.\\

\begin{proof}[Proof of Theorem~\textup{\ref{thm:rigidity}}]
Let $\Gamma$ be an invariant graph with height function $s$, and let $\bar\Gamma$, $\underline\Gamma$ be the central invariant graphs of Lemma~\ref{lem:symmetrize}. Theorem~\ref{thm:rigiditycentral} applies to both and therefore there are integers $\bar q,\underline q\ge2$ such that \eqref{eq:halfturn} holds for the induced lift of $\bar\Gamma$ with $\bar q$ and for that of $\underline\Gamma$ with $\underline q$.  Therefore:
\begin{equation}\label{eq:halfturn-bars}
\tilde F^{\,\bar q}\bigl(\theta,\bar s(\theta)\bigr)=\bigl(\theta+\pi,\bar s(\theta)\bigr),\qquad
\tilde F^{\,\underline q}\bigl(\theta,\underline s(\theta)\bigr)=\bigl(\theta+\pi,\underline s(\theta)\bigr)\qquad\text{for all }\theta\in\R .
\end{equation}

\smallskip
By Lemma~\ref{lem:rotation} (b): 
$$\rho(\bar\Gamma)=-\frac12+\frac1{2\bar q} \quad {\rm and} \quad \rho(\underline\Gamma)=-\frac12+\frac1{2\underline q}.$$ 
By Lemma~\ref{lem:symmetrize}, $\Gamma$ meets both $\bar\Gamma$ and $\underline\Gamma$, so Lemma~\ref{lem:rotation} (a) gives $\rho(\Gamma)=\rho(\bar\Gamma)=\rho(\underline\Gamma)$, whence $\bar q=\underline q=:q$. 
\medskip
Now let $\theta\in\R$. Either $s(\theta)=\bar s(\theta)$ or $s(\theta)=\underline s(\theta)$ (Lemma~\ref{lem:symmetrize}), and in both cases \eqref{eq:halfturn-bars} yields $\tilde F^{\,q}(\theta,s(\theta))=(\theta+\pi,s(\theta))$, which  is \eqref{eq:halfturn-chart}. Applying $\Psi$ and using $\Psi(\theta+\pi,\sigma)=-\Psi(\theta,\sigma)$ it reads $F^q(x)=-x$ for every $x\in\Gamma$.

\medskip
\noindent
(i) On $\Gamma$, $T^q=(-1)^qF^q=-(-1)^q\,\id$, which is equal to $\id$ for $q$ odd, and to $-\id$ for $q$ even; so  $T^{2q}=\id$ in either case. 

\medskip
\noindent (ii) It follows from Lemma \ref{lem:rotation} that  $\rho(\Gamma)=\rho(\bar\Gamma)=-\frac12+\frac1{2q}$, a rational number (compare also with Remark \ref{rem:noirrational}).

\medskip
\noindent (iii) From $T(\Gamma)=\Gamma$ we get $F(\Gamma)=-T(\Gamma)=-\Gamma$, hence $F^n(\Gamma)=(-1)^n\Gamma$ for all $n$. Since $F^q(\Gamma)=-\Gamma$, we conclude that  $\Gamma=-\Gamma$ when $q$ is even.
\end{proof}

\medskip

\begin{remark}\label{rem:noirrational}
For invariant graphs with irrational rotation number, symmetrization might not be  necessary. 
In fact, let $\Gamma$ be such a graph and observe that  $-\Gamma$ is invariant with the same rotation number.  But, there cannot be two distinct invariant graphs with the same irrational rotation number (see, for example, \cite[Proposition 5.7.1]{Her83}). So $\Gamma=-\Gamma$, hence it is central.
However, since the classical statement in the literature is proven for diffeomorphisms, whereas here $\hat T$ is only a homeomorphism, we decided to provide a self-contained argument.
\end{remark}

\bigskip

\section{The $C^1$-generic absence of invariant graphs consisting of periodic orbits}\label{sec:variational}

The Rigidity Theorem (Theorem~\ref{thm:rigidity}) reduces the problem of excluding invariant graphs to excluding, for each $q\ge2$, an entire family of orbits $x$ with $F^q(x)=-x$ and total lifted angular advance $\pi$, defined for every initial contact angle (we call these {\it half-turn orbits}).\\

In this section we show that, for each $q$ and each compact range of tangent lengths, the existence of at least one initial angle carrying no half-turn orbit is an open and dense condition on $h\in\cH$.  

\medskip

\subsection{Half-turn configurations and their action}\label{sec:action}

The main tool will be the variational characterization of half-turn orbits as critical points of an action functional, which is a classical tool in the study of twist maps (see for example \cite{Gole}). 

\medskip

Let us start by fixing the setting.

\begin{definition}\label{def:config}
Fix $q\ge2$. The \emph{half-turn configuration space} is
\[
\cC_q:=\bigl\{\boldsymbol\theta=(\theta_0,\dots,\theta_{q-1})\in\R^q:\ \theta_0<\theta_1<\dots<\theta_{q-1}<\theta_0+\pi\bigr\},
\]
and indices are extended to all of $\Z$ by $\theta_{i+q}:=\theta_i+\pi$. 

\smallskip
For $h\in\cH$ and $\boldsymbol\theta\in\cC_q$ we set
\begin{align}
s_i^+&:=\frac{h(\theta_{i+1})+h(\theta_i)\cos(\theta_{i+1}-\theta_i)}{\sin(\theta_{i+1}-\theta_i)}+h'(\theta_i),\label{eq:splus}\\
s_i^-&:=\frac{h(\theta_{i-1})+h(\theta_i)\cos(\theta_i-\theta_{i-1})}{\sin(\theta_i-\theta_{i-1})}-h'(\theta_i),\label{eq:sminus}
\end{align}
which are $q$-periodic in $i$, since $h$ is $\pi$-periodic. 

\medskip
Let us denote by $$\cC_q^+(h):=\{\boldsymbol\theta\in\cC_q:\ s_i^\pm>0\text{ for }0\le i<q\},$$ which is an open subset of $\cC_q$. 

\medskip
For $0<\eta-\theta<\pi$ define
\begin{equation}\label{eq:L}
L_h(\theta,\eta):=\frac{2h(\theta)h(\eta)+\bigl(h(\theta)^2+h(\eta)^2\bigr)\cos(\eta-\theta)}{2\sin(\eta-\theta)},
\end{equation}
and the \emph{action}: 
\begin{equation}\label{eq:action}
A_h(\boldsymbol\theta):=\sum_{i=0}^{q-1}L_h(\theta_i,\theta_{i+1}),\qquad\boldsymbol\theta\in\cC_q .
\end{equation}
\end{definition}

\medskip

\begin{remark}
{\bf (i)} Note that $L_h(\theta+\pi,\eta+\pi)=L_h(\theta,\eta)$, so $A_h$ is invariant under $\boldsymbol\theta\mapsto\boldsymbol\theta+(\pi,\dots,\pi)$. Observe that $A_h$ is $C^1$ when $h\in C^1$ (it is as regular as $h$ is).

\smallskip
\noindent{\bf (ii)}  It is a well-known fact that for $C^2$ tables with positive curvature, the outer billiard map is an exact area-preserving twist map, whose generating function is the area of the curvilinear triangle bounded by the two tangent segments from a point to the table and by the arc of $\partial K$ between the points of tangency; see Tabachnikov~\cite{Tab95b,Tab95,Tab05} and Boyland~\cite{Boy96}. 
Since we consider the map $F=-T$, the relevant triangle for a pair $\theta<\eta<\theta+\pi$ is the one denoted by $\Delta$ in Figure~\ref{fig:kite} (b). 

\smallskip
We use instead the signed area $L_h(\theta,\eta)$ of the \emph{kite} $OAyB$ of Figure~\ref{fig:kite} (a), where $A=-h(\theta)\bn(\theta)$ and $B=h(\eta)\bn(\eta)$ are the feet of the perpendiculars from the origin to $\ell_{\theta+\pi}$ and $\ell_\eta$.
One checks that
\[
\operatorname{area}\Delta(\theta,\eta)=L_h(\theta,\eta)+\psi(\theta)-\psi(\eta)-\tfrac12\operatorname{area}(K),
\qquad \psi(\theta):=-\tfrac12\int_0^{\theta}\bigl(h(u)^2-h'(u)^2\bigr)\dd u .
\]
Since $\psi(\theta_0)-\psi(\theta_0+\pi)=\tfrac12\operatorname{area}(K)$, for a configuration $\boldsymbol\theta\in\cC_q$ we obtain
\[
\sum_{i=0}^{q-1}\operatorname{area}\Delta(\theta_i,\theta_{i+1})=A_h(\boldsymbol\theta)-\tfrac{q-1}{2}\operatorname{area}(K).
\]
Thus $A_h$ and the classical action have the same critical points.
\end{remark}

\begin{figure}[ht]
\centering
\begin{tikzpicture}[scale=1.1,>=Stealth]
\begin{scope}
\coordinate (O) at (0,0);
\coordinate (A) at (-0.867,-1.237);
\coordinate (B) at (-1.314,1.103);
\coordinate (y) at (-2.385,-0.174);
\coordinate (mp) at (-1.519,-0.781);
\coordinate (pe) at (-1.786,0.540);
\coordinate (pth) at (1.519,0.781);
\fill[black!12] (O) -- (A) -- (y) -- (B) -- cycle;
\draw[black!60!black,thick] (O) -- (A) -- (y) -- (B) -- cycle;
\draw[thick] (0,0) ellipse (2 and 1.2);
\node at (0.55,-0.35) {$K$};
\draw[gray] ($(mp)!-0.55!(y)$) -- ($(y)!-0.35!(mp)$);
\draw[gray] ($(pe)!-0.75!(y)$) -- ($(y)!-0.35!(pe)$);
\draw[black!60!black] ($(A)+0.16*(0.574,0.819)$) -- ($(A)+0.16*(0.574,0.819)+0.16*(-0.819,0.574)$) -- ($(A)+0.16*(-0.819,0.574)$);
\draw[black!60!black] ($(B)+0.16*(-0.766,0.643)$) -- ($(B)+0.16*(-0.766,0.643)+0.16*(-0.643,-0.766)$) -- ($(B)+0.16*(-0.643,-0.766)$);
\fill (O) circle (1.3pt) node[right=1pt,font=\small] {$O$};
\fill (A) circle (1.3pt) node[below right=-1pt,font=\small] {$A=-a\,\bn(\theta)$};
\fill (B) circle (1.3pt) node[above right=-1pt,font=\small] {$B=b\,\bn(\eta)$};
\fill (y) circle (1.4pt) node[left=2pt,font=\small] {$y$};
\fill (mp) circle (1.2pt) node[below=2pt,font=\small] {$-p(\theta)$};
\fill (pe) circle (1.2pt) node[above=3pt,font=\small] {$p(\eta)$};
\fill (pth) circle (1.2pt) node[above right=-1pt,font=\small] {$p(\theta)$};
\node[font=\small,black!60!black] at (-0.55,-0.5) {$a$};
\node[font=\small,black!60!black] at (-0.5,0.6) {$b$};
\node[font=\small,gray] at (-2.25,-0.9) {$\ell_{\theta+\pi}$};
\node[font=\small,gray] at (-2.55,0.75) {$\ell_\eta$};
\node[font=\small,black!60!black] at (-1.55,-0.15) {$L_h$};
\node[font=\small] at (0,-2) {(a)};
\end{scope}
\begin{scope}[xshift=6.2cm]
\coordinate (O) at (0,0);
\coordinate (y) at (-2.385,-0.174);
\coordinate (mp) at (-1.519,-0.781);
\coordinate (pe) at (-1.786,0.540);
\fill[black!15] (y) -- (pe) arc[start angle=153.28,end angle=220.59,x radius=2,y radius=1.2] -- cycle;
\draw[black!60!black,thick] (O) -- (mp) -- (y) -- (pe) -- cycle;
\draw[thick] (0,0) ellipse (2 and 1.2);
\node at (0.55,-0.35) {$K$};
\draw[gray] ($(mp)!-0.55!(y)$) -- ($(y)!-0.35!(mp)$);
\draw[gray] ($(pe)!-0.75!(y)$) -- ($(y)!-0.35!(pe)$);
\fill (O) circle (1.3pt) node[right=1pt,font=\small] {$O$};
\fill (y) circle (1.4pt) node[left=2pt,font=\small] {$y$};
\fill (mp) circle (1.2pt) node[below=2pt,font=\small] {$-p(\theta)$};
\fill (pe) circle (1.2pt) node[above=3pt,font=\small] {$p(\eta)$};
\node[font=\small] at (-2.05,-0.62) {$s^+$};
\node[font=\small] at (-2.2,0.32) {$s^-$};
\node[font=\small,black!60!black] at (-2.10,-0.14) {$\Delta$};
\node[font=\small] at (0,-2) {(b)};
\end{scope}
\end{tikzpicture}
\caption{\small (a) The supporting-line quadrilateral defining
the signed area $L_h(\theta,\eta)$, with
$a=h(\theta)$ and $b=h(\eta)$.
(b) The corresponding curvilinear triangle $\Delta$,
bounded by the two tangent segments and the 
boundary arc.}
\label{fig:kite}
\end{figure}
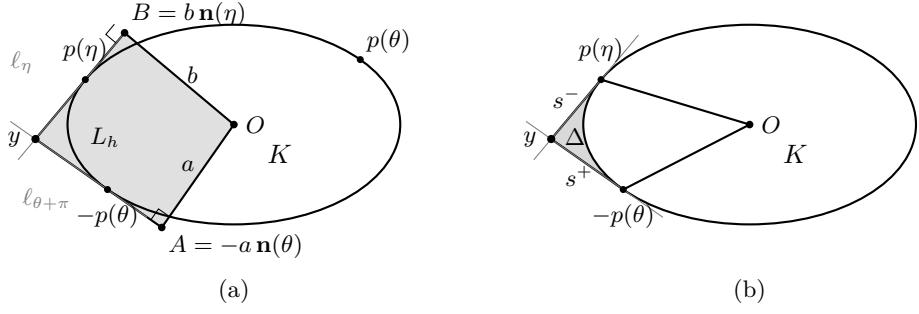

\bigskip

Now, we want to prove the following classical result in the study of twist maps.

\medskip

\begin{proposition}\label{prop:variational}
Let $h\in\cH$ and $\boldsymbol\theta\in\cC_q^+(h)$, and put $x_i:=\Psi(\theta_i,s_i^+)$ for $i\in\Z$ \textup{(}so $x_{i+q}=-x_i$\textup{)}. Then:
$$\nabla A_h(\boldsymbol\theta)=0 \qquad \Longleftrightarrow \qquad \tilde F(\theta_i,s_i^+)=(\theta_{i+1},s_{i+1}^+) \quad \forall\, i.$$
If this is the case:
\begin{equation}\label{eq:return}
F^q(x_0)=-x_0,\qquad \tilde F^{\,q}(\theta_0,s_0^+)=(\theta_0+\pi,s_0^+).
\end{equation}
Conversely, if $(\theta_0,s_0)\in\R\times(0,\infty)$ satisfies $\tilde F^{\,q}(\theta_0,s_0)=(\theta_0+\pi,s_0)$, then the lifted contact angles $\theta_i$ of $\tilde F^{\,i}(\theta_0,s_0)$, $0\le i<q$, form a configuration $\boldsymbol\theta\in\cC_q^+(h)$ with $\nabla A_h(\boldsymbol\theta)=0$ and $s_0^+=s_0$.
\end{proposition}

\bigskip

Let us  prove some auxiliary results. \\

We now introduce two-point tangent-length functions to clarify
the geometric meaning of the quantities in
\eqref{eq:splus} and \eqref{eq:sminus}.
For $0<\eta-\theta<\pi$, define
\begin{align}
s^+(\theta,\eta)
&:=
\frac{h(\eta)+h(\theta)\cos(\eta-\theta)}
     {\sin(\eta-\theta)}
+h'(\theta),
\label{eq:splus-two-point}\\
s^-(\theta,\eta)
&:=
\frac{h(\theta)+h(\eta)\cos(\eta-\theta)}
     {\sin(\eta-\theta)}
-h'(\eta).
\label{eq:sminus-two-point}
\end{align}
For a configuration $\boldsymbol\theta\in\cC_q$, these definitions give
\[
s^+(\theta_i,\theta_{i+1})=s_i^+,
\qquad
s^-(\theta_i,\theta_{i+1})=s_{i+1}^-.
\]
Equivalently, $s_i^-=s^-(\theta_{i-1},\theta_i)$.

\medskip

\begin{lemma}\label{lem:s-geometric}
Let $0<\eta-\theta<\pi$ and let $s^\pm(\theta,\eta)$ be as in  \eqref{eq:splus-two-point} and  \eqref{eq:sminus-two-point}.\\
 If $s^+(\theta,\eta)>0$, then
\[
F\bigl(\Psi(\theta,s^+(\theta,\eta))\bigr)=p(\eta)+s^-(\theta,\eta)\,\bt(\eta).
\]
Consequently, if also $s^-(\theta,\eta)>0$, then $\tilde F(\theta,s^+(\theta,\eta))=(\eta,s^-(\theta,\eta))$. \\
Conversely, if $\tilde F(\theta,s)=(\eta,\sigma)$ with $0<\eta-\theta<\pi$, then $s=s^+(\theta,\eta)$ and $\sigma=s^-(\theta,\eta)$.
\end{lemma}

\medskip

\begin{proof}
Let $d=\eta-\theta$, $s=s^+(\theta,\eta)$, $x=\Psi(\theta,s)$ and $y=F(x)=-p(\theta)+s\bt(\theta)$. Then
\[
y\cdot\bn(\eta)=-h(\theta)\cos d-h'(\theta)\sin d+s\sin d=h(\eta)
\]
by the definition of $s^+$, so $y$ lies on the supporting line with normal $\bn(\eta)$, and
\begin{eqnarray*}
y\cdot\bt(\eta)-h'(\eta)&=&h(\theta)\sin d-h'(\theta)\cos d+s\cos d-h'(\eta)\\
&=&\frac{h(\theta)+h(\eta)\cos d}{\sin d}-h'(\eta)=s^-(\theta,\eta),
\end{eqnarray*}
where we substituted $$s\cos d=\bigl(h(\eta)\cos d+h(\theta)\cos^2d\bigr)/\sin d+h'(\theta)\cos d.$$
 Thus $y=p(\eta)+s^-\bt(\eta)$.
 
 \medskip
  If $s^->0$, then in $(\theta, s)$ coordinates, we deduce  that $y$  has contact angle $\eta$, and $\eta-\theta\in(0,\pi)$ identifies the lift. 
  
  \medskip
For the converse, observe that  \eqref{eq:exact1} says $s=s^+(\theta,\eta)$, and \eqref{eq:exact2}, rewritten as above, says $\sigma=s^-(\theta,\eta)$.
\end{proof}

\medskip

Let us now determine the gradient of the action.\\

\begin{lemma}\label{lem:gradient}
For $h\in\cH$ and $\boldsymbol\theta\in\cC_q$, we have:
\begin{equation}\label{eq:gradient}
\frac{\partial A_h}{\partial\theta_i}(\boldsymbol\theta)=\frac{(s_i^+)^2-(s_i^-)^2}{2},\qquad 0\le i<q .
\end{equation}
\end{lemma}

\medskip

\begin{proof}
To simplify the notation, let $d=\eta-\theta$, $a=h(\theta)$, $b=h(\eta)$, $a'=h'(\theta)$, $b'=h'(\eta)$. Then:
$$L_h=\frac{2ab+(a^2+b^2)\cos d}{2\sin d}.$$ 
Differentiating in $\theta$: 
\begin{eqnarray*}
\partial_1L_h&=&\frac{2a'b+2aa'\cos d+(a^2+b^2)\sin d}{2\sin d}+\frac{(2ab+(a^2+b^2)\cos d)\cos d}{2\sin^2d}\\
&=&\frac{a'(b+a\cos d)}{\sin d}+\frac{a^2+b^2+2ab\cos d}{2\sin^2d}.
\end{eqnarray*}
On the other hand, with $s^+=(b+a\cos d)/\sin d+a'$,
\[
\frac{(s^+)^2+a^2-a'^2}{2}=\frac{(b+a\cos d)^2}{2\sin^2d}+\frac{a'(b+a\cos d)}{\sin d}+\frac{a^2}{2}
=\frac{a'(b+a\cos d)}{\sin d}+\frac{b^2+2ab\cos d+a^2}{2\sin^2d}.
\]
Hence,
$$\partial_1L_h(\theta,\eta)=\tfrac12\bigl((s^+)^2 + h(\theta)^2 - h'(\theta)^2\bigr).$$ 
Similarly, we obtain: 
$$\partial_2L_h(\theta,\eta)=\tfrac12\bigl(h'(\eta)^2-h(\eta)^2-(s^-)^2\bigr)$$ 
with $s^-=(a+b\cos d)/\sin d-b'$. 

\smallskip
In the  sum \eqref{eq:action} the variable $\theta_i$ appears in $L_h(\theta_i,\theta_{i+1})$ and in $L_h(\theta_{i-1},\theta_i)$ (for $i=0$ the latter is $L_h(\theta_{q-1},\theta_q)$ with $\theta_q=\theta_0+\pi$), 
so  the terms $\pm\tfrac12(h(\theta_i)^2-h'(\theta_i)^2)$ cancel and \eqref{eq:gradient} follows.
\end{proof}

\bigskip

We can now prove the variational principle stated in Proposition \ref{prop:variational}.

\begin{proof}[Proof of Proposition \ref{prop:variational}]
By Lemma~\ref{lem:gradient} and positivity of $s_i^\pm$, we get:
$$\nabla A_h(\boldsymbol\theta)=0 \qquad \Longleftrightarrow \qquad s_i^+=s_i^- \quad \forall\, i.$$
 Lemma~\ref{lem:s-geometric} applied to $(\theta,\eta)=(\theta_i,\theta_{i+1})$ gives 
 $$\tilde F(\theta_i,s_i^+)=(\theta_{i+1},s^-(\theta_i,\theta_{i+1}))=(\theta_{i+1},s_{i+1}^-).$$ 
 So if $s_{i+1}^-=s_{i+1}^+$ for all $i$, then $\tilde F(\theta_i,s_i^+)=(\theta_{i+1},s^+_{i+1})$.
 
 \smallskip
 For the other implication, the above relation forces $s_{i+1}^+=s_{i+1}^-$. Iterating, we obtain:
 $$\tilde F^{\,q}(\theta_0,s_0^+)=(\theta_q,s_q^+)=(\theta_0+\pi,s_0^+) \quad {\rm and} \quad F^q(x_0)=\Psi(\theta_0+\pi,s_0^+)=-x_0.$$

\smallskip
For the converse, let $(\theta_i,s_i):=\tilde F^{\,i}(\theta_0,s_0)$. The increments $\theta_{i+1}-\theta_i$ lie in $(0,\pi)$ and sum to $\pi$ over $0\le i<q$, so $\boldsymbol\theta\in\cC_q$, and the $q$-periodic extension is consistent with $\theta_q=\theta_0+\pi$, $s_q=s_0$. By the converse part of Lemma~\ref{lem:s-geometric} 
$$s_i=s^+(\theta_i,\theta_{i+1})=s_i^+ \quad {\rm and} \quad s_{i+1}=s^-(\theta_i,\theta_{i+1})=s_{i+1}^-.$$ Hence $s_i^+=s_i^-=s_i>0$ for all $i$, {\it i.e.},\ $\boldsymbol\theta\in\cC_q^+(h)$ and $\nabla A_h(\boldsymbol\theta)=0$.
\end{proof}

\bigskip
\bigskip

\subsection{Continuity in the table}\label{sec:continuitytable}

Let us briefly discuss how the map $F$ depends on the table, since this will be important in order to determine the residual set $\cR$.

\medskip

\begin{proposition}\label{prop:continuity}
The map $(h,\theta,s)\mapsto\tilde F_h(\theta,s)$ is continuous on $\cH\times\R\times(0,\infty)$, where $\cH$ is endowed with the $C^1$ topology. Consequently, so are all iterates $(h,\theta,s)\mapsto\tilde F_h^{\,n}(\theta,s)$, for every fixed $n\ge1$.
\end{proposition}

\medskip

\begin{proof}
We write $\theta_h:=\theta_{K_h}$ and $s_h:=s_{K_h}$ for the contact angle and the tangent length of the table $K_h$, and $p_h=h\bn+h'\bt$ for its boundary parametrization. \\

Let us prove that $(h,\theta,s)\mapsto\tilde F_h(\theta,s)$ is continuous on $\cH\times\R\times(0,\infty)$.

\begin{itemize}
\item[(i)]  We start by proving that $(h,y)\mapsto(\theta_h(y),s_h(y))\in\Tt\times(0,\infty)$ is continuous on $$\cO:=\{(h,y)\in\cH\times\R^2:\ y\notin K_h\}.$$ 
Observe that $\cO$ is open, in fact: 
$$y\notin K_h \qquad \Longleftrightarrow \qquad m(h,y):=\max_{\theta\in\Tt}\bigl(y\cdot\bn(\theta)-h(\theta)\bigr)>0,$$ and $m$ is continuous on $\cH\times\R^2$, since $|m(h,y)-m(h',y')|\le|y-y'|+\|h-h'\|_\infty$.

\smallskip
\noindent To prove continuity, 
let $(h_n,y_n)\to(h,y)$ in $\cO$, and put $\varphi_n:=\theta_{h_n}(y_n)\in\Tt$, $\sigma_n:=s_{h_n}(y_n)>0$. By Definition~\ref{def:contact},
\begin{equation}\label{eq:cont-1}
y_n\cdot\bn(\varphi_n)=h_n(\varphi_n),\qquad \sigma_n=y_n\cdot\bt(\varphi_n)-h_n'(\varphi_n).
\end{equation}
Let $(\varphi_{n_k})$ be any subsequence converging to some $\varphi\in\Tt$. Since $h_{n_k}\to h$ and $h_{n_k}'\to h'$ uniformly and $h,h'$ are continuous,
\[
h_{n_k}(\varphi_{n_k})\to h(\varphi),\qquad h_{n_k}'(\varphi_{n_k})\to h'(\varphi).
\]
Passing to the limit in \eqref{eq:cont-1} along $n_k$ gives
\[
y\cdot\bn(\varphi)=h(\varphi),\qquad \sigma_{n_k}\to\sigma:=y\cdot\bt(\varphi)-h'(\varphi)\ge0 .
\]
The first relation says that $y$ lies on the supporting line of $K_h$ with normal $\bn(\varphi)$, and the second that $y=p_h(\varphi)+\sigma\bt(\varphi)$. If $\sigma$ were $0$, we would have $y=p_h(\varphi)\in K_h$, contradicting $(h,y)\in\cO$; hence $\sigma>0$. Then, $y=\Psi_h(\varphi,\sigma)$ with $\sigma>0$, and Lemma~\ref{lem:chart} identifies $\varphi=\theta_h(y)$ and $\sigma=s_h(y)$.
Since this holds for every converging subsequence, this proves  continuity.

\medskip
\item[(ii)] We now prove that the map $(h,\theta,s)\mapsto y_h(\theta,s):=F_h(\Psi_h(\theta,s))$ is continuous from $\cH\times\R\times(0,\infty)$ to $\R^2$, with values in $\{y:(h,y)\in\cO\}$.

\smallskip
In fact, by \eqref{eq:F}, 
$$y_h(\theta,s)=\Psi_h(\theta,s)-2p_h(\theta)=-p_h(\theta)+s\,\bt(\theta).$$ The map 
$$(h,\theta)\mapsto p_h(\theta)=h(\theta)\bn(\theta)+h'(\theta)\bt(\theta)$$ is jointly continuous on $\cH\times\R$ by the estimate used in item (i), and $(\theta,s)\mapsto s\,\bt(\theta)$ is continuous; hence $y_h(\theta,s)$ is continuous in $(h,\theta,s)$. Its values lie outside $K_h$ because $F_h$ maps $\R^2\setminus K_h$ to itself (Proposition~\ref{prop:T-basic} (a)).

\medskip 
\item[(iii)] The map $(h,\theta,s)\mapsto\bigl(\theta_h(y_h(\theta,s)),\,s_h(y_h(\theta,s))\bigr)\in\Tt\times(0,\infty)$ is continuous.

\noindent In fact, this is the composition of the continuous map $(h,\theta,s)\mapsto(h,y_h(\theta,s))\in\cO$ of item (ii) with the continuous map of item (i). By Proposition~\ref{prop:exact}, the first component is the contact angle $\theta^+$ of $F_h(\Psi_h(\theta,s))$ (mod~$2\pi$), and the second is $s^+$.

\medskip
\item[(iv)] The lift
\[
\tilde F_h(\theta,s)=\bigl(\theta+d_h(\theta,s),\ s_h(y_h(\theta,s))\bigr)
\]
is continuous on $\cH\times\R\times(0,\infty)$, with values in $\R\times(0,\infty)$.

\smallskip 
\noindent In fact, by Proposition~\ref{prop:exact}, $\theta^+-\theta\in(0,\pi)\pmod{2\pi}$, {\it i.e.},\ the continuous $\Tt$-valued function $(h,\theta,s)\mapsto\theta_h(y_h(\theta,s))-\theta$ takes its values in the open arc $(0,\pi)\subset\Tt$. The projection $\R\to\Tt$ restricts to a homeomorphism of the interval $(0,\pi)\subset\R$ onto $(0,\pi)$; composing with its inverse, the representative $d_h(\theta,s)\in(0,\pi)\subset\R$ of $\theta^+-\theta$ is a continuous function of $(h,\theta,s)$. 

\end{itemize}

\smallskip
Therefore, the map $\Phi:(h,\theta,s)\mapsto(h,\tilde F_h(\theta,s))$ is a continuous map of $\cH\times\R\times(0,\infty)$ to itself. Observe that $(h,\tilde F_h^{\,n}(\theta,s))=\Phi^n(h,\theta,s)$, so each iterate is continuous as a composition of continuous maps.
\end{proof}

\bigskip

\bigskip

\subsection{The open \& dense conditions}

We want now to show that for a given $q\geq 2$, typically an outer billiard map coming from a support function in $\cH$ does not admit an invariant curve consisting of orbits all satisfying the {\it half-turn return} condition \eqref{eq:halfturn-chart} with period $q$. This follows a classical strategy for problems of this kind, which we adapt below to our setting.

\medskip
Let us start with the following definition.\\

\begin{definition}\label{def:U}
For integers $q,m\ge2$ let $\cU_{q,m}$ be the set of $h\in\cH$ for which there is an angle $\theta_*\in\R$ such that
\begin{equation}\label{eq:U}
\tilde F_h^{\,q}(\theta_*,s)\ne(\theta_*+\pi,s)\qquad\text{for every }s\in[1/m,m].
\end{equation}
\end{definition}

\bigskip
The main goal of this section is to prove the following theorem.\\

\begin{theorem}\label{thm:dense}
For all $q,m\ge2$ the set $\cU_{q,m}$ is open and dense in $\cH$.
\end{theorem}

\medskip
One can easily prove openness.\\

\begin{proposition}\label{prop:open}
$\cU_{q,m}$ is open in $\cH$.
\end{proposition}

\begin{proof}
Let $h\in\cU_{q,m}$ with angle $\theta_*$. The continuous function $$s\mapsto|\tilde F_h^{\,q}(\theta_*,s)-(\theta_*+\pi,s)|$$ has a positive minimum $\mu$ on the compact interval $[1/m,m]$. \\
By Proposition~\ref{prop:continuity} and compactness, there is a $C^1$-neighbourhood of $h$ on which $$\sup_{s\in[1/m,m]}|\tilde F_{h'}^{\,q}(\theta_*,s)-\tilde F_h^{\,q}(\theta_*,s)|<\mu.$$ Every $h'$ in this neighbourhood belongs to $\cU_{q,m}$ with the same $\theta_*$.
\end{proof}

\medskip

In order to prove density, let us first prove some auxiliary results.\\

Let us observe that smooth tables are dense in $\cH$.

\begin{lemma}\label{lem:smooth-dense}
The set of $h\in\cH\cap C^\infty(\Tp)$ with $h+h''>0$  is dense in $\cH$.
\end{lemma}

\begin{proof}
Let $h\in\cH$ and let $(\phi_\varepsilon)_\varepsilon$ be a standard family of smooth non-negative $\pi$-periodic mollifiers with $\int\phi_\varepsilon=1$ and $\supp\phi_\varepsilon\to\{0\}$. Then $h_\varepsilon:=h*\phi_\varepsilon$ is smooth, $h_\varepsilon\to h$ and $h_\varepsilon'=h'*\phi_\varepsilon\to h'$ uniformly, and $h_\varepsilon+h_\varepsilon''=\rho_h*\phi_\varepsilon\ge0$ is a smooth function. For $c>0$, $h_\varepsilon+c$ satisfies 
\[
\rho_{h_\varepsilon+c}
=
\bigl(\rho_h*\phi_\varepsilon+c\bigr)\,\dd\theta
\ge c\,\dd\theta,
\]
hence it belongs to $\cH$, and $\|h_\varepsilon+c-h\|_{C^1}\to0$ as $\varepsilon,c\to0$.
\end{proof}

\bigskip

A more subtle point is to show that, after an arbitrarily small perturbation of $h$, the zeros of $\nabla A_h$ in a given compact set are \emph{nondegenerate}, {\it i.e.}\ the Hessian of $A_h$ is invertible there. 
\medskip

\begin{lemma}\label{lem:transversality}
Let $h_0\in\cH\cap C^\infty$ with $h_0+h_0''>0$, let $\bar N\subset\cC_q^+(h_0)$ be compact, and let $\varepsilon>0$. There is $h\in\cH\cap C^\infty$ with $h+h''>0$ and $\|h-h_0\|_{C^2}<\varepsilon$ such that $\bar N\subset\cC_q^+(h)$ and every zero of $\nabla A_h$ in $\bar N$ is nondegenerate \textup{(}i.e., the Hessian of $A_h$ is invertible there\textup{)}. In particular $\nabla A_h$ has only finitely many zeros in $\bar N$.
\end{lemma}

\begin{proof}
Let $Z_0:=\{\boldsymbol\theta\in\bar N:\nabla A_{h_0}(\boldsymbol\theta)=0\}$, a compact set. 

If $Z_0=\emptyset$, we may take $h=h_0$: there are no zeros of $\nabla A_{h_0}$ in $\bar N$, so the nondegeneracy condition is vacuously satisfied, and all the other conclusions of the lemma hold. 

\smallskip
Henceforth assume $Z_0\ne\emptyset$.
For $c\in Z_0$ the angles $c_0,\dots,c_{q-1}$ are pairwise distinct modulo $\pi$, so there are smooth $\pi$-periodic functions $\phi_{c,0},\dots,\phi_{c,q-1}$ with
\[
\phi_{c,i}(c_j)=0,\qquad\phi_{c,i}'(c_j)=\delta_{ij}\qquad(0\le i,j<q).
\]
Consider, for $a\in\R^q$, the family $h_{c,a}:=h_0+\sum_i a_i\phi_{c,i}$ and $E_c(a,\boldsymbol\theta):=\nabla A_{h_{c,a}}(\boldsymbol\theta)$, a smooth map on a neighbourhood of $(0,c)$ in $\R^q\times\cC_q$. \\
By Lemma~\ref{lem:gradient}, $$\frac{\partial A_h}{\partial\theta_i}=\tfrac12((s_i^+)^2-(s_i^-)^2),$$ where $s_i^\pm$ depend on $h$ only through $h(\theta_{i-1}),h(\theta_i),h(\theta_{i+1})$ and $h'(\theta_i)$. 
 
\noindent At $(a,\boldsymbol\theta)=(0,c)$ the perturbation $\phi_{c,i}$ leaves all values $h(c_j)$ and all derivatives $h'(c_j)$, $j\ne i$, unchanged, so
\[
\frac{\partial}{\partial a_i}\Bigl(\frac{\partial A_{h_{c,a}}}{\partial\theta_j}\Bigr)\Big|_{(0,c)}=\delta_{ij}\,\tfrac12\bigl(2s_i^+\cdot1-2s_i^-\cdot(-1)\bigr)=\delta_{ij}\,(s_i^++s_i^-)=2\delta_{ij}\,s_i ,
\]
where $s_i=s_i^+=s_i^->0$ at the critical configuration $c$.
Thus $D_aE_c(0,c)$ is invertible. By continuity, there is
an open neighbourhood $W_c$ of $c$, contained in
$\cC_q^+(h_0)$, such that $D_aE_c(0,\boldsymbol\theta)$
is invertible for every $\boldsymbol\theta\in W_c$.
Choose an open neighbourhood $V_c$ of $c$ whose closure
is compact and satisfies
\[
\overline{V_c}\subset W_c.
\]

By compactness of the nonempty set $Z_0$, finitely many
of these smaller neighbourhoods,
$V_{c_1},\ldots,V_{c_K}$, cover $Z_0$.
For $\mathbf a=(a^1,\ldots,a^K)\in\R^{qK}$, define
\[
h_{\mathbf a}
:=
h_0+\sum_{k=1}^K\sum_{i=0}^{q-1}
a_i^k\,\phi_{c_k,i},
\qquad
E(\mathbf a,\boldsymbol\theta)
:=
\nabla A_{h_{\mathbf a}}(\boldsymbol\theta),
\]
and put
\[
V:=\bigcup_{k=1}^K V_{c_k}.
\]

At $\mathbf a=0$, the derivative with respect to the
$k$-th block of parameters satisfies
\[
D_{a^k}E(0,\boldsymbol\theta)
=
D_aE_{c_k}(0,\boldsymbol\theta).
\]
It is therefore invertible on $\overline{V_{c_k}}$.
Since this closure is compact,
\[
\gamma_k:=
\min_{\boldsymbol\theta\in\overline{V_{c_k}}}
\left|\det D_{a^k}E(0,\boldsymbol\theta)\right|>0.
\]
By continuity and compactness, and since there are only
finitely many indices $k$, there exists $\delta>0$ such
that
\[
\left|\det D_{a^k}E(\mathbf a,\boldsymbol\theta)\right|
\ge\frac{\gamma_k}{2}
\]
whenever $|\mathbf a|<\delta$ and
$\boldsymbol\theta\in\overline{V_{c_k}}$,
for every $1\le k\le K$.

Decreasing $\delta$ if necessary, we may also arrange that,
for every $|\mathbf a|<\delta$,
\[
\|h_{\mathbf a}-h_0\|_{C^2}<\varepsilon,
\qquad
h_{\mathbf a}>0,
\qquad
h_{\mathbf a}+h_{\mathbf a}''>0,
\qquad
\bar N\subset\cC_q^+(h_{\mathbf a}).
\]
These conditions follow from the corresponding strict
inequalities at $\mathbf a=0$ and compactness.
In particular, $h_{\mathbf a}\in\cH$.

For every $\boldsymbol\theta\in V$, at least one block
$D_{a^k}E(\mathbf a,\boldsymbol\theta)$ is invertible.
Consequently $D_{\mathbf a}E$ has rank $q$ on
$B(0,\delta)\times V$, so $E$ is a submersion there.

\smallskip
We claim that there is $\delta'\in(0,\delta]$ such that for $|\mathbf a|<\delta'$ all zeros of $E(\mathbf a,\cdot)$ in $\bar N$ lie in $V$. Otherwise there are $\mathbf a_n\to0$ and $\boldsymbol\theta_n\in\bar N\setminus V$ with $E(\mathbf a_n,\boldsymbol\theta_n)=0$; a limit point $\boldsymbol\theta\in\bar N\setminus V$ satisfies $\nabla A_{h_0}(\boldsymbol\theta)=0$ by continuity of $E$, {\it i.e.},  $\boldsymbol\theta\in Z_0\subset V$, yielding a contradiction.

\smallskip
Now $\mathcal M:=E^{-1}(0)\cap\bigl(B(0,\delta')\times V\bigr)$ is a smooth submanifold of dimension $qK$, and 
\begin{eqnarray*}
\pi:\mathcal M &\longrightarrow& B(0,\delta')\\
(\mathbf a,\boldsymbol\theta) &\longmapsto& \mathbf a
\end{eqnarray*} 
is smooth. By Sard's theorem  almost every $\mathbf a\in B(0,\delta')$ is a regular value of $\pi$. Fix such an $\mathbf a$ and let $(\mathbf a,\boldsymbol\theta)\in\mathcal M$. The tangent space 
$$T_{(\mathbf a,\boldsymbol\theta)}\mathcal M=\{(\alpha,\beta):D_{\mathbf a}E\,\alpha+D_{\boldsymbol\theta}E\,\beta=0\}$$ projects onto $\R^{qK}$, which means that for every $\alpha$ there is $\beta$ with $D_{\boldsymbol\theta}E\,\beta=-D_{\mathbf a}E\,\alpha$; since $D_{\mathbf a}E$ is onto $\R^q$, so is $D_{\boldsymbol\theta}E(\mathbf a,\boldsymbol\theta)$, which is therefore an invertible $q\times q$ matrix, coinciding with the Hessian of $A_{h_{\mathbf a}}$ at $\boldsymbol\theta$. 

Thus every zero of $\nabla A_{h_{\mathbf a}}$ in $\bar N$ (all of which lie in $V$) is nondegenerate. Nondegenerate zeros are isolated, so there are finitely many in the compact set $\bar N$. This completes the proof. 
\end{proof}

\medskip
We can now prove the main result of this section (Theorem \ref{thm:dense}).\\

\begin{proof}[Proof of Theorem \ref{thm:dense}]
We only need to prove density, since openness has been proven in Proposition~\ref{prop:open}. 

\smallskip
To prove density, let $h\in\cH$ and let $\mathcal N$ be a $C^1$-neighbourhood of $h$. By Lemma~\ref{lem:smooth-dense} we may pick $h_0\in\mathcal N\cap C^\infty$ with $h_0+h_0''>0$.

\medskip
Let us consider the set  $$Y:=\{(\theta,s)\in[0,\pi]\times[1/m,m]:\ \tilde F_{h_0}^{\,q}(\theta,s)=(\theta+\pi,s)\},$$ which is compact. 

\smallskip

If $Y=\emptyset$, then
\[
\tilde F_{h_0}^{\,q}(0,s)\ne(\pi,s)
\qquad\text{for every }s\in[1/m,m].
\]
Thus $h_0\in\cU_{q,m}$, with $\theta_*=0$, and since
$h_0\in\mathcal N$, the required density conclusion
already holds. 

\smallskip
We may therefore assume $Y\ne\emptyset$.

\medskip
To simplify the exposition, in the rest of the proof, we will call  
$(\theta,s)\in[0,\pi]\times[1/m,m]$  a  {\it half-returning state for $h$} if $\tilde F_{h}^{\,q}(\theta,s)=(\theta+\pi,s)$.

\medskip
For every $h'\in\cH$ and every half-returning state
$(\theta,s)$ for $h'$, write
\[
(\theta_i^{h'},s_i^{h'})
:=
\tilde F_{h'}^{\,i}(\theta,s),
\qquad 0\le i\le q,
\]
and define
\[
\kappa_{h'}(\theta,s)
:=
(\theta_0^{h'},\ldots,\theta_{q-1}^{h'}).
\]
By Proposition~\ref{prop:variational},
\[
\kappa_{h'}(\theta,s)\in\cC_q^+(h'),
\qquad
\nabla A_{h'}\bigl(\kappa_{h'}(\theta,s)\bigr)=0.
\]
By Proposition~\ref{prop:continuity}, the map
$(h',\theta,s)\mapsto\kappa_{h'}(\theta,s)$ is jointly
continuous on the set of triples for which $(\theta,s)$
is a half-returning state for $h'$.

In particular, every $(\theta,s)\in Y$ is a half-returning
state for $h_0$. Writing $\kappa:=\kappa_{h_0}$, we obtain
that $\kappa(Y)$ is a compact subset of the open set
$\cC_q^+(h_0)$. Choose an open set $N$ and a compact set
$\bar N$ such that
\[
\kappa(Y)\subset N\subset\bar N\subset\cC_q^+(h_0).
\]

\medskip

 We claim that there is a neighbourhood $\mathcal N_1\subset\mathcal N$ of $h_0$ such that  for all $h'\in\mathcal N_1$ every half-returning state $(\theta,s)\in[0,\pi]\times[1/m,m]$ of $h'$  has configuration $\kappa_{h'}(\theta,s)\in N$. Otherwise there would exist $h_n\to h_0$ and half-returning states $(\theta_n,s_n)$ for $h_n$ with $\kappa_{h_n}(\theta_n,s_n)\notin N$; passing to a subsequence $(\theta_n,s_n)\to(\theta,s)\in[0,\pi]\times[1/m,m]$, 
 Proposition~\ref{prop:continuity} would imply that
$(\theta,s)\in Y$ and
\[
\kappa_{h_n}(\theta_n,s_n)
\longrightarrow
\kappa_{h_0}(\theta,s)=\kappa(\theta,s).
\]
Since $N$ is open and
$\kappa_{h_n}(\theta_n,s_n)\notin N$ for every $n$,
the limit $\kappa(\theta,s)$ also lies outside $N$.
This contradicts $(\theta,s)\in Y$ and $\kappa(Y)\subset N$.

\medskip

Now, apply Lemma~\ref{lem:transversality} to $h_0$, $\bar N$ and an $\varepsilon$ so small that the resulting $h_1$ lies in $\mathcal N_1$ (recall $\|\cdot\|_{C^1}\le\|\cdot\|_{C^2}$). Every half-returning state $(\theta,s)\in[0,\pi]\times[1/m,m]$ of $h_1$ has, by what we discussed above and Proposition~\ref{prop:variational}, a configuration which is a zero of $\nabla A_{h_1}$ lying in $N\subset\bar N$. There are finitely many such zeros, and each determines its initial state $(\theta_0,s_0^+)$. Hence the set $\Theta_1$ of initial angles $\theta\in[0,\pi]$ of half-returning states of $h_1$ with $s\in[1/m,m]$ is finite. Since $\tilde F_{h_1}$ commutes with $\theta\mapsto\theta+\pi$, the half-returning initial angles in $\R$ are $\Theta_1+\pi\Z$. Any $\theta_*\in\R\setminus(\Theta_1+\pi\Z)$ satisfies \eqref{eq:U}, so $h_1\in\cU_{q,m}\cap\mathcal N$. This completes the proof.
\end{proof}

\bigskip

\subsection{The construction of the residual set $\cR$}\label{sec:baire}

In this section we want to construct the residual set $\cR\subset\cH$ of the Main Theorem.\\

The main result of this section is the following theorem.

\medskip
\begin{theorem}\label{thm:no-graphs}
There is a residual subset $\cR\subset\cH$, such that for every $h\in\cR$ the measure $\rho_h$ is purely singular and the associated outer billiard map does not admit any  invariant graph.
In particular $\cR$ meets every $C^1$-neighbourhood of  the support functions of disks or ellipses.
\end{theorem}

\medskip

\begin{remark}\label{re:proofmthmiii}
Observe that  Theorem \ref{thm:no-graphs} implies parts (i) and (ii) of the Main Theorem. 
\end{remark}

\medskip

\begin{remark}
As we shall see in  Appendix \ref{app:Radon}, there exists $h\in \cH$ with $\rho_h$ purely singular such that the corresponding outer billiard map admits an invariant graph. Therefore, $\cR$ will be a proper subset of $\cS$, where $\cS:=\{h\in\cH:\ \rho_h\perp\dd\theta\}$.
\end{remark}

\medskip

Before proving this theorem, let us prove some auxiliary results that will allow us to find a good candidate for the set $\cR$ (see Definition \ref{def:R}).

\medskip

\begin{lemma}\label{lem:Baire}
$\cH$ is a $G_\delta$ subset of the Banach space $C^1(\Tp)$. Hence $\cH$, with the relative $C^1$ topology, is completely metrizable and in particular a Baire space. Moreover, on $$\cP:=\{h\in C^1(\Tp):\rho_h\ge0\}$$ the map $h\mapsto\rho_h$ is continuous into the space of finite positive measures with the weak topology.
\end{lemma}

\medskip
\begin{proof}
First of all, the cone $\cP$ is closed; in fact,  if $h_n\to h$ in $C^1$ and $\rho_{h_n}\ge0$, then for every smooth $\phi\ge0$,  we have
$$\langle\rho_h,\phi\rangle=\int(h\phi-h'\phi')\dd\theta=\lim\int(h_n\phi-h_n'\phi')\dd\theta\ge0,$$ so $\rho_h$ is a positive distribution, hence a positive measure.\\
 On $\cP$, $\rho_h$ has total mass $\int_{\Tp}h\dd\theta$ (because $\int Dh'=0$), and $h_n\to h$ in $C^1$ implies 
 $$\langle\rho_{h_n},\phi\rangle\to\langle\rho_h,\phi\rangle$$ for smooth $\phi$. Then, approximating a continuous $\phi$ uniformly by smooth $\phi_k$ and using the uniform bound on the masses proves weak convergence, and therefore continuity.\\
  For an open set $U\subset\Tp$ the map $h\mapsto\rho_h(U)$ is then lower semicontinuous on $\cP$, and for a closed set $C$ the map $h\mapsto\rho_h(C)$ is upper semicontinuous (see, 
  for example,  \cite[Theorem~8.2.3]{Bog07}). Hence, if $\mathfrak U$ is a countable base of open arcs of $\Tp$, then
\[
\cH=\cP\ \cap\ \{h:\min h>0\}\ \cap\ \bigcap_{U\in\mathfrak U}\{h\in\cP:\rho_h(U)>0\}.
\]
Therefore, $\cH$ is the intersection of a closed set, an open set and countably many relatively open subsets of  $\cP$, hence it is a $G_\delta$ in $C^1(\Tp)$. A $G_\delta$ subset of a complete metric space is completely metrizable (Alexandrov's theorem), hence Baire.
\end{proof}

\medskip

Let us now prove that the functions $h$ with purely singular radius of curvature measure are generic in $\cH$.\\

\begin{lemma}\label{lem:S}
The set $\cS:=\{h\in\cH:\ \rho_h\perp\dd\theta\}$ corresponding to tables with purely singular radius of curvature measure is a dense $G_\delta$ subset of $\cH$.
\end{lemma}

\begin{proof}
For $n\ge1$ let
\[
\cS_n:=\bigl\{h\in\cH:\ \exists\,U\subset\Tp\text{ open with }|U|<1/n\text{ and }\rho_h(\Tp\setminus U)<1/n\bigr\}.
\]

Let us first show that {$\cS_n$ is open.} In fact, for fixed $U$, the set $$\{h\in\cH:\rho_h(\Tp\setminus U)<1/n\}$$ is open by the upper semicontinuity in Lemma~\ref{lem:Baire}. It suffices to observe that $\cS_n$ is the union of these sets over $U$.

\medskip
In order to prove that $\cS$ is a $G_\delta$ set, we  will prove that {$\cS=\bigcap_n\cS_n$.}

 If $\rho_h$ is purely singular, it is supported on  a Borel null set $E$; choosing an open $U\supset E$ with $|U|<1/n$ gives $\rho_h(\Tp\setminus U)\le\rho_h(\Tp\setminus E)=0$, so $h\in\cS_n$ for all $n$.

Conversely, let $h\in\bigcap_n\cS_n$. 
Choose $U_j$ for $n=2^j$, so that 
$$\sum_j|U_j|<\infty , \qquad  {\rm and} \qquad \sum_j\rho_h(\Tp\setminus U_j)<\infty.$$ The set $E:=\limsup_jU_j=\bigcap_k\bigcup_{j\ge k}U_j$ has $|E|\le\sum_{j\ge k}|U_j|$ for every $k$, hence $|E|=0$, while 
$$\rho_h(\Tp\setminus E)=\rho_h\bigl(\bigcup_k\bigcap_{j\ge k}(\Tp\setminus U_j)\bigr)\le\sum_k\inf_{j\ge k}\rho_h(\Tp\setminus U_j)=0.$$ So $\rho_h$ is supported on a null set.

\medskip
Finally, let us prove that {$\cS$ is dense.} Let $h\in\cH$ and $\varepsilon>0$. By Lemma~\ref{lem:smooth-dense} there is a smooth $h_0\in\cH$ with $\|h_0-h\|_{C^1}<\varepsilon/2$ and $\rho_0:=\rho_{h_0}=g_0\dd\theta$ with $g_0$ smooth and positive. Fix an integer $N$ and divide $\Tp$ into $N$ consecutive arcs $I_1,\dots,I_N$ of length $\pi/N$. Let $\mu$ be a probability measure on $[0,1]$ which is atomless, purely singular and of full support. 
Let $$\rho_1:=\sum_{k=1}^N\rho_0(I_k)\,(\alpha_k)_\#\mu,$$ where $\alpha_k:[0,1]\to\overline{I_k}$ is the affine parametrization and $\#$ denotes the push-forward. 
Then $\rho_1$ is a finite positive atomless measure of full support on $\Tp$, purely singular, with $\rho_1(I_k)=\rho_0(I_k)$ for all $k$, so that for every arc $J\subset\Tp$
\begin{equation}\label{eq:cell}
|\rho_1(J)-\rho_0(J)|\le2\max_k\rho_0(I_k)\le\frac{2\pi\|g_0\|_\infty}{N}.
\end{equation}
Let $h_1$ be defined from $\rho_1$ by \eqref{eq:h-from-rho}; by Proposition~\ref{prop:reconstruction}, $h_1\in\cH$ and $\rho_{h_1}=\rho_1$ is purely singular, {\it i.e.},\ $h_1\in\cS$.

We need to estimate $h_1-h_0$ in the $C^1$ norm. 
For fixed $\theta$ and either measure $\rho\in\{\rho_0,\rho_1\}$, integrate over the arc $[\theta,\theta+\pi)$ (see Proposition \ref{prop:reconstruction}): 
\[
h(\theta)=\int_{[\theta,\theta+\pi)}\tfrac12\sin(u-\theta)\dd\rho(u)=-\int_\theta^{\theta+\pi}R_\theta(u)\,\tfrac12\cos(u-\theta)\dd u,
\]
where we integrated by parts and put $R_\theta(u):=\rho([\theta,u))$.
Similarly, differentiating under the integral, 
and integrating again by parts,
\[
h'(\theta)=-\int_{[\theta,\theta+\pi)}\tfrac12\cos(u-\theta)\dd\rho(u)=\tfrac12\rho(\Tp)-\int_\theta^{\theta+\pi}R_\theta(u)\,\tfrac12\sin(u-\theta)\dd u .
\]
Since $\rho_1(\Tp)=\rho_0(\Tp)$ and using \eqref{eq:cell}:
$$|R^{(1)}_\theta(u)-R^{(0)}_\theta(u)|\le \frac{2\pi\|g_0\|_\infty}{N} \qquad  \forall\, \theta,u. $$ 
Therefore  
$$\max\{ \|h_1-h_0\|_{\infty},\, \|h'_1-h'_0\|_{\infty} \}\le\pi^2\|g_0\|_\infty/N<\varepsilon/4 \qquad \mbox{for $N$ large}.$$ 
Thus $\|h_1-h\|_{C^1}<\varepsilon$ with $h_1\in\cS$,
which completes the proof.
\end{proof}

\bigskip

We can now define the residual set $\cR$ and prove Theorem \ref{thm:no-graphs}.
 
\begin{definition}\label{def:R}
\[
\cR:=\cS\ \cap\ \bigcap_{q\ge2,\ m\ge2}\cU_{q,m}.
\]
\end{definition}

\medskip

\begin{proof}[Proof of Theorem \ref{thm:no-graphs}]
First of all observe that $\cR$ is a residual subset of $\cH$.  In fact,  by Lemmata~\ref{lem:Baire}, \ref{lem:S} and Theorem~\ref{thm:dense}, $\cR$ is a countable intersection of open dense subsets of the Baire space $\cH$, hence a dense $G_\delta$. 
In particular,  support functions of disks or  ellipses (with centre at the origin) belong to $\cH$, hence $\cR$ meets every $C^1$ neighbourhood of these support functions.

\smallskip
Let $h\in\cR$ and suppose $\Gamma$ is an invariant graph, with height function $s$. Since $\rho_h$ is purely singular, Theorem~\ref{thm:rigidity} provides an integer $q\ge2$ such that
\[
\tilde F^{\,q}\bigl(\theta,s(\theta)\bigr)=\bigl(\theta+\pi,s(\theta)\bigr)\qquad\text{for all }\theta\in\R.
\]
Since $s$ is continuous and positive on the compact circle $\Tt$, there is an integer $m\ge2$ with $s(\theta)\in[1/m,m]$ for every $\theta$. Then, no $\theta_*\in\R$ satisfies \eqref{eq:U} for this $q,m$,  hence\ $h\notin\cU_{q,m}$, contradicting the fact that $h\in\cR$.
\end{proof}

\medskip

\section{From the absence of invariant graphs to unbounded orbits: proof of the Main Theorem}\label{sec:topology}

In this section we complete the proof of the Main Theorem. In Section~\ref{sec:baire} we produced a residual set $\cR\subset\cH$ of tables with no invariant graph (Theorem~\ref{thm:no-graphs}). 
It is known  that for a twist homeomorphism of an annulus, invariant essential curves are the only obstruction to the existence of orbits crossing from one boundary circle to the other (see Mather's variational approach~\cite{Mat91}). 

The outer billiard map $T$ is, however, only a homeomorphism at $C^1$ regularity, and the exterior of $K$ has infinite area, so the classical smooth, area-preserving formulation of Mather's theory does not apply directly. We use instead Le Calvez's recent topological formulation~\cite{LC23}, which replaces smoothness by a purely topological monotonicity (twist) condition and replaces area-preservation by a non-wandering hypothesis.

\begin{remark}
According to our definition of invariant graphs (see Definition \ref{def:graph}), we are considering graphs that do not touch the boundary $\Tt\times \{0\}$ ({\it i.e.}, the corresponding curve in the plane does not touch $\partial K$). To apply Le Calvez's theory, we should take care of them. Moreover, since we will compactify the exterior of $K$ (see Section \ref{sec:compactify}), we should also deal with the possible existence of invariant graphs that touch the boundary at infinity. See Section \ref{subsec:boundary-graphs}.
\end{remark}

\medskip

\subsection{Compactification}\label{sec:compactify}
First of all, we need to compactify the exterior of $K$ to a closed annulus $\A$, so that $T$ extends to a homeomorphism $\hat T$ of $\A$.

\medskip

\begin{definition}
Let $\A:=\Tt\times[0,1]$ with boundary circles $C_0:=\Tt\times\{0\}$ and $C_1:=\Tt\times\{1\}$, and let $\cV$ be the \emph{vertical foliation} by the segments $\{\theta\}\times[0,1]$. The map
\[
\chi:\ \Tt\times(0,\infty)\to\Int\A,\qquad \chi(\theta,s):=\Bigl(\theta,\frac{s}{1+s}\Bigr),
\]
is a homeomorphism and $\Psi\circ\chi^{-1}:\Int\A\to\R^2\setminus K$ identifies the interior of the annulus with the exterior of $K$. The circle $C_0$ corresponds to $\partial K$ and $C_1$ to the boundary circle at infinity. 
\end{definition}

\medskip
Let us introduce the extension of $T$ to $\A$.

\medskip
\begin{definition}
Let $\hat T:\A\to\A$ be defined as
\begin{equation}\label{eq:That}
\hat T:=\chi\circ\Psi^{-1}\circ T\circ\Psi\circ\chi^{-1}\ \text{ on }\Int\A,
\end{equation}
with $\hat T(\theta,0):=(\theta,0)$, and $\hat T(\theta,1):=(\theta-\pi,1)$.
\end{definition}

\medskip

In terms of Proposition~\ref{prop:exact}, for $v=s/(1+s)\in(0,1)$ we have $\hat T(\theta,v)=(\theta+\delta,v^+)$ with $\delta=d(\theta,s)-\pi\in(-\pi,0)$ and $v^+=s^+/(1+s^+)$. The following asymptotics at the boundary hold.

\medskip

\begin{lemma}\label{lem:asymptotics}
Uniformly in $\theta$:
\begin{enumerate}[label=\textup{(\alph*)}]
\item as $s\to0$: $d(\theta,s)\to\pi$ and $s^+(\theta,s)\to0$;
\item as $s\to\infty$: $d(\theta,s)\to0$ and $s^+(\theta,s)\to\infty$.
\end{enumerate}
\end{lemma}

\begin{proof}

\noindent (a) \emph{The claim $d\to\pi$.} By Proposition~\ref{prop:exact} (ii), if $\pi-d\ge\varepsilon$ then $s\ge s\sin d\ge c_\varepsilon$; hence $s<c_\varepsilon$ forces $\pi-d<\varepsilon$. Since $c_\varepsilon$ depends only on $\varepsilon$, this gives $d(\theta,s)\to\pi$ as $s\to0$, uniformly in $\theta$.

\smallskip
\emph{The claim $s^+\to0$.} Let us first prove that
{tangent lengths tend to $0$ uniformly as the point approaches $K$}: namely, that for every $\varepsilon>0$ there is $\delta>0$ such that $\dist(y,K)<\delta$ and $y\notin K$ imply $s_K(y)<\varepsilon$.\\
We proceed by contradiction; assume that for some $\varepsilon>0$, there are points $y_n\notin K$ with $\dist(y_n,K)\to0$ and $s_K(y_n)\ge\varepsilon$ for all $n$. Since $K$ is compact and $\dist(y_n,K)\to0$, the sequence $(y_n)_n$ is bounded; passing to a subsequence, if necessary, we can assume that $y_n\to y$ for some $y$, and deduce $\dist(y,K)=0$, which  gives $y\in K$ (as $K$ is closed); actually, $y\in\partial K$, because $K$ has a neighbourhood of each interior point contained in $K$, which the exterior points $y_n$ can never enter. Passing to a further subsequence, we can assume that $\theta_K(y_n)\to\varphi\in\Tt$. The tangent lengths $s_K(y_n)\le|y_n|+M$ are bounded, where $M=\max_{\partial K}|y|$ (Proposition~\ref{prop:T-basic} (d)), so along a yet further subsequence $s_K(y_n)\to\sigma\ge\varepsilon>0$, and
\[
y=\lim_n\Psi(\theta_K(y_n),s_K(y_n))=p(\varphi)+\sigma\,\bt(\varphi).
\]
Thus $y\ne p(\varphi)$ lies on the supporting line at $p(\varphi)$ and in $K$, contradicting strict convexity.
This proves that the tangent lengths tend to $0$ uniformly as the point approaches $K$.

\smallskip
Now let $x=\Psi(\theta,s)$. Since $p(\theta)\in K$,
\[
\dist\bigl(T(x),K\bigr)\ \le\ \bigl|T(x)-p(\theta)\bigr|\ =\ \bigl|{-s\,\bt(\theta)}\bigr|\ =\ s ,
\]
a bound independent of $\theta$. Given $\varepsilon>0$, let $\delta>0$ be the one coming from the uniform convergence of $s_K$ discussed above. Then, $s<\delta$ gives $\dist(T(x),K)<\delta$, hence 
$$s^+(\theta,s)=s_K(T(x))<\varepsilon \qquad \forall\,\theta.$$ This is exactly $s^+(\theta,s)\to0$ as $s\to0$, uniformly in $\theta$.

\medskip
\noindent (b) By Proposition~\ref{prop:T-basic} (d):
 $$s^+\ge|T(x)|-M\ge|x|-3M\ge s-4M\to\infty,$$
 where $M=\max_{\partial K}|y|$.  Then Proposition~\ref{prop:exact} (i) gives 
 $$\sin^2(d/2)\le M^2/(s\,s^+)\to0,$$ from which we deduce that $d\to0$.
\end{proof}

\medskip

Let us prove some needed properties of the map $\hat T$.

\medskip

\begin{proposition}\label{prop:That}
$\hat T$ is a homeomorphism of $\A$ which preserves each boundary circle and is isotopic to the identity. Its lift to $\tilde\A:=\R\times[0,1]$,
\[
\tilde T(\theta,v)=(\theta+\delta(\theta,v),v^+),\qquad \delta\in(-\pi,0)\text{ on }\Int\tilde\A,\quad \delta=0\text{ on }\R\times\{0\},\quad \delta=-\pi\text{ on }\R\times\{1\},
\]
is continuous and commutes with $(\theta,v)\mapsto(\theta+2\pi,v)$.
\end{proposition}

\begin{proof}
$\hat T$ is a bijection of $\A$, being a bijection of $\Int\A$, of $C_0$ and of $C_1$. It is continuous on $\Int\A$ by Proposition~\ref{prop:T-basic} (a), and at the boundary by Lemma~\ref{lem:asymptotics}; in fact, as $v\to0$, $\theta+\delta\to\theta$ and $v^+\to0$ uniformly, while as $v\to1$, $\theta+\delta\to\theta-\pi$ and $v^+\to1$ uniformly. A continuous bijection of a compact Hausdorff space is a homeomorphism. 

The lift $\tilde T$ is continuous for the same reason, and clearly commutes with translation by $(2\pi,0)$.

 Finally, $\hat T$ preserves each boundary component and its restriction to $C_1$ is a rotation, so it induces the identity on $H_1(\A;\Z)\cong\Z$; a homeomorphism of the closed annulus with these properties is isotopic to the identity (see, for example, \cite{Eps66}).
\end{proof}

\bigskip

We need to prove now that $\hat T$ satisfies a sort of {\it topological twist condition}, in order to apply Le Calvez's theory  \cite{LC23}. More specifically, we prove monotonicity along fibers.

\medskip

\begin{proposition}\label{prop:twist}
Let $p_1:\tilde\A\to\R$ denote the projection on the first component.
For every $\theta\in\R$, the function $v\mapsto p_1\tilde T(\theta,v)$ is continuous and strictly decreasing on $[0,1]$, from $\theta$ to $\theta-\pi$, and $v\mapsto p_1\tilde T^{-1}(\theta,v)$ is continuous and strictly increasing on $[0,1]$, from $\theta$ to $\theta+\pi$.
\end{proposition}

\medskip

\begin{remark}\label{rem:toptwist}
Thus, in the terminology of twist maps, $\hat T^{-1}$ tilts vertical fibres to the right and $\hat T$ tilts them to the left: $\hat T^{-1}$ is a topological analogue of a \emph{positive twist map} in the sense of \cite{LC23}, except that no differentiability is available. 
\end{remark}

\medskip

\begin{proof}[Proof of Proposition \ref{prop:twist}]
On $(0,1)$ write $v=s/(1+s)$ and $x(s)=\Psi(\theta,s)$, so that $$p_1\tilde T(\theta,v)=\theta_K(T(x(s)))+2\pi k$$ for a fixed integer $k$, where $\theta_K(T(x(s)))$ is lifted continuously in $s$, and the lift is determined by $\delta\in(-\pi,0)$.
 \\
Since $T(x(s))=p(\theta)-s\bt(\theta)$ is $C^1$ in $s$ and $\theta_K$ is $C^1$, the chain rule and \eqref{eq:Dtheta} give
\[
\frac{\partial}{\partial s}\,\theta_K\bigl(T(x(s))\bigr)=D\theta_K\bigl(T x\bigr)\bigl(-\bt(\theta)\bigr)=\frac{\bt(\theta)\cdot\bn(\theta^+-\pi)}{s^+}=\frac{\sin(\theta^+-\pi-\theta)}{s^+}=-\frac{\sin d}{s^+}<0,
\]
where $\theta^+-\pi$ is the contact angle of $T(x)$. So $p_1\tilde T(\theta,\cdot)$ is strictly decreasing on $(0,1)$; it is continuous on $[0,1]$ with values $\theta$ at $v=0$ and $\theta-\pi$ at $v=1$, and for $v\in(0,1)$ it equals $\theta+\delta$ with $\delta\in(-\pi,0)$. Hence it is strictly decreasing on all of $[0,1]$.

\medskip
For $T^{-1}$, by Proposition~\ref{prop:T-basic} (a)
$$T^{-1}(x(s))=2p(\check\theta_K(x(s)))-x(s)$$ has contact angle $\check\theta_K(x(s))$, and by Lemma~\ref{lem:chart}
\[
\frac{\partial}{\partial s}\,\check\theta_K(x(s))=D\check\theta_K(x)\,\bt(\theta)=\frac{\bt(\theta)\cdot\bn(\check\theta_K(x))}{\check s_K(x)}=\frac{\sin(\check\theta_K(x)-\theta)}{\check s_K(x)}>0,
\]
because $\check\theta_K(x)-\theta\in(0,\pi)$ (see Lemma \ref{lem:two-tangents}). The endpoint values are $\theta$ and $\theta+\pi$ (apply Lemma~\ref{lem:asymptotics} to $T^{-1}$, or note that $\tilde T^{-1}$ is the inverse of $\tilde T$), and the same argument gives strict monotonicity on $[0,1]$.
\end{proof}

\medskip

\subsection{Le Calvez's topological approach to Mather's mechanism}\label{subsec:LC}

In this section, we recall the results of \cite{LC23} that we use.  Let us start with some definitions.

\medskip

\begin{definition}
(i) An \emph{essential simple loop} in $\A$ is a Jordan curve not contractible in $\A$; it is \emph{transverse to $\cV$} if it meets each vertical fibre exactly once, {\it i.e.}, if it is the graph $\{(\theta,\psi(\theta))\}$ of a continuous function $\psi:\Tt\to[0,1]$. 

\smallskip
\noindent (ii) A \emph{lower annulus} is a subset $U\subset\A$ containing $C_0$, disjoint from $C_1$ and homeomorphic to $\Tt\times[0,1)$; an \emph{upper annulus} is a subset containing $C_1$, disjoint from $C_0$ and homeomorphic to $\Tt\times(0,1]$. Such sets are open in $\A$ by invariance of domain. 

\smallskip
\noindent (iii) A homeomorphism $f$ of $\A$ is \emph{non-wandering} if there is no non-empty open set $W$ with $f^k(W)\cap W=\emptyset$ for all $k\ge1$.
\end{definition}

\medskip
Hereafter, we write $\Homeo_*(\A)$ for the group of homeomorphisms of $\A$ isotopic to the identity. Moreover, as above,  $p_1:\tilde\A\to\R$ will denote the projection on the first component.

\medskip
Le Calvez introduces, for $f\in\Homeo_*(\A)$ and a radial foliation $\mathcal F$, the notions of \emph{$\mathcal F$-increasing} and \emph{$\mathcal F$-decreasing} homeomorphism (\cite[Section 2.4]{LC23}); we say \emph{$\mathcal F$-monotone} for either. The definition is formulated through an abstract angle function on pairs of points of $\tilde\A$; we will not need it explicitly, only the following facts.

\begin{proposition}\label{prop:monotone}
$\hat T$ is $\cV$-monotone in the sense of \cite[Section 2.4]{LC23}. More precisely, $\hat T^{-1}$ is $\cV$-decreasing and $\hat T$ is $\cV$-increasing.
\end{proposition}

\begin{proof}
\cite[Proposition~4]{LC23} states that a positive twist map is $\cV$-decreasing. Its proof uses the $C^1$ twist hypothesis only to establish four order relations between pairs of points on a common vertical fibre and their images, namely: if $\tilde z=(\theta,v)$ and $\tilde z'=(\theta,v')$ with $v<v'$, then 
$$p_1\tilde f(\tilde z)<p_1\tilde f(\tilde z') \qquad {\rm and} \qquad p_1\tilde f^{-1}(\tilde z)>p_1\tilde f^{-1}(\tilde z'),$$ together with the symmetric statements obtained by exchanging $\tilde z$ and $\tilde z'$. The remaining part of the proof is a lifting argument on a simply connected set of pairs and uses no regularity. \\
For $f=\hat T^{-1}$ these four relations are exactly the content of Proposition~\ref{prop:twist}, including pairs with one point on a boundary circle. Hence $\hat T^{-1}$ is $\cV$-decreasing, and by \cite[Proposition~2 (1)]{LC23} its inverse $\hat T$ is $\cV$-increasing.
\end{proof}

\medskip

Let us cite two fundamental results from \cite{LC23}. 

\medskip

\begin{theorem}[Le Calvez {\cite[Theorem~5]{LC23}}]\label{thm:LC-Birkhoff}
Let $f\in\Homeo_*(\A)$ be non-wandering and $\cV$-monotone. If $U$ is an $f$-invariant lower or upper annulus, then its frontier $\Fr U$ is an essential simple loop transverse to $\cV$.
\end{theorem}

\begin{theorem}[Le Calvez {\cite[Theorem~6]{LC23}}]\label{thm:LC-Mather}
Let $f\in\Homeo_*(\A)$ be non-wandering and $\cV$-monotone. If $C_0$ and $C_1$ are the only $f$-invariant essential simple loops transverse to $\cV$, then there is $z\in\Int\A$ with
\[
\lim_{k\to-\infty}\dist\bigl(f^k(z),C_0\bigr)=0\qquad\text{and}\qquad\lim_{k\to+\infty}\dist\bigl(f^k(z),C_1\bigr)=0 .
\]
\end{theorem}

\medskip

We refer the reader to \cite{LC23} for their proofs. 

\medskip

\begin{remark}
{\bf (i)} Le Calvez states both theorems for $\mathcal F$-increasing maps and notes that they hold verbatim for $\mathcal F$-decreasing ones. \\
{\bf (ii)}
Theorem~\ref{thm:LC-Birkhoff} is the topological form of Birkhoff's theorem that the frontier of an invariant region containing a boundary circle is a graph~\cite{Bir32,Her83}, while Theorem~\ref{thm:LC-Mather} is the topological form of the theorem on orbits crossing a Birkhoff region of instability by Mather (see \cite{Mat91}). 
\end{remark}

\bigskip

Both Theorem~\ref{thm:LC-Birkhoff} and Theorem~\ref{thm:LC-Mather} require the homeomorphism to be {\it non-wandering}. For an area-preserving homeomorphism of a compact annulus it is automatic by Poincar\'e's recurrence theorem. In our case, however, it does not necessarily hold, since  the invariant measure of $T$ is the Euclidean area, which is infinite on $\R^2\setminus K$, and the corresponding invariant measure of $\hat T$ is infinite near $C_1$. Nevertheless, we observe that a wandering open set is itself a source of unbounded orbits.

\medskip

\begin{lemma}\label{lem:wandering}
Suppose there is a non-empty open set $W\subset\R^2\setminus K$ with $T^n(W)\cap W=\emptyset$ for all $n\ge1$. Then $|T^nx|\to\infty$ as $n\to+\infty$ for Lebesgue-almost every $x\in W$.
\end{lemma}

\begin{proof}
The sets $T^n(W)$, $n\ge0$, are pairwise disjoint (if $T^n W\cap T^mW\ne\emptyset$ with $n>m$, then $T^{n-m}W\cap W\ne\emptyset$). For $R>0$ let $B_R$ be the disk of radius $R$; by area preservation (Proposition~\ref{prop:T-basic} (c)),
\[
\int_W\sum_{n\ge0}\mathbf 1_{B_R}(T^nx)\dd x=\sum_{n\ge0}\Leb\bigl(T^n(W)\cap B_R\bigr)\le\Leb(B_R)<\infty .
\]
Hence for a.e.\ $x\in W$ the orbit visits $B_R$ only finitely many times. Taking $R\in\N$ and a countable intersection, a.e.\ $x\in W$ visits every disk finitely often, {\it i.e.},\ $|T^nx|\to\infty$.
\end{proof}

\medskip

\begin{corollary}\label{cor:dichotomy}
Either $\hat T$ is non-wandering on $\A$, or the outer billiard map has a positive measure set of forward escaping orbits.
\end{corollary}

\begin{proof}
A non-empty open subset $W$ of $\A$ with $\hat T^n(W)\cap W=\emptyset$ for all $n\ge1$ meets $\Int\A$ (a non-empty open subset of $\A$ is not contained in $\partial\A$), and $W\cap\Int\A$ corresponds to a wandering open set in $\R^2\setminus K$. The claim follows by applying Lemma~\ref{lem:wandering}.
\end{proof}

\medskip

\subsection{Invariant graphs touching the boundary circles}\label{subsec:boundary-graphs}

A priori Theorem~\ref{thm:rigidity} says nothing about possible invariant graphs touching $C_0$ or $C_1$, and we now treat those directly.\\

Let us first recall the following classical fact and some of its implications. Hereafter in this section, we denote by $\Gamma_\psi$ a graph with height function $\psi$.

\medskip

\begin{definition}
By {\it invariant graph in}  $\A$ we mean a set 
$$
\Gamma_\psi=\{(\theta, \psi (\theta)),\; \theta\in \Tt\}
$$
with $\psi: \Tt \longrightarrow [0,1]$ a continuous function and $\hat{T}(\Gamma_\psi)=\Gamma_\psi$;
when $0 <\psi <1$, it corresponds, via $\Psi\circ \chi^{-1}$, to an invariant graph in the sense of Definition \ref{def:graph}.
\end{definition}

\medskip

\begin{lemma}\label{lem:rotationboundary}
Let $\Gamma_\psi$ be an invariant graph, $\tilde\Gamma\subset\tilde\A$ its lift, and $G(\theta):=p_1\tilde T(\theta,\psi(\theta))$, so that $G:\R\to\R$ is the lift of the homeomorphism induced by $\hat T$ on $\Gamma_\psi$ and $p_1\tilde T^n(\theta,\psi(\theta))=G^n(\theta)$. Then $G$ is an increasing homeomorphism with $G(\theta+2\pi)=G(\theta)+2\pi$, and the limit
\[
\rho(\Gamma_\psi):=\lim_{n\to\infty}\frac{G^n(\theta)-\theta}{2\pi n}
\]
exists and is independent of $\theta$ \textup{(}the \emph{rotation number} of $\Gamma_\psi$ for the lift $\tilde T$\textup{)}, and
\begin{equation}\label{eq:bounded-deviation}
\bigl|G^n(\theta)-\theta-2\pi n\,\rho(\Gamma_\psi)\bigr|<2\pi\qquad\text{for all }\theta\in\R,\ n\ge1 .
\end{equation}
Moreover:
\begin{enumerate}[label=\textup{(\roman*)}]
\item if $\Gamma_\psi$ meets $C_0$, then $\rho(\Gamma_\psi)=0$;
\item if $\Gamma_\psi$ meets $C_1$, then $\rho(\Gamma_\psi)=-\frac 12$;
\item consequently no invariant graph meets both $C_0$ and $C_1$.
\end{enumerate}
\end{lemma}

\smallskip

This follows from the classical  theory of homeomorphisms of the circle (see for example \cite[Section 11.1]{KH95}).

\medskip

\begin{remark}\label{rem:AM}
The rotation number alone does not suffice for excluding the possibility of invariant graphs touching the boundaries.
In fact, for a $C^2$ exact area-preserving twist map, this would be possible via Aubry--Mather theory~\cite{AL83,Ban88, Gole, Mat82}, in particular it would follow from Mather's graph theorem: since all points on $C_0$ and $C_1$ are action-minimizing,  there cannot be any other invariant graph with that rotation number. 
This variational theory, however, is formulated for generating functions of class $C^2$, or at least $C^1$ with a strictly monotone twist, whereas in our case the twist is only topological and the map is a homeomorphism (and the compactification is degenerate at $C_1$, where the derivative of the action  tends to infinity). Therefore, we prove the needed facts directly from the rotation-number estimates and the exact billiard equations..
\end{remark}

\medskip

Let us start with the boundary at infinity.

\medskip

\begin{lemma}\label{lem:divergence}
There is a constant $c_K>0$ such that for every $x\in\R^2\setminus K$, with $x_j=T^jx=\Psi(\varphi_j,s_j)$ and $d_j:=d(\varphi_j,s_j)\in(0,\pi)$ the increment of Proposition~\textup{\ref{prop:exact}} at $x_j$, then
\[
d_j\ \ge\ \frac{c_K}{1+s_j}\qquad\text{and}\qquad s_j\le|x|+(2j+1)M \qquad(j\ge0),
\]
where $M:=\max_{y\in\partial K}|y|$. In particular $\sum_{j\ge0}d_j=\infty$.
\end{lemma}

\medskip

\begin{proof}
Let $h_{\min}=\min h$ and choose $d_0\in(0,\pi/2)$ with $\tan d_0\le h_{\min}/\|h'\|_\infty$ (or any $d_0\in(0,\pi/2)$ if $h'\equiv0$). For $0<d\le d_0$ the right-hand side of \eqref{eq:exact1} satisfies
\[
h(\theta+d)+h(\theta)\cos d+h'(\theta)\sin d\ \ge\ h_{\min}+h_{\min}\cos d_0-\|h'\|_\infty\sin d_0\ \ge\ h_{\min}.
\]
Hence if $d_j\le d_0$ then 
$$s_jd_j\ge s_j\sin d_j\ge h_{\min}  \quad \Longrightarrow \quad d_j\ge \frac{h_{\min}}{s_j} \ge \frac{h_{\min}}{1+s_j}.$$
If $d_j>d_0$ then $d_j>d_0/(1+s_j)$. Take $c_K:=\min(h_{\min},d_0)$. The bound on $s_j$ follows from Proposition~\ref{prop:T-basic} (d). Clearly, 
$$\sum_j \frac{c_K}{1+|x|+ (2j+1)M}= +\infty.$$
\end{proof}

\medskip

Let us now prove that there cannot be invariant graphs touching the boundary at infinity.

\medskip
\begin{proposition}\label{prop:infinity}
If an invariant graph $\Gamma_\psi$ meets $C_1$, then $\Gamma_\psi=C_1$.
\end{proposition}

\begin{proof}
By Lemma~\ref{lem:rotationboundary} (ii), $\rho(\Gamma_\psi)=-\frac 12$. 
Suppose that $\Gamma_\psi\ne C_1$.
Since $\Gamma_\psi$ meets $C_1$,
Lemma~\ref{lem:rotationboundary} (iii) implies that
it is disjoint from $C_0$.
We can therefore choose $\theta$ such that
$0<\psi(\theta)<1$.
Set
\[
z:=(\theta,\psi(\theta)),\qquad
x:=\Psi\circ\chi^{-1}(z)\in\R^2\setminus K.
\]
Along the orbit of $x$, write
\[
d_j:=d\bigl(\theta_K(T^jx),s_K(T^jx)\bigr),
\qquad
\delta_j:=d_j-\pi\in(-\pi,0).
\]
The lifted displacement along the corresponding orbit
on $\Gamma_\psi$ is therefore
\begin{equation}\label{eq:finite-displacement}
G^n(\theta)-\theta=\sum_{j<n}\delta_j=-n\pi+\sum_{j<n}d_j.
\end{equation}
Then, using \eqref{eq:bounded-deviation} and \eqref{eq:finite-displacement}, we deduce that
\[
\sum_{j<n}d_j=\bigl|G^n(\theta)-\theta-2\pi n\rho(\Gamma_\psi)\bigr|<2\pi\qquad\text{for all }n,
\]
contradicting $\sum_jd_j=\infty$ (Lemma~\ref{lem:divergence}). Therefore,  $\psi\equiv1$.
\end{proof}

\medskip

Now we discuss the possibility of graphs touching $C_0$. However, being interested in escaping orbits, these objects do not really represent an obstruction. So, we do not really show that they cannot exist (that might be the case, see Remark \ref{rem:AM}), but we only show that they must be uniformly bounded (and, therefore, away from $C_1$).

\medskip

\begin{proposition}\label{prop:inner}
If an invariant graph $\Gamma_\psi$ meets $C_0$, then every point of $\Gamma_\psi\cap\Int\A$ corresponds to a point $x \in \R^2\setminus K$ with $|x|\le12M$, where $M:=\max_{y\in\partial K}|y|$. Consequently,
\[
\psi\le c_K':=\frac{13M}{1+13M}<1 .
\]
Consequently $\Gamma_\psi\cap C_1=\emptyset$ and $\Gamma_\psi\setminus C_0\subset\Int\A$, {\it i.e.}, an invariant graph touching  $C_0$ lies entirely in the compact region $\Tt\times[0,c_K']$, uniformly away from $C_1$.
\end{proposition}

\medskip

\begin{proof}
By Lemma~\ref{lem:rotationboundary} (i), $\rho(\Gamma_\psi)=0$, so by \eqref{eq:bounded-deviation} and \eqref{eq:finite-displacement} every finite point $x$ ({\it i.e.}, point not on $C_1$) of $\Gamma_\psi$ satisfies 
\begin{equation}\label{eq:boundlemmarot}
|\sum_{j<n}\delta_j|<2\pi \qquad \forall\,n\ge1.
\end{equation}

Let us assume by contradiction that $|x|>12M$.
Let $x_j=T^jx$, $s_j=s_K(x_j)$. By Proposition~\ref{prop:T-basic} (d)
$$|x_j|\ge|x|-2Mj>6M  \qquad {\rm for}\;0\le j\le3,$$ so $s_j\ge|x_j|-M>5M$ for $0\le j\le3$. \\
Applying Proposition~\ref{prop:exact} (i) at $x_j$ for $j=0,1,2$ (the tangent length of $F(x_j)=-x_{j+1}$ is $s_{j+1}$), we get:
\[
4M^2\ \ge\ 4s_js_{j+1}\sin^2(d_j/2)\ >\ 100M^2\sin^2(d_j/2),
\]
so 
$$\sin(d_j/2)<1/5<\tfrac12  \qquad \Longrightarrow \quad d_j<\pi/3 \quad {\rm for}\;j=0,1,2.$$ 
Therefore 
$$\sum_{j<3}\delta_j=\sum_{j<3}(d_j-\pi)<\pi-3\pi=-2\pi,$$ contradicting \eqref{eq:boundlemmarot}. 

Hence, $|x|\le12M$ and consequently:
$$ s\le|x|+M\le13M \quad{\rm and} \quad \psi=\frac{s}{1+s}\le \frac{13M}{1+13M}.$$

\end{proof}

\bigskip

\subsection{Proof of the Main Theorem}\label{subsec:main-proof}

We can now prove the Main Theorem stated in Section \ref{sec:mainthm}.

\medskip

Recall that items (i) and (ii) follow from Theorem \ref{thm:no-graphs} (see Remark \ref{re:proofmthmiii}). Therefore, to complete the proof of the Main Theorem, we only need to prove (iii).

\medskip

\begin{proof}[Proof of the Main Theorem (iii)]
If $\hat T$ is not non-wandering,
Corollary~\ref{cor:dichotomy} already provides forward
escaping orbits. 
Let $h\in\cR$ and assume that $\hat T$ is non-wandering on $\A$.  We show that there is then an orbit asymptotic to $C_1$.

\medskip

Since we have not excluded a priori the existence of invariant graphs touching $C_0$, we construct a new (possibly smaller) annulus $A_*$ to which we can apply Le Calvez's result (Theorem~\ref{thm:LC-Mather}). 

\medskip

\begin{itemize}
\item {\bf Step 1:} 
Let $\cG$ be the set of invariant graphs $\Gamma_\psi\ne C_1$. If $\Gamma_\psi\in\cG$ met $C_1$ it would coincide with $C_1$ by Proposition~\ref{prop:infinity}; if it met neither $C_0$ nor $C_1$ it would be an invariant graph, excluded by Theorem~\ref{thm:no-graphs}. Hence every $\Gamma_\psi\in\cG$ meets $C_0$, has rotation number $0$ (Lemma~\ref{lem:rotationboundary}), and by Proposition~\ref{prop:inner} satisfies $$\psi\le c:=\frac{13M}{1+13M}<1,$$
where $M:=\max_{y\in\partial K}|y|$. 

\medskip
\item {\bf Step 2:} 
For $\Gamma_\psi\in\cG$, let
\[
D_\psi:=\{(\theta,v)\in\A:v\le\psi(\theta)\},
\qquad
V_\psi:=\{(\theta,v)\in\A:v>\psi(\theta)\}.
\]

\smallskip
We claim that $D_\psi$ is $\hat T$-invariant.
Since $\psi\le c<1$, the set $V_\psi$ contains $C_1$.
It is connected: every point of $V_\psi$ can be joined to
$C_1$ by a vertical segment contained in $V_\psi$, and $C_1$
is connected. Moreover, the continuous function
$(\theta,v)\mapsto v-\psi(\theta)$ never vanishes on
$\A\setminus\Gamma_\psi$, so its sign is constant on each
connected component of this complement. Thus $V_\psi$ is
precisely the connected component of
$\A\setminus\Gamma_\psi$ containing $C_1$.

Because $\hat T$ is a homeomorphism preserving both
$\Gamma_\psi$ and $C_1$, it maps this component onto itself:
\[
\hat T(V_\psi)=V_\psi.
\]
Since $D_\psi=\A\setminus V_\psi$, bijectivity of $\hat T$
then gives
\[
\hat T(D_\psi)
=
\A\setminus\hat T(V_\psi)
=
\A\setminus V_\psi
=
D_\psi.
\]

 Let
\[
D:=\overline{\bigcup_{\Gamma_\psi\in\cG}D_\psi}.
\]
Then $D$ is closed, $\hat T$-invariant (as $\hat T$ is a homeomorphism), contains $C_0$, and it is contained in $\{v\le c\}$.\\  
Define
\[
a(\theta):=\max\{v:(\theta,v)\in D\}\in[0,c],
\]
where this maximum exists because $D\cap(\{\theta\}\times[0,1])$ is compact and non-empty. Observe that $a$ is upper semicontinuous, because $D$ is closed. Moreover 
$$D=\{(\theta,v):v\le a(\theta)\},$$ since if $(\theta,v)\in D$ and $v'<v$ then $(\theta,v')\in D$ (this holds for each $D_\psi$, for the union, and passes to the closure). \\
Consequently
\[
U:=\A\setminus D=\{(\theta,v):v>a(\theta)\}
\]
is open, $\hat T$-invariant, contains $C_1$ and  it is disjoint from $C_0$.

\medskip
\item {\bf Step 3:} 
Fix $n_0$ with $c+1/n_0<1$. For $n\ge n_0$ let us define
\[
a_n(\theta):=\sup_{\theta'\in\Tt}\bigl[a(\theta')-n\,\dist(\theta,\theta')\bigr]+\tfrac1n .
\]
Each $a_n$ is $n$-Lipschitz (being a supremum of $n$-Lipschitz functions), and pointwise
$$1>a_n\ge a+1/n>a \qquad {\rm and} \qquad a_{n+1}<a_n.$$ 
Moreover $a_n\downarrow a$. In fact,  given $\theta$ and $\varepsilon>0$, upper semicontinuity gives $\eta>0$ with $a(\theta')<a(\theta)+\varepsilon$ for $\dist(\theta,\theta')<\eta$, so 
$$\sup_{\theta'}[a(\theta')-n\dist(\theta,\theta')]\le\max\{a(\theta)+\varepsilon,\,c-n\eta\}\to a(\theta)+\varepsilon.$$
Therefore, it follows that 
$$U=\bigcup_{n\ge n_0}\{v>a_n(\theta)\}$$ and that $\{v\ge a_n\}\subset\{v>a_{n+1}\}$.

\smallskip

For $n\ge n_0$, set
\[
b_n:=\frac{1}{n-n_0+2}.
\]
Then $b_{n_0}=1/2$, the sequence $(b_n)_n$ is strictly
decreasing, and $b_n\to0$.

For each $\theta\in\Tt$, define
\[
H_\theta:(a(\theta),1]\longrightarrow(0,1]
\]
by increasing affine interpolation as follows:
\begin{itemize}
\item map $[a_{n_0}(\theta),1]$ affinely onto $[1/2,1]$;
\item for every $n\ge n_0$, map
$[a_{n+1}(\theta),a_n(\theta)]$ affinely onto
$[b_{n+1},b_n]$.
\end{itemize}
The prescriptions agree at every common endpoint:
\[
H_\theta(1)=1,
\qquad
H_\theta(a_n(\theta))=b_n
\quad(n\ge n_0).
\]
Since $a_n(\theta)\downarrow a(\theta)$ and $b_n\downarrow0$,
$H_\theta$ is an increasing homeomorphism from
$(a(\theta),1]$ onto $(0,1]$.

Define now
\[
\Theta:U\longrightarrow\Tt\times(0,1],
\qquad
\Theta(\theta,v):=(\theta,H_\theta(v)).
\]
The breakpoints $a_n(\theta)$ depend continuously on
$\theta$, and the affine formulas agree at adjacent
endpoints. Near every point of $U$, only finitely many
interpolation pieces occur, so $\Theta$ is continuous.
The same argument applies to its inverse, using the
breakpoints $b_n$. Thus $\Theta$ is a homeomorphism.
It fixes $C_1$ pointwise, and hence $U$ is an upper
annulus in the sense of Section~\ref{subsec:LC}.

\medskip

\item{\bf Step 4:} 
 By Propositions~\ref{prop:That} and~\ref{prop:monotone}, $\hat T\in\Homeo_*(\A)$ is $\cV$-monotone, and it is non-wandering by assumption. Theorem~\ref{thm:LC-Birkhoff} applied to $U$ shows that $\Fr U$ is the graph of a continuous function $\psi_*$.\\
  For every $\theta$ the point $(\theta,a(\theta))$ lies in $D=\A\setminus U$ and is a limit of points $(\theta,a(\theta)+\varepsilon)\in U$, so it belongs to $\Fr U$. Since $\Fr U$ meets each fibre exactly once, $\psi_*=a$. Hence $a$ is continuous, and $\Gamma_*:=\Gamma_a=\Fr U$ is an invariant graph (the frontier of an invariant set, by \eqref{eq:frontier}) with $a\le c<1$, so $\Gamma_*\in\cG$. \\
$\Gamma_*$ is the maximal invariant graph for $\cG$ in the sense that if $\Gamma_\psi\in\cG$ then $D_\psi\subset D$, {\it i.e.},\ $\psi\le a$.

\end{itemize}

\bigskip

Let 
$$A_*:=\{(\theta,v):a(\theta)\le v\le1\}=\overline U,$$
which is a $\hat T$-invariant closed set, and let 
\begin{eqnarray}
\Theta_*:A_* &\longrightarrow& \A \nonumber \\
(\theta,v) &\longmapsto& \left( \theta,  \frac{v-a(\theta)}{1-a(\theta)}\right),
\end{eqnarray}
which is a homeomorphism preserving each vertical fibre and is increasing on it, mapping $\Gamma_*$ onto $C_0$ and $C_1$ onto $C_1$.

\medskip

Define 
\begin{equation}
f:=\Theta_*\circ\hat T\circ\Theta_*^{-1}\in\Homeo(\A).
\end{equation}

Observe that $f$ satisfies  the hypotheses of Theorem~\ref{thm:LC-Mather}. In fact:
 
 \medskip

\begin{itemize}
\item $f$ preserves $C_0$ and $C_1$ and acts on $C_1$ as a rotation. 

\smallskip
\item $f$ is isotopic to the identity. Indeed, $f$ preserves
each boundary component, and since $\Theta_*$ fixes $C_1$
pointwise, its restriction to $C_1$ is
\[
f(\theta,1)=(\theta-\pi,1).
\]
This rotation has degree one. Since the inclusion
$C_1\hookrightarrow\A$ induces an isomorphism on first
homology, $f$ induces the identity on
$H_1(\A;\Z)\cong\Z$. The annulus isotopy criterion used
in the proof of Proposition~\ref{prop:That}
(see \cite{Eps66}) therefore shows that $f$ is isotopic
to the identity. Hence $f\in\Homeo_*(\A)$.\smallskip
\item $f$ is $\cV$-monotone: the four order relations of Proposition~\ref{prop:twist} concern pairs of points on a common vertical fibre and the fibres of their images, and are preserved by the fibre-preserving, fibrewise increasing homeomorphism $\Theta_*$; so the argument of Proposition~\ref{prop:monotone} applies verbatim to $f$.

\smallskip
\item $f$ is non-wandering: let $W\subset\A$ be non-empty and open. Then $\Theta_*^{-1}(W)$ is a non-empty relatively open subset of $A_*=\overline U$, hence meets the open set $U$, and $W':=\Theta_*^{-1}(W)\cap U$ is a non-empty open subset of $\A$. Since $\hat T$ is non-wandering, $\hat T^n(W')\cap W'\ne\emptyset$ for some $n\ge1$, and then $f^n(W)\cap W\ne\emptyset$.

\smallskip
\item $C_0$ and $C_1$ are the only $f$-invariant essential loops transverse to $\cV$: such a loop corresponds to an invariant graph $\Gamma_\psi\subset A_*$, {\it i.e.},\ with $\psi\ge a$. Either $\Gamma_\psi=C_1$, or $\Gamma_\psi\in\cG$ and then $\psi\le a$ by Step~4, so $\psi=a$ and $\Gamma_\psi=\Gamma_*$.
\end{itemize}

\bigskip

Therefore, Theorem~\ref{thm:LC-Mather} provides $z\in\Int\A$ with $$\dist(f^k(z),C_1)\to0 \quad {\rm as}\quad k\to+\infty.$$
 Let 
 $$(\theta,v)=\Theta_*^{-1}(z)\in U\cap\Int\A \quad {\rm and} \quad x:=\Psi\chi^{-1}(\theta,v)\in\R^2\setminus K.$$
  Writing $\hat T^k(\theta,v)=(\theta_k,v_k)$ we have $v_k\to1$, which implies 
  $$s_K(T^kx)= \frac{v_k}{1-v_k}\to\infty.$$ Therefore 
  $$|T^kx|\ge s_K(T^kx)-M\to\infty$$ by Proposition~\ref{prop:T-basic} (d).  This concludes the proof.
\end{proof}

\medskip

\begin{remark}
The proof gives slightly more: in the non-wandering case, the escaping orbit is bi-asymptotic, its backward iterates accumulating on the invariant graph $\Gamma_*$ (which could possibly be $\partial K$ itself), and Theorem~\ref{thm:LC-Mather} also provides an orbit coming from infinity and accumulating on $\Gamma_*$ in forward time. 
\end{remark}

\bigskip

\appendix

\section{Proofs of some properties of the outer billiard map}\label{app:proofs}

In this Appendix we collect the proofs of the results stated in Section \ref{subsec:T}.

\medskip

We recall that $h\in\cH$, $K=K_h$, and $p$, $\bn$, $\bt$ are as in Proposition~\ref{prop:reconstruction}. Moreover, $\bt(\theta)$ is the unit tangent to $\partial K$ at $p(\theta)$ pointing in the counterclockwise direction.

\medskip

Let us first state and prove an elementary fact about the supporting lines through an exterior point.

\medskip

\begin{lemma}\label{lem:two-tangents}
Let  $x\in\R^2\setminus K$. Define 
$$f_x(\theta):=x\cdot\bn(\theta)-h(\theta),$$
which is a $C^1$ function on $\Tt$ with
\[
f_x'(\theta)=x\cdot\bt(\theta)-h'(\theta)=\bigl(x-p(\theta)\bigr)\cdot\bt(\theta).
\]
Then $f_x$ has exactly two zeros $\theta_1\neq\theta_2$ in $\Tt$, both simple, with $f_x'(\theta_1)>0>f_x'(\theta_2)$. The set $\{f_x>0\}$ is the open arc running counterclockwise from $\theta_1$ to $\theta_2$, and its length lies in $(0,\pi)$.
\end{lemma}

\begin{proof}
The zeros of $f_x$ are the normal angles of the supporting lines of $K$ through $x$. A point outside a compact convex body with non-empty interior lies on exactly two supporting lines 
(see, for example, \cite[Section 1.3]{Sch14}), so $f_x$ has exactly two zeros. 

\smallskip
If $f_x(\theta_0)=0$ then $x$ lies on the supporting line $\{y\cdot\bn(\theta_0)=h(\theta_0)\}$, which meets $K$ only at $p(\theta_0)$. Hence $$x=p(\theta_0)+\sigma\bt(\theta_0) \quad {\rm with} \quad \sigma=(x-p(\theta_0))\cdot\bt(\theta_0)=f_x'(\theta_0),$$ and $\sigma\neq0$ because $x\notin K$. 

\smallskip
Thus both zeros are simple, $f_x$ changes sign at each of them, and the derivatives at the two zeros have opposite signs. The function $f_x$ takes positive values (by separation, some $\bn$ satisfies $x\cdot\bn>h(\bn)$) and negative values (at $\bn=-x/|x|$, $f_x=-|x|-h<0$). Hence $\{f_x>0\}$ is one of the two open arcs determined by the zeros, namely the one beginning at the zero where $f_x$ increases through $0$. 

\smallskip
Finally, choose real representatives of the two zeros such
that
\[
\theta_1<\theta_2<\theta_1+2\pi
\]
and the positivity arc $\{f_x>0\}$ is represented by
$(\theta_1,\theta_2)$.
Since $h$ is $\pi$-periodic and
$\bn(\theta+\pi)=-\bn(\theta)$, we have
\[
f_x(\theta)+f_x(\theta+\pi)=-2h(\theta)<0.
\]
Therefore the positivity arc contains no pair of antipodal
angles, which implies $\theta_2-\theta_1\le\pi$.
If equality held, its endpoints would be antipodal zeros,
and hence
\[
0=f_x(\theta_1)+f_x(\theta_2)
 =f_x(\theta_1)+f_x(\theta_1+\pi)
 =-2h(\theta_1)<0,
\]
a contradiction. Consequently
\[
0<\theta_2-\theta_1<\pi,
\]
as claimed.
\end{proof}

\medskip

\subsection{Proof of Lemma \ref{lem:chart}}

\begin{proof}
For $(\theta,s)$ with $s>0$, the point $y=\Psi(\theta,s)$ satisfies 
$$f_y(\theta)=0 \quad {\rm and} \quad f_y'(\theta)=s>0.$$
Since $y$ lies on the supporting line with normal $\bn(\theta)$ and $y\neq p(\theta)$ (as $s > 0$), strict convexity gives $y \not\in K$; by Lemma ~\ref{lem:two-tangents}
the zero of $f_y$ with positive derivative is $\theta_K(y)$, so $\theta_K(y) = \theta$ and  $s_K(y)=s$.

\smallskip
Conversely $\Psi(\theta_K(x),s_K(x))=x$ by Definition~\ref{def:contact}. So $\Psi$ is a continuous bijection with inverse $(\theta_K,s_K)$.

\smallskip
It remains to show that $\theta_K$ is $C^1$, since $s_K(x)=x\cdot\bt(\theta_K(x))-h'(\theta_K(x))$ is then continuous.

Consider 
$$G(x,\theta):=x\cdot\bn(\theta)-h(\theta),$$ a $C^1$ function on $\R^2\times\R$ with $\partial_\theta G(x,\theta_K(x))=s_K(x)>0$. By the implicit function theorem there are a neighbourhood $W$ of $x$ and a $C^1$ function $\vartheta:W\to\R$ such that 
$$\vartheta(x)=\theta_K(x) \quad {\rm and} \quad G(x',\vartheta(x'))=0 \quad \forall\, x'\in W,$$ 
$\vartheta(x')$ being the unique zero of $f_{x'}$ near $\theta_K(x)$. Moreover, by continuity we can assume (possibly reducing $W$)  that $\partial_\theta G(x',\vartheta(x'))>0$ on $W$.

\smallskip
 We claim that $\vartheta=\theta_K$ on $W$ (possibly shrinking $W$). Indeed, as $x'\to x$ the two zeros of $f_{x'}$ converge to the two zeros of $f_x$ (the zeros vary in a compact set, and any limit of zeros of $f_{x'}$ is a zero of $f_x$ by uniform convergence of $f_{x'}\to f_x$; both zeros of $f_x$ are limits since $f_x$ changes sign at each of them).  The zero of $f_{x'}$ near $\theta_K(x)$ has positive derivative, so it is $\theta_K(x')$, and it equals $\vartheta(x')$ by uniqueness. 
 
 \smallskip
 Differentiating $G(x,\theta_K(x))=0$ gives $$v\cdot\bn(\theta_K)+s_K\,D\theta_K(x)v=0,$$ which is \eqref{eq:Dtheta}. The statement for $\check\theta_K$ is proven in the same way, using $$\partial_\theta G(x,\check\theta_K(x))=-\check s_K(x)<0.$$
\end{proof}

\medskip

\subsection{Proof of Proposition~\ref{prop:T-basic}}

\begin{proof}
(a) Let $x=\Psi(\theta,s)$ and $y=T(x)=p(\theta)-s\bt(\theta)$. Then $f_y(\theta)=0$ and 
$$f_y'(\theta)=(y-p(\theta))\cdot\bt(\theta)=-s<0,$$ so $\theta=\check\theta_K(y)$ and $\check s_K(y)=s$. 

\smallskip
Define $U(y):=2p(\check\theta_K(y))-y$. Then $$U(T(x))=2p(\theta)-y=x.$$ 

\smallskip
Conversely, for any $y\notin K$ 
$$U(y)=p(\check\theta_K(y))+\check s_K(y)\bt(\check\theta_K(y))$$ has contact angle $\check\theta_K(y)$ by Lemma~\ref{lem:chart}, so $$T(U(y))=2p(\check\theta_K(y))-U(y)=y.$$ Both $T$ and $U$ are continuous by Lemma~\ref{lem:chart}.

\bigskip
\noindent (b) If $x=\Psi(\theta,s)$ then 
$$-x=-p(\theta)-s\bt(\theta)=p(\theta+\pi)+s\bt(\theta+\pi)=\Psi(\theta+\pi,s),$$ so 
$$T(-x)=2p(\theta+\pi)+x=-2p(\theta)+x=-T(x).$$

\bigskip
\noindent (c) Consider the $C^1$ map 
\begin{eqnarray}
\Phi:\Tt\times\R &\longrightarrow& \R^2 \nonumber\\
(\theta,w) &\longmapsto& h(\theta)\bn(\theta)+w\,\bt(\theta),\label{eq:itemcchart}
\end{eqnarray}
 whose Jacobian determinant is
\[
\det\bigl[\partial_\theta\Phi,\partial_w\Phi\bigr]=\det\bigl[(h'-w)\bn+h\bt,\ \bt\bigr]=h'(\theta)-w .
\]
On $\Omega_+:=\{w>h'(\theta)\}$ we have $$\Phi(\theta,w)=\Psi(\theta,w-h'(\theta)),$$  so $\Phi|_{\Omega_+}$ is a homeomorphism onto $\R^2\setminus K$ with non-vanishing Jacobian, hence a $C^1$ diffeomorphism. \\
On $\Omega_-:=\{w<h'(\theta)\}$ we have
\[
\Phi(\theta,w)
=
T\bigl(\Psi(\theta,h'(\theta)-w)\bigr).
\]
Since $T$ and $\Psi$ are homeomorphisms, $\Phi|_{\Omega_-}$
is a homeomorphism onto $\R^2\setminus K$.
Its Jacobian does not vanish, so it is a $C^1$ diffeomorphism.

\smallskip
The reflection $$\sigma(\theta,w):=(\theta,2h'(\theta)-w)$$ exchanges $\Omega_\pm$, preserves $\dd\theta\dd w$ and the function $|w-h'(\theta)|$, and satisfies $$T\circ\Phi=\Phi\circ\sigma \quad {\rm on}\; \Omega_+,$$ because there $$\Phi(\sigma(\theta,w))=2p(\theta)-\Phi(\theta,w) \quad {\rm and} \quad \theta_K(\Phi(\theta,w))=\theta.$$  
For a Borel set $B\subset\R^2\setminus K$, with $\tilde B:=\Phi^{-1}(B)\cap\Omega_+$, the change of variables formula gives
\[
\Leb(T(B))=\int_{\sigma(\tilde B)}|h'(\theta)-w|\dd\theta\dd w=\int_{\tilde B}|h'(\theta)-(2h'(\theta)-w)|\dd\theta\dd w=\Leb(B).
\]

\bigskip
\noindent (d) $|T(x)|=|2p-x|\le 2M+|x|$, and the same for $T^{-1}$. Since $x=p(\theta)+s\bt(\theta)$ with $|p|\le M$, we get $s-M\le|x|\le s+M$.
\end{proof}

\bigskip

\subsection{Proof of Proposition \ref{prop:DT}}

\begin{proof}
(a) The function $p$ is continuous of bounded variation with $Dp=\bt\,\hat\rho_h$ by \eqref{eq:Dp}, so it is differentiable at Lebesgue-almost every $\theta$ with $p'(\theta)=r(\theta)\bt(\theta)$, which is the density of the absolutely continuous part of $Dp$~\cite[Section 3.5]{Fol99}. \\
Let $E$ be a Borel null set containing the exceptional angles. For $x=\Psi(\theta,s)$ with $\theta=\theta_K(x)\notin E$, the map $\theta_K$ is differentiable at $x$ (Lemma~\ref{lem:chart}) and $p$ is differentiable at $\theta_K(x)$, so by the chain rule $p\circ\theta_K$ is differentiable at $x$ with
\[
D(p\circ\theta_K)(x)\,v=p'(\theta)\,D\theta_K(x)v=r(\theta)\,\bt(\theta)\Bigl(-\frac{v\cdot\bn(\theta)}{s}\Bigr),
\]
and \eqref{eq:DT} follows from $T=2\,p\circ\theta_K-\mathrm{Id}$. \\ 
Using the $C^1$ chart $\Phi$ defined in \eqref{eq:itemcchart}, the exceptional set is given by $$\Phi(\{(\theta,w):\theta\in E,\ w>h'(\theta)\}),$$ the image of a Lebesgue-null set under a $C^1$ map, hence Lebesgue-null.

\medskip
\noindent (b) If $\rho_h$ is purely singular then $r=0$ a.e., and \eqref{eq:DT} reduces to $DT(x)=-\mathrm{Id}$.
\end{proof}

\medskip

\subsection{Proof of Proposition \ref{prop:exact}}

\begin{proof}
Let $$y=F(x)=-p(\theta)+s\bt(\theta)=p(\theta+\pi)-s\bt(\theta+\pi).$$ 
Then $f_y(\theta+\pi)=0$ with $f_y'(\theta+\pi)=-s<0$, so $\theta+\pi=\check\theta_K(y)$, while $\theta^+=\theta_K(y)$. By Lemma~\ref{lem:two-tangents} the positive arc of $f_y$ runs from $\theta^+$ to $\theta+\pi$ and has length in $(0,\pi)$, {\it i.e.},  $(\theta+\pi)-\theta^+\in(0,\pi)\pmod{2\pi}$, which is the claim. \\
For the equations \eqref{eq:exact1} and \eqref{eq:exact2}, take the scalar product of
\begin{equation}\label{eq:vector}
p(\theta^+)+s^+\bt(\theta^+)=-p(\theta)+s\,\bt(\theta)
\end{equation}
with $\bn(\theta^+)$ and with $\bt(\theta^+)$; the first gives 
$$h(\theta^+)=-h(\theta)\cos d-h'(\theta)\sin d+s\sin d,$$
and the second 
$$h'(\theta^+)+s^+=h(\theta)\sin d-h'(\theta)\cos d+s\cos d,$$
where we used that 
$$
\bn(\theta)\cdot\bn(\theta^+)=\cos d, \quad \bt(\theta)\cdot\bn(\theta^+)=\sin d, \quad \bn(\theta)\cdot\bt(\theta^+)=-\sin d, \quad \bt(\theta)\cdot\bt(\theta^+)=\cos d.
$$
 Finally $T(x)=-y=\Psi(\theta^+-\pi,s^+)$ by the computation in Proposition~\ref{prop:T-basic} (b).\\
 
 \smallskip
Let us now prove items (i) and (ii).

\smallskip
\noindent (i) From \eqref{eq:vector}
$$p(\theta)+p(\theta^+)=s\bt(\theta)-s^+\bt(\theta^+);$$
 squaring gives $$s^2+(s^+)^2-2ss^+\cos d=(s-s^+)^2+2ss^+(1-\cos d).$$ The bound follows from $|p|\le M$.

\smallskip
\noindent (ii) Since 
$$h(\theta)\cos d+h'(\theta)\sin d=p(\theta)\cdot\bn(\theta^+) \qquad {\rm and} \qquad p(\theta)=-p(\theta+\pi),$$
 equation \eqref{eq:exact1} reads 
 $$s\sin d=h(\theta^+)-p(\theta+\pi)\cdot\bn(\theta^+)=(p(\theta^+)-p(\theta+\pi))\cdot\bn(\theta^+),$$ which is positive by \eqref{eq:unique-max} because $\theta+\pi\not\equiv\theta^+$ (since $d<\pi$). The function $(\alpha,\beta)\mapsto(p(\alpha)-p(\beta))\cdot\bn(\alpha)$ is continuous and positive on the compact set $\{\dist_{\Tt}(\alpha,\beta)\ge\varepsilon\}$, whence $c_\varepsilon$.

\end{proof}

\bigskip

\section{A singular Radon table with a four-periodic invariant curve}\label{app:Radon}

In this appendix we describe how to construct examples of tables that are $C^1$, centrally symmetric, strictly convex, with purely singular radius of curvature measure  and admit a central invariant curve consisting of periodic points of period $4$ (Proposition \ref{prop:Radon-invariant}). Moreover,  these tables can be chosen arbitrarily $C^1$-close to disks or ellipses (see Remark \ref{rem:closedisk}).

\medskip
\begin{remark}
This construction is inspired by the  construction of {\it Radon} curves in \cite{BMT16}. 
Let us briefly recall  the notion of Radon curves.

Let $\gamma$ be a $C^1$ closed convex curve in the plane, symmetric with respect to the origin. Let
$x,y \in \gamma.$ One says that $x$ is {\it Birkhoff orthogonal} to $y$ if $y$ is parallel to the tangent line
to $\gamma$ at $x$.  This relation is not necessarily symmetric; if it is symmetric, then $\gamma$ is called
a Radon curve. Radon curves form an infinite-dimensional family, with ellipses providing simple examples.

It was observed in \cite{BBT22} that a $C^1$ centrally symmetric strictly convex curve is a {Radon curve} if and only if its outer billiard carries an invariant curve of $4$-periodic orbits.

\end{remark}

\medskip

 Choose a positive, atomless, purely singular measure $\mu$
on $[0,\pi/2]$ with full support such that
\begin{equation}\label{eq:mu-norm}
\int_0^{\pi/2}\sin\theta\,\dd\mu(\theta)
=
\int_0^{\pi/2}\cos\theta\,\dd\mu(\theta)
=1.
\end{equation}
Such a measure exists: take a countable positive combination
of Cantor measures transported onto a basis of subintervals
of $[0,\pi/2]$, symmetrize it under
$\theta\mapsto\pi/2-\theta$, and rescale; see Figure \ref{fig:radon} (b). 

\medskip
Define for $\theta\in[0,\frac{\pi}{2}]$:
\begin{equation}\label{eq:PQ}
P(\theta):=(1,0)+\int_{[0,\theta)}\bt(u)\dd\mu(u),\qquad
c(\theta):=P(\theta)\cdot\bn(\theta),\qquad
Q(\theta):=\frac{\bt(\theta)}{c(\theta)}.
\end{equation}
Then $P(0)=(1,0)$ and, by \eqref{eq:mu-norm}, $P(\frac{\pi}{2})=(0,1)$. Observe that the first coordinate of $P$ decreases from $1$ to $0$ and the second increases from $0$ to $1$, so $P$ lies in the closed first quadrant and $P\ne0$.

\smallskip

Since $DP=\bt\,\mu$, the arc $P$ is an arc of a strictly convex curve with outward normal $\bn(\theta)$ at $P(\theta)$ and radius measure $\mu$ (Proposition~\ref{prop:reconstruction}); in particular $c$ is its support function, $c>0$ (both $P$ and $\bn$ lie in the first quadrant, and $c(0)=c(\frac{\pi}{2})=1$), and
\begin{equation}\label{eq:cQ}
c'=P\cdot\bt,\qquad P=c\,\bn+c'\,\bt,\qquad Q'=-\frac{P}{c^2},\qquad \det(P,Q)=\det(\bn,\bt)=1,
\end{equation}
where the first identity follows from $Dc=(\bt\cdot\bn)\mu+(P\cdot\bt)\dd\theta$, the second from the fact that $c$ is the support function of $P$,   the third from $\bt'=-\bn$ and the second, and the fourth is a direct computation.

\smallskip
Thus $Q$ is a regular $C^1$ arc, parametrized by $\theta$, from $Q(0)=(0,1)$ to $Q(\frac{\pi}{2})=(-1,0)$, with tangent direction $-P/|P|$ at $Q(\theta)$; and the tangent direction of $P$ at $P(\theta)$ is $\bt(\theta)$, a positive multiple of $Q(\theta)$.

\medskip

Observe that only $P$  is chosen freely, while $Q$ is determined from $P$ by \eqref{eq:PQ}.

\medskip

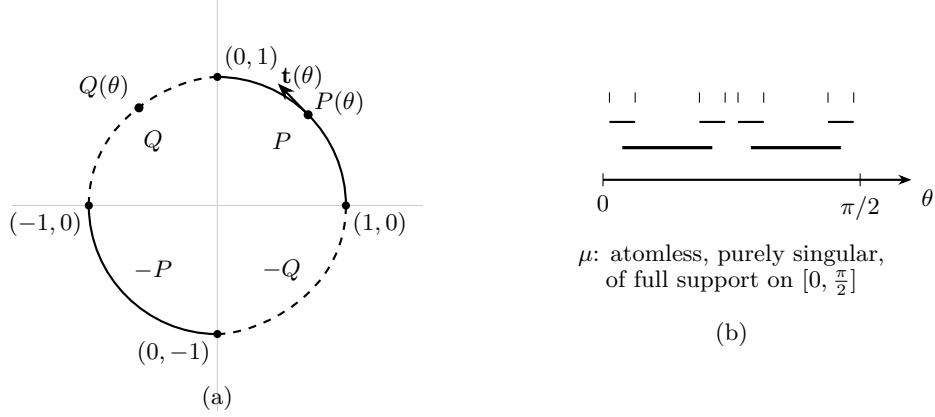
\begin{figure}[ht]
\centering
\begin{tikzpicture}[scale=1.7,>=Stealth]
\begin{scope}
\draw[gray!40] (-1.6,0) -- (1.6,0);
\draw[gray!40] (0,-1.6) -- (0,1.6);
\draw[thick] (1,0) .. controls (1,0.55) and (0.55,1) .. (0,1);
\draw[thick,dashed] (0,1) .. controls (-0.3,1.0) and (-1,0.7) .. (-1,0);
\draw[thick] (-1,0) .. controls (-1,-0.55) and (-0.55,-1) .. (0,-1);
\draw[thick,dashed] (0,-1) .. controls (0.3,-1.0) and (1,-0.7) .. (1,0);
\fill (1,0) circle (0.9pt) node[below right=-1pt,font=\small] {$(1,0)$};
\fill (0,1) circle (0.9pt) node[above right=-1pt,font=\small] {$(0,1)$};
\fill (-1,0) circle (0.9pt) node[below left=-1pt,font=\small] {$(-1,0)$};
\fill (0,-1) circle (0.9pt) node[below left=-1pt,font=\small] {$(0,-1)$};
\fill (0.706,0.706) circle (1pt);
\node[font=\small] at (0.95,0.8) {$P(\theta)$};
\draw[->,thick] (0.706,0.706) -- (0.47,0.95);
\node[font=\small] at (0.66,1.0) {$\bt(\theta)$};
\fill (-0.61,0.76) circle (1pt);
\node[font=\small] at (-0.88,0.92) {$Q(\theta)$};
\node[font=\small] at (0.5,0.5) {$P$};
\node[font=\small] at (-0.5,0.5) {$Q$};
\node[font=\small] at (-0.5,-0.5) {$-P$};
\node[font=\small] at (0.5,-0.5) {$-Q$};
\node[font=\small] at (0,-1.5) {(a)};
\end{scope}
\begin{scope}[xshift=3cm,yshift=0.2cm]
\draw[thick,-{Stealth[length=5pt]}] (0,0) -- (2.4,0) node[below right,font=\small] {$\theta$};
\draw (0,0.05) -- (0,-0.05) node[below,font=\small] {$0$};
\draw (2,0.05) -- (2,-0.05) node[below,font=\small] {$\pi/2$};
\foreach \x in {0.5,1.5} { \draw[very thick] (\x-0.35,0.25) -- (\x+0.35,0.25); }
\foreach \x in {0.15,0.85,1.15,1.85} { \draw[thick] (\x-0.1,0.45) -- (\x+0.1,0.45); }
\foreach \x in {0.05,0.25,0.75,0.95,1.05,1.25,1.75,1.95} { \draw (\x,0.6) -- (\x,0.68); }
\node[font=\small,align=center] at (1,-0.7) {$\mu$: atomless, purely singular,\\[-1pt] of full support on $[0,\frac{\pi}{2}]$};
\node[font=\small] at (1,-1.2) {(b)};
\end{scope}
\end{tikzpicture}
\caption{\small (a) A schematic representation of the table $K$  bounded by the four arcs, $P$ (solid), $Q$ (dashed), $-P$, $-Q$. 
(b) A schematic illustration of Cantor-type components.
The full-support measure $\mu$ is obtained by a countable
mixture of such components placed in a basis of subintervals
of $[0,\pi/2]$.}
\label{fig:radon}
\end{figure}
\medskip

Let $\beta(\theta):=\arg P(\theta)\in[0,\frac{\pi}{2}]$. Since $\det(P,DP)=\det(P,\bt)\mu=(P\cdot\bn)\mu$, then
\begin{equation}\label{eq:beta}
D\beta=\frac{c}{|P|^2}\,\mu ,
\end{equation}
so $\beta$ is an increasing homeomorphism of $[0,\frac{\pi}{2}]$ onto itself. 

\bigskip
Consider the closed curve obtained by concatenating $P$, $Q$, $-P$, $-Q$, as in Figure~\ref{fig:radon} (a).  Let us observe some  properties of the concatenated curve:
\begin{itemize}
\item At the four junctions the tangent directions match: at $(0,1)$ both $P$ and $Q$ have tangent direction $(-1,0)$, at $(-1,0)$ both $Q$ and $-P$ have direction $(0,-1)$, and so on.
\item  The tangent angle increases from $\pi/2$ to $\pi$ along $P$ (it equals $\theta+\pi/2$), from $\pi$ to $3\pi/2$ along $Q$ (it equals $\beta(\theta)+\pi$), and then continues through $2\pi$ along $-P$, $-Q$; it is strictly increasing on each arc, by full support of $\mu$ and by \eqref{eq:beta}.  A closed $C^1$ curve whose tangent angle is strictly increasing with total variation $2\pi$ is the boundary of a strictly convex body.
\end{itemize}

\medskip

Hence, the four arcs bound a centrally symmetric strictly convex body $K$ with regular $C^1$ boundary. Its support function $h$ equals $c$ on $[0,\frac{\pi}{2}]$ and is $\pi$-periodic.

\medskip

\begin{proposition} \label{prop:B2}
$K$ is a centrally symmetric strictly convex body with regular $C^1$ boundary. 
The radius of curvature measure $\rho_h$ of $K$ is atomless, purely singular and of full support: $h\in\cH$ and $h\in\cS$ (as defined in Lemma \ref{lem:S}).
\end{proposition}

\medskip

\begin{proof}
Central symmetry, strict convexity and regularity of the boundary were proven above.

Let us prove the second part.
On $P$ the radius measure is $\mu$, purely singular by construction. On $Q$ the arclength element is $$|Q'|\dd\theta=|P|c^{-2}\dd\theta$$ and the normal angle at $Q(\theta)$ is $\beta(\theta)+\pi/2$; hence the radius of curvature measure of the arc $Q$, {\it i.e.}\ its arclength expressed in the normal angle, is
\[
\rho_Q=(\beta+\pi/2)_{\#}\bigl(|P|\,c^{-2}\dd\theta\bigr),
\]
where $\#$ denotes the push-forward.\\
Let $E\subset[0,\frac{\pi}{2}]$ be a Lebesgue-null Borel set carrying $\mu$. For the increasing homeomorphism $\beta$ one has $|\beta(A)|=D\beta(A)$ for every Borel set $A$ (both sides are measures which agree on intervals). By \eqref{eq:beta} 
$$|\beta(E^c)|=D\beta(E^c)=\int_{E^c}c|P|^{-2}\dd\mu=0.$$ The absolutely continuous measure $|P|c^{-2}\dd\theta$ is supported on the Lebesgue-full set $E^c$, so its push-forward $\rho_Q$ is supported on the null set $\beta(E^c)+\pi/2$, hence it is purely singular. \\
It is atomless and of full support because $\beta$ is a homeomorphism and the density $|P|c^{-2}$ is positive. 
\\ 
Summarizing, the radius of curvature measure $\rho_h$ of $K$, which on $\Tp$ is $\mu$ on $[0,\frac{\pi}{2})$ and $\rho_Q$ on $[\frac{\pi}{2},\pi)$, is atomless, purely singular and of full support, therefore $h\in\cH$ and $h\in\cS$.
\end{proof}

\bigskip
Now, we want to prove the existence of an invariant curve consisting of periodic points of period $4$.

\medskip

Define $\mathcal J:\partial K\to\partial K$ by
\[
\mathcal J(P(\theta))=Q(\theta),\qquad \mathcal J(Q(\theta))=-P(\theta),\qquad \mathcal J(-x)=-\mathcal J(x).
\]
Then $\mathcal J^2=-\id$, $\det(x,\mathcal Jx)=1$ by \eqref{eq:cQ}, and $\mathcal Jx$ is a positive multiple of the positively oriented tangent vector at $x$: at $P(\theta)$ the tangent is $\bt(\theta)$ and $Q=\bt/c$, while at $Q(\theta)$ the tangent direction is $-P/|P|$. 

\smallskip
Writing $x=p(\varphi)$ with $\varphi$ the normal angle at $x$, these two facts give $\mathcal Jx=|\mathcal Jx|\,\bt(\varphi)$ with $h(\varphi)|\mathcal Jx|=\det(x,\mathcal Jx)=1$, that is
\begin{equation}\label{eq:Jone}
\mathcal J p(\varphi)=\frac{\bt(\varphi)}{h(\varphi)} .
\end{equation}
Define
\[
\Gamma:=\{x+\mathcal Jx:\ x\in\partial K\} = \Psi\bigl(\operatorname{graph}\tfrac1h\bigr)
\]
where the second equality follows from  \eqref{eq:Jone}. $\Gamma$ is therefore the graph of a $C^1$ height function $s=1/h$.
By Lemma~\ref{lem:chart}, the contact point of $x+\mathcal Jx$ is $x$.

\medskip

\begin{proposition}\label{prop:Radon-invariant}
$\Gamma$ is a central invariant  graph and 
\[
T^2\big|_\Gamma=-\id ,\qquad T^4\big|_\Gamma=\id.
\]
\end{proposition}

\medskip
\begin{proof}
Write $Z(x):=x+\mathcal Jx$ for $x\in\partial K$, so $\Gamma=Z(\partial K)$. Observe that $Z$ is injective, since $x$ is recovered from $Z(x)$ as its contact point. Since $\mathcal J(-y)=-\mathcal Jy$ for every $y$, then
\[
Z(-x)=-x+\mathcal J(-x)=-x-\mathcal Jx=-Z(x).
\]
Therefore $$\Gamma=Z(\partial K)=Z(-\partial K)=-Z(\partial K)=-\Gamma,$$
{\it i.e.,} $\Gamma$ is central.\\

Reflection through the contact point of $Z(x)$, which is $x$, gives
\[
T(Z(x))=T(x+\mathcal Jx)=x-\mathcal Jx .
\]
On the other hand, using $\mathcal J^2=-\id$ and $\mathcal J(-y)=-\mathcal Jy$ with $y=\mathcal Jx$, we get:
\begin{eqnarray*}
 Z(-\mathcal Jx)&=&-\mathcal Jx+\mathcal J(-\mathcal Jx)=-\mathcal Jx-\mathcal J^2x
=x-\mathcal Jx .
\end{eqnarray*}
Define $\varphi := -\mathcal J: \partial K \to \partial K$. Hence $T(Z(x))=Z(\varphi(x))$ for every $x\in\partial K$, {\it i.e.},  $T\circ Z=Z\circ\varphi$.

Observe that
$$\varphi^2(x)=-\mathcal J(-\mathcal Jx)=\mathcal J(\mathcal Jx)=\mathcal J^2x=-x,$$ so $\varphi^2=-\id$ on $\partial K$, and $\varphi^4=(-\id)^2=\id$. In particular $\varphi$ is a bijection of $\partial K$ onto itself, with inverse $\varphi^{-1}=\varphi^3$, since $\varphi^4=\id$.

\smallskip
From $T\circ Z=Z\circ\varphi$ we get
$$T^n\circ Z=Z\circ\varphi^n \qquad \forall\; n\ge 1.$$ 
In particular, using that $Z$ is a bijection of $\partial K$ onto $\Gamma$:
$$T^2(Z(x))=Z(\varphi^2(x))=Z(-x)=-Z(x) \quad \forall \;x \qquad \Longrightarrow \qquad T^2\big|_\Gamma=-\id.$$

Similarly:
$$T^4(Z(x))=Z(\varphi^4(x))=Z(x)\quad \forall \;x \qquad \Longrightarrow \qquad T^4\big|_\Gamma=\id.$$

\smallskip
 Finally
 $$T(\Gamma)=T(Z(\partial K))=Z(\varphi(\partial K))=Z(\partial K)=\Gamma,$$
  since $\varphi$ is a bijection of $\partial K$ onto itself. This proves that $\Gamma$ is invariant.
\end{proof}

\medskip

A few remarks are in order.

\begin{remark}\label{rem:regularityGamma}
Although $\Gamma$ is a $C^1$ graph in tangent coordinates, it is not a regular $C^1$ curve in the plane, and it is not convex. On its first quarter, parametrized by $\theta\mapsto P(\theta)+Q(\theta)$,
\begin{equation}\label{eq:RadonDGamma}
D(P+Q)=\bt\,\mu-\frac{P}{c^2}\dd\theta .
\end{equation}
The two vector measures are mutually singular and prescribe two different unit tangent directions, $\bt(\theta)$ on a set carrying $\mu$ and $-P(\theta)/|P(\theta)|$ on a set of full Lebesgue measure, both dense in every interval; they are never parallel, since $\det(\bt,-P)=\det(P,\bt)=c>0$, and in fact the angle between them is bounded below because $c/|P|$ is continuous and positive. For a rectifiable curve the unit tangent exists arclength-a.e.\ and coincides with the polar direction of the derivative measure, so a continuous unit tangent would have to agree with both directions on dense sets, which is not the case.

\smallskip
 Nor can $\Gamma$ be convex, since the tangent angle of a convex curve is monotone in arclength, hence has one-sided limits everywhere and at most countably many discontinuities, and at any continuity point the two dense directions in \eqref{eq:RadonDGamma} would have to coincide.
 
 \smallskip
 This example illustrates that an invariant graph can have a \(C^1\) height function in tangent coordinates while failing to be a regular \(C^1\) curve, or a convex curve, in the plane.
 \end{remark}

\medskip
\begin{remark}\label{rem:closedisk}
The above construction can be adapted to obtain curves that are arbitrarily $C^1$-close, for example, to the unit disk (and therefore to any disk or ellipse, in the light of Remark \ref{rem:affinvariance}). We outline how to modify it  below.

Fix an atomless, purely singular probability measure $\nu$ on
$[0,1]$ with full support, symmetric under the reflection
$u\mapsto1-u$. Such a measure can be obtained by symmetrizing
an atomless, purely singular, full-support probability measure
on $[0,1]$.

For each $m\ge1$, put $\delta_m:=\frac{\pi}{2m}$ and divide $[0,\frac{\pi}{2}]$ into
the cells
\[
I_{k,m}:=[k\delta_m,(k+1)\delta_m],
\qquad 0\le k\le m-1.
\]
Let
\[
\alpha_{k,m}(u):=k\delta_m+\delta_m u,
\qquad u\in[0,1],
\]
and define
\[
\tilde\mu_m
:=
\delta_m\sum_{k=0}^{m-1}(\alpha_{k,m})_\#\nu.
\]
Thus every cell carries an affine copy of the same measure
$\nu$, multiplied by $\delta_m$. The measure $\tilde\mu_m$
is atomless, purely singular, and has full support on $[0,\frac{\pi}{2}]$.

Let $R(\theta):=\frac{\pi}{2}-\theta$ and $r(u):=1-u$. Since
\[
R\circ\alpha_{k,m}
=
\alpha_{m-1-k,m}\circ r
\qquad\text{and}\qquad
r_\#\nu=\nu,
\]
we have $R_\#\tilde\mu_m=\tilde\mu_m$.
Hence $\tilde\mu_m$ is symmetric under $\theta\mapsto \frac{\pi}{2}-\theta$.
Moreover, $\tilde\mu_m$ assigns every cell the same total mass
as Lebesgue measure:
\[
\tilde\mu_m(I_{k,m})=\delta_m=|I_{k,m}|.
\]
The common endpoints have zero mass. Since an interval
$J\subset[0,\frac{\pi}{2}]$ has at most two partially covered cells, it follows
that
\[
\bigl|\tilde\mu_m(J)-|J|\bigr|\le2\delta_m.
\]

By symmetry, the sine and cosine moments coincide.
Writing
\[
\lambda_m
:=
\int_0^{\frac{\pi}{2}}\sin\theta\,\dd\tilde\mu_m(\theta)
=
\int_0^{\frac{\pi}{2}}\cos\theta\,\dd\tilde\mu_m(\theta),
\]
the preceding discrepancy estimate gives
\[
\lambda_m\longrightarrow
\int_0^{\frac{\pi}{2}}\sin\theta\,\dd\theta=1.
\]
Therefore
\[
\mu_m:=\lambda_m^{-1}\tilde\mu_m
\]
satisfies \eqref{eq:mu-norm}, with normalization factor tending
to $1$.

\smallskip
Observe that for the Lebesgue measure the construction \eqref{eq:PQ} returns the unit circle: 
$$P=\bn, \quad c=1, \quad Q=\bt, \quad \beta=\theta.$$ 

Integration by parts as in the proof of Lemma~\ref{lem:S} gives $P_m\to\bn$ uniformly, hence $c_m=P_m\cdot\bn\to1$ and $c_m'=P_m\cdot\bt\to0$ uniformly, $Q_m\to\bt$ uniformly, and $\beta_m\to\theta$ uniformly together with the inverses (uniform limits of increasing homeomorphisms of $[0,\frac{\pi}{2}]$ converging to the identity). 

On the dual arc the support function and its derivative at the normal angle $\varphi=\beta_m(\theta)+\pi/2$ are $Q_m(\theta)\cdot\bn(\varphi)$ and $Q_m(\theta)\cdot\bt(\varphi)$; as functions of $\varphi$ they converge uniformly, respectively, to 
$$\bt(\varphi-\pi/2)\cdot\bn(\varphi)=1 \qquad {\rm and} \qquad \bt(\varphi-\pi/2)\cdot\bt(\varphi)=0.$$ 

Hence the support functions $h_m$ of the resulting tables satisfy $\|h_m-1\|_{C^1(\Tp)}\to0$.
\end{remark}

\end{document}